\documentclass[twoside,reqno]{amsart}
\usepackage{amsfonts,amsmath,amscd,amsthm,amssymb}
\usepackage{mathtools}
\usepackage[mathscr]{euscript} 
\usepackage{graphics}
\usepackage{url}
\usepackage{wrapfig}
\usepackage{lscape}
\usepackage{rotating}
\usepackage{epsfig}
\usepackage{cite}
\usepackage{color}
\usepackage{booktabs}
\usepackage{multirow}
\usepackage{lipsum,multicol}

\usepackage{tikz}
\usetikzlibrary{patterns}

\usepackage{amssymb}
\usepackage{bbm}
\usepackage{multirow}
\newtheorem{theorem}{Theorem}[section]
\newtheorem{lemma}[theorem]{Lemma}

\newtheorem{corollary}[theorem]{Corollary}

\theoremstyle{definition}

\theoremstyle{remark}
\newtheorem{remark}[theorem]{Remark}

\numberwithin{equation}{section}
\numberwithin{table}{section}

\def\ho{h^{k+1}} % h^{k+1}

\def\jj{j = 1, \ldots, N}

\def\ih{\mathcal{I}_h}
\def\jump#1{[\![{#1}]\!]} % the jump of a function
\def\mean#1{\{\!\!\{{#1}\}\!\!\}} % the mean of a function

\def\normI#1{\|{#1}\|_{L^2(I)}}  % the usual L^2 norm defined on the domain \I
\def\oh{\Omega_h}    % the partition \Omega_h
\def\poh{\partial \Omega_h}  % all the boundaries of the partition
\def\pk{\partial K}                  % the boundaries of an element K
\def\intc#1{\int_{I_j}{#1}\mathrm{dx} } % the integral of a function at each cell
\def\intid#1{\int_{I}{#1}\mathrm{dx} } % the integral of a function at whole interval or domain
\def\intk#1{\int_K{#1} \mathrm{d{\boldsymbol x}}} % the integral of a function at an arbitrary element
\def\intdk#1{\int_{\partial K}{#1} \mathrm{ds}} % the integral in the y variable

\begin{document}

\title[Optimal energy conserving  DG]{Optimal energy-conserving discontinuous Galerkin methods for 
linear symmetric hyperbolic systems}

\author{Guosheng Fu}
\address{Division of
Applied Mathematics, Brown University, Providence, RI
 02912}
\curraddr{}
\email{guosheng\_fu@brown.edu}

%    author two information
\author{Chi-Wang Shu}
\address{Division of
Applied Mathematics, Brown University, Providence, RI
 02912}
\curraddr{}
\email{shu@dam.brown.edu}
\thanks{The research of the second author was supported
  by ARO grant W911NF-15-1-0226 and NSF grant DMS-1719410.}

%    \subjclass is required.

%\date{August 22, 2013 and, in revised form, October 30, 2014 and November XX, 2014.}

\dedicatory{}


%    Abstract is required.
\begin{abstract}
We propose energy-conserving discontinuous Galerkin (DG) methods
for symmetric linear hyperbolic systems on general unstructured meshes.
Optimal a priori error estimates of order $k+1$ are obtained for the semi-discrete scheme 
in one dimension, and in multi-dimensions on Cartesian meshes when tensor-product polynomials of
degree $k$ are used.
A high-order energy-conserving  Lax-Wendroff time discretization is also presented.

Extensive numerical results in one dimension, and two dimensions on both rectangular and 
triangular meshes are presented to support the theoretical
findings and to assess the new methods.
One particular method (with the doubling of unknowns) 
is found to be optimally convergent on triangular meshes for all the examples considered in this paper.
The method is also compared with the classical (dissipative) upwinding  
DG method and (conservative) DG method with a central flux.
It is numerically observed for the new method to have a superior performance for long-time 
simulations.
% Optimal a priori error estimates of order $k+1$ can be proven in multi-dimension on 
% tensor-product meshes.
% 
% including the equations for
% advection, acoustics, aerodynamics, elastodynamics, and electromagnetism. 
% The method is defined on Cartesian meshes that use piecewise tensor product polynomials 
% as basis functions. 
% Numerical experiments are shown to demonstrate the theoretical results.
\end{abstract}

\maketitle
\section{Introduction}
Wave propagation problems arise in science, engineering and 
industry, and they are significant to geoscience, 
petroleum engineering, telecommunication, 
and the defense industry (see \cite{Durran99,Kampanis08} and the references therein).
Efficient and accurate numerical methods to solve wave propagation problems are of 
fundamental importance to these applications. Experience reveals 
that energy conserving numerical methods, which conserve the discrete approximation of 
energy, are favorable because they are able to maintain the phase and shape of the waves 
accurately, especially for long time simulation. 
% In this paper, we present an energy-conserving discontinuous Galerkin (DG)
% method for linear symmetric hyperbolic systems, including the equations for advection, acoustics, 
% electromagnetism, elastodynamics, and aeroacoustics.

A vast amount of literature can be found on the numerical approximation of 
wave problems modeled by linear hyperbolic systems. 
All types of numerical methods, including finite difference, finite element, finite
volume and spectral methods have their proponents.  Here, we will confine our attention
in finite element methods, in particular, discontinuous Galerkin (DG) methods.
The DG methods, c.f. \cite{Cockburn00}, belong to a class of finite element methods 
using discontinuous piecewise polynomial
spaces for both the numerical solution and the test functions. 
They allow arbitrarily unstructured meshes, and have compact stencils. Moreover, they
easily accommodate arbitrary h-p adaptivity. 

Various DG methods can be applied to solve linear hyperbolic systems. 
We mention the classical Runge-Kutta DG method of Cockburn and Shu \cite{CockburnShu01}, 
the nodal DG method of Hesthaven and Warburton \cite{HesthavenWarburton02},
the space-time DG method of Falk and Richter \cite{FalkRichter00} and 
Monk and Richter \cite{MonkRichter05}. 
All these DG methods use approximate/exact Riemann solvers to define the numerical flux, and 
are dissipative by design.

A suboptimal energy-conserving DG method using central fluxes,  
has been presented by Fezoui et. al. in \cite{Fezoui05}
for the Maxwell’s equations. 
Chung and Engquist \cite{ChungEngquist09} have proposed an optimal,
energy conserving DG method for the acoustic wave equation on staggered grids.
More recently, Xing et. al. \cite{XingChouShu14} proposed 
an optimal, energy conserving DG method using alternating fluxes
for the acoustic wave equation on Cartesian grids. These DG methods do not rely on 
approximate/exact Riemann solvers to define the numerical flux.

Our work can be considered as a continuation of \cite{XingChouShu14}
on the search for optimal, energy-conserving 
DG methods for general linear symmetric hyperbolic systems.
We propose an energy-conserving DG method for linear symmetric hyperbolic systems on general 
unstructured meshes.
The method on Cartesian meshes 
is identical to the DG method using alternating fluxes \cite{XingChouShu14} 
for the acoustic wave equation considered therein. They may be different on general triangular meshes.
We prove optimal convergence of the proposed semi-discrete DG method in one-space dimension.
In particular, we present, for the first time, an optimal, energy-conserving DG method for the 
scalar advection equation on general non-uniform meshes in one dimension.
Similar to \cite{XingChouShu14}, the semi-discrete DG method can be also proven to be 
optimally convergent in multi-dimensions on Cartesian meshes, essentially using the superconvergence result 
of Lesaint and Raviart \cite{LasaintRaviart74} for the tensor-product 
Gauss-Radau projection.
% The error estimates of the proposed DG method on general triangular meshes is
On the other hand, on general triangular meshes, 
we are only able to 
prove a suboptimal convergence for the proposed method 
using a standard $L^2$-projection type analysis.
However, in all our numerical results on unstructured triangular meshes presented in this paper,
including the scalar advection equation, the acoustic equations, and the equations for elastodynamics,  
the method is observed to be optimally convergent.
% ample numerical results are presented on unstructured triangular meshes to 
% show the optimal convergence of the proposed method.
% on triangular meshes, while 
% in certain other cases (acoustics with zero background velocity and elastodynamics), 
% we only observe suboptimal convergence. Such suboptimal convergence on triangular meshes 
% was also recently documented in 
% \cite{LiShiShu17} for the Maxwell's equations. 
A theoretical study of the convergence property of the
proposed method on triangular meshes consists of our ongoing work.

% We also present a high-order energy-conserving Lax-Wendroff \cite{LaxWendroff60} time discretization
% for the resulting semi-discrete scheme.

The rest of the paper is organized as follows. 
In Section \ref{sec:1d}, we first present and analyze 
the semi-discrete energy-conserving DG method for linear symmetric hyperbolic systems 
in one dimension. We also present the high-order energy-conserving 
Lax-Wendroff time discretization. 
In Section \ref{sec:2d}, the  method is extended to multi-dimensions.
Numerical results are reported in Section \ref{sec:num}. Finally, we conclude in 
Section \ref{sec:conclude}.

\section{Energy-conserving DG methods for the one-dimensional case}\label{sec:1d}
In this section, we present and analyze the  energy-conserving DG methods
for linear symmetric hyperbolic systems in one dimension. The extension to 
the multidimensional case will be consider in the next section.

\subsection{Notation and definitions in the one-dimensional case}
In this subsection, we shall first introduce some notation and definitions in the one-dimensional case,
which will be used throughout this section.
% which will be used in our analysis for one-dimensional linear conservation laws.
\subsubsection{The meshes}
Let us denote by $\ih$ a tessellation of the computational interval $I = [0, 1]$,
consisting of cells $I_j = (x_{j-\frac12},x_{j+\frac12})$ with $1
\le j \le N$, where
$$
0 = x_\frac12 < x_\frac32 < \cdots < x_{N+\frac12} = 1.
$$
The following standard notation of DG methods will be used.
Denote $x_j =
(x_{j-\frac12} + x_{j+\frac12})/2$, $h_j = x_{j+\frac12} -
x_{j-\frac12}$, $h = \max_jh_j$, and $\rho = \min_jh_j$.
The mesh is assumed to be regular in the sense that $h / \rho$ 
is always bounded during mesh refinements, namely, there
exists a positive constant $\gamma$ such that $\gamma h \le \rho \le h$.
We denote by $p_{j+\frac12}^-$ and $p_{j+\frac12}^+$ the values
of $p$ at the discontinuity point $x_{j+\frac12}$, from the left
cell, $I_j$, and from the right cell, $I_{j+1}$, respectively.
In what follows, we employ $\jump p = p^+ - p^-$ and $\mean p = \frac12(p^+ +
p^-)$ to represent the jump and the mean value of $p$ at each element
boundary point. The following discontinuous piecewise polynomials space is chosen
as the finite element space:
\begin{align}
\label{space-1d}
V_h \equiv V_h^{k} = \left\{v \in L^2(I): v|_{I_j} \in P^{k}(I_j),
~~ j = 1, \ldots, N \right\}, 
\end{align}
where $P^{k}(I_j)$ denotes the set of polynomials of degree up to $k
\ge 0$ defined on the cell $I_j$.

\subsubsection{Function spaces and norms}
Denote $H^1(I)$ as the space of $L^2$ functions on $I$ whose derivative is also an $L^2$ function.
Denote $\|\cdot\|_{I_j}$ the standard $L^2$-norm on the cell $I_j$, and 
$\|\cdot\|_I$ the $L^2$-norm on the whole interval.
% For any integer $ l \ge 0$, denote by $ \| \cdot \|_{H^{l}(I_j)}$ and $ \| \cdot \|_{W^{l,\infty}(I_j)}$ the
% standard Sobolev norms on the cell $I_j$.
% Then, the norms of the broken Sobolev spaces
% $W^{l,p}(\ih) := \{ u \in L^2 (I) : u|_{I_j} \in W^{l,p}(I_j), ~\forall \jj \}$
% with $p = 2, \infty$ are given by
% $$
%  \| u \|_{W^{l,2}(\ih)}  ={\norm u}_{H^l(\ih)} = \left( \summ {\norm u}_{H^l(I_j)}^2 \right)^{\frac12}, \quad
% \| u \|_{W^{l,\infty}(\ih)} = \max_{1\le j \le N} \| u \|_{W^{l,\infty}(I_j)}.
% $$
% In the case $l = 0$, we denote $\normI u = \norm u_{H^0(\ih)}$.

\newcommand{\bld}[1]{\boldsymbol{#1}}
\subsection{Energy-conserving DG methods for linear symmetric hyperbolic systems}
\label{sub:sys}
We first start with a general form of energy-conserving DG method for the following linear symmetric 
hyperbolic system:
%  \begin{subequations}
\begin{align}
 \label{1d-symmetric-sys}
 \bld B_0\,{\bld u}_t + \bld B_1 \bld u_x & =0, &&\hspace{-2.8cm}  (x,t)\in I\times (0,T],
\end{align}
with initial condition $\bld u(x,0)=\bld u_0(x)$, and periodic boundary condition.
Here the unknown is $\bld u:I\times (0,T]\rightarrow\mathbb{R}^m$, and $\bld u=(u^1,\cdots, u^m)$.
The matrix $\bld B_0: I \rightarrow\mathbb{R}^{m\times m}$ is a diagonal matrix with positive, piecewise constant diagonal entries, and 
$\bld B_1\in\mathbb{R}^{m\times m}$ is a {\it symmetric constant coefficient} matrix.
% \end{subequations}

The semi-discrete DG method for \eqref{1d-symmetric-sys} reads as follows. 
Find, for any time $t \in (0, T]$,
the unique function $\bld u_h=\bld u_h(t)\in [V_h^k]^m$ such that
\begin{subequations}
 \label{scheme:sys1d}
\begin{align} 
% \small
\intc{\bld B_0(\bld u_h)_t\cdot \bld v_h} - \intc{\bld B_1\,\bld u_h\cdot (\bld v_h)_x} +
\widehat{\bld B_1\bld u}_h\cdot \bld v_h^-|_{j+\frac12} -
\widehat{\bld B_1\bld u}_h \cdot\bld v_h^+|_{j-\frac12} = &\;0,
\end{align}
holds for all $\bld v_h\in [V_h^k]^m$ and all $\jj.$
The {\it consistent} 
numerical fluxes $\widehat{\bld B_1\bld u}_h$ is single-valued on the cell boundaries $x_{j-1/2}$, 
and it is given by the following form 
\begin{align}\label{flux:sys1d}
\widehat{\bld B_1\bld u}_h|_{j-\frac12} = 
\bld B_1\mean{\bld u_h}
+ \bld R_{j-\frac12}\jump{\bld u_h},
\end{align}
where $\bld R_{j-\frac12}\in \mathbb{R}^{m\times m}$ is a, yet to be determined, {\it stabilization} matrix 
at $x_{j-\frac12}$. 
\end{subequations}
% We assume $\bld R_{j-1/2}$ do not depend on the solution $\bld u_h$.
\begin{remark}[The stabilization matrix $\bld R_{j-\frac12}$]
Since the matrices $\bld R_{j-\frac12}$ do not depend on the numerical solution $\bld u_h$, 
the semi-discrete DG scheme \eqref{scheme:sys1d} with the numerical flux \eqref{flux:sys1d} is a
linear scheme. 
In the most general form of a {\it local} numerical flux, the stabilization 
 $\bld R_{j-\frac12}$ may depend on $\bld u_h^\pm$ at the interface $x_{j-\frac12}$, which will leads to 
 a nonlinear scheme for the linear equation \eqref{1d-symmetric-sys}.
 We always consider the linear numerical flux \eqref{flux:sys1d} in this work, as we do not see any advantage of a 
 nonlinear scheme for the equation \eqref{1d-symmetric-sys}.
%  However, we remark that {} energy conserving 
\end{remark}

Summing the equations \eqref{scheme:sys1d} for all $j$, and using the periodic boundary condition, we have
\begin{align} 
% \small
 \label{scheme:sys1d-w}
\sum_{j=1}^N\left(\intc{\bld B_0(\bld u_h)_t\cdot \bld v_h} - \intc{\bld B_1\,\bld u_h\cdot (\bld v_h)_x} -
\widehat{\bld B_1\bld u}_h\cdot \jump{\bld v_h}|_{j-\frac12}\right) = &\;0,
\end{align}

As is well-known, the linear symmetric hyperbolic system \eqref{1d-symmetric-sys} admits an important conserved quantity
-- the energy, 
\[
 E(t)=\int_I (\bld B_0 \bld u(t))\cdot\bld u(t)\mathrm{dx},
%  =E(0),\quad \forall t.
\]
that is, $E(t)=E(0)$ for all $t>0$.
% We are particularly interested in DG methods that conserve the discrete analogs of energy.
Experiences show that schemes conserving  the discrete analogs of energy
often produce approximations that behave better for long time simulation.
We are particularly interested in deriving optimally-convergent, and energy-conserving DG methods.
We call the semi-discrete DG method an energy-conserving DG method if the discrete energy
\begin{align}
\label{energy:sys1d}
 E_h(t):=\int_I (\bld B_0 \bld u_h\cdot\bld u_h)\mathrm{dx}
\end{align}
is conserved for all time.

The following theorem provide a sufficient and necessary condition for energy conservation of 
the DG methods \eqref{scheme:sys1d}.
\begin{theorem}
\label{thm:sys1d}
 The (continuous-in-time) energy $E_h(t)$ \eqref{energy:sys1d} is conserved by the semi-discrete DG scheme
 \eqref{scheme:sys1d}
 for any initial condition $\bld u_0(x)$, for all time $t>0$ if and only if the stabilization matrix 
 $\bld{R}_{j-1/2}$ is anti-symmetric for all $j$.
\end{theorem}
\begin{proof}
 Taking $\bld v_h = \bld u_h$ in the scheme \eqref{scheme:sys1d-w}, and using the definition of the numerical flux 
 \eqref{flux:sys1d}, we have 
 \begin{align}
 \label{temp1}
\int_I{\bld B_0(\bld u_h)_t\cdot \bld u_h}\mathrm{dx} 
=&\;\sum_{j=1}^N \intc{\bld B_1\,\bld u_h\cdot (\bld u_h)_x}\\
&\;+\sum_{j=1}^N (\bld B_1\mean{\bld u_h}
+ \bld R_{j-\frac12}\jump{\bld u_h})\cdot \jump{\bld u_h}|_{j-\frac12}.\nonumber
 \end{align}
 By symmetry of $\bld B_1$, we have 
 $\bld B_1\,\bld u_h\cdot (\bld u_h)_x=\frac12(\bld B_1\,\bld u_h\cdot \bld u_h)_x$. Applying 
 integration by parts
 of each of the above integral on the right hand side, and using the periodic boundary condition, 
 the right side of \eqref{temp1} can be simplifies as
 \begin{align*}
\sum_{j=1}^N\left. \left(\bld B_1\mean{\bld u_h}
+ \bld R_{j-\frac12}\jump{\bld u_h}\right)\cdot \jump{\bld u_h}
-(\frac12\bld B_1\,\bld u_h^+\cdot \bld u_h^+-
\frac12\bld B_1\,\bld u_h^-\cdot \bld u_h^-)
\right|_{j-\frac12}.
 \end{align*}
 A simple calculation yields 
\[ 
\bld B_1\mean{\bld u_h}\cdot\jump{\bld u_h}
-(\frac12\bld B_1\,\bld u_h^+\cdot \bld u_h^+-
\frac12\bld B_1\,\bld u_h^-\cdot \bld u_h^-) = 0.
\]
Combining this with \eqref{temp1} and \eqref{energy:sys1d}, we have 
\begin{align}
\frac{d}{dt} \frac12E_h(t)
 =\sum_{j=1}^N\left.\bld R_{j-\frac12}\jump{\bld u_h}\cdot \jump{\bld u_h}
\right|_{j-\frac12}.
\end{align}
Requiring $\frac{d}{dt}E_h(t) = 0$ for all time, for 
any initial condition $\bld u_h(0)$ simply implies that 
\[
\bld R_{j-\frac12}\bld v\cdot \bld v=0 \quad\forall \bld v\in \mathbb{R}^m,\quad \forall j.
\]
Hence, $\bld R_{j-\frac12}$ must be an anti-symmetric matrix for all $j$.
\end{proof}

\begin{remark}[Scalar case, $m=1$]
Theorem \ref{thm:sys1d} implies that,
in the case $m=1$,  there exists only {\it one} energy-conserving DG method of the form 
\eqref{scheme:sys1d}, where
 $\bld R_{j-\frac12}\equiv 0$ for all $j$, and the resulting numerical flux is nothing but 
the central flux.

It is well-known that DG methods with a central flux provide suboptimal $L^2$-convergence order of $k$ when 
polynomials of degree $k$ is used, 
with the exception that optimal convergence order of $k+1$ can be proven under the 
stringent assumption that mesh is uniform and polynomial degree $k$ is even, c.f. \cite{CockburnShu98}.
Violating either of these assumptions results in suboptimal convergence.
% We show numerical results in section \ref{sec:num1d} to convince the reader that 
% violating either of these assumptions will results a suboptimal convergent method.
% In particularly, DG methods with a central flux using polynomials of even degree $k$ on 
% nonuniform meshes convergences only with order $k$.
While there seems no hope to obtain optimal-convergent, energy-conserving DG methods 
for the scalar advection equation on nonuniform mesh,
we show in the next subsection, that by simply doubling the number of unknowns, we 
can obtain an optimal-convergent, energy-conserving DG method on general 
nonuniform meshes.
\end{remark}

\subsection{Optimal  energy-conserving DG method for advection}
\label{sub:ad}
In this subsection, we consider the following advection equation
\begin{align}
\label{advection1d}
 u_t + c\, u_x & =0, &&\hspace{-2.8cm}  (x,t)\in I\times (0,T] ,
\end{align}
with a smooth periodic initial condition $u(x,0) = u_0(x)$ for $x\in I$.
Again, we assume periodic boundary condition for simplicity.
Here we assume the speed $c$ is a piecewise positive constant on the mesh. 
Note that the equation \eqref{advection1d} can be recast into the general form \eqref{1d-symmetric-sys} with
$\bld B_0 = c^{-1}$, $\bld B_1 = 1$.

To derive the energy-conserving DG method for the advection equation \eqref{advection1d}, we shall first 
double the unknowns
by introducing an auxiliary {\it zero} function $\phi(x,t) = 0$, which shall be thought of as 
the solution of an advection equation using the opposite speed as that for $u(x,t)$, but with {\it zero} 
initial data.
Then, we get the following $2\times 2$ system:
 \begin{subequations}
 \label{aux-adv}
\begin{align}
 u_t + c\, u_x & =0, &&\hspace{-2.8cm}  (x,t)\in I\times (0,T] ,\\
 \phi_t - c\, \phi_x & =0, &&\hspace{-2.8cm}  (x,t)\in I\times (0,T] ,
\end{align}
with initial condition $u(x,0)=u_0(x)$ and $\phi(x,0)=0$.
\end{subequations}
Note that this system 
can be recast into the general form \eqref{1d-symmetric-sys} with
$\bld B_0 = \mathrm{diag}([c^{-1},c^{-1}])$, 
$\bld B_1 = \left[\begin{tabular}{cc}
             $1$ &$0$\\
             $0$ & $-1$
            \end{tabular}\right]$.
% However, we directly introduction the DG method for the equations \eqref{aux-adv}, 
% which is equivalent to \eqref{scheme:sys1d} with the above choices of $\bld B_0$ and $\bld B_1$.

The semi-discrete DG method for \eqref{aux-adv} is as follows. 
Find, for any time $t \in (0, T]$,
the unique function $(u_h, \phi_h) = (u_h(t), \phi_h(t))
\in V_h^k\times V_h^k$ such that
 \begin{subequations}
 \label{scheme:adv1d}
\begin{align} \label{scheme:adv1d-1}
\intc{(u_h)_tv_h} - \intc{c\,u_h (v_h)_x} +
c^{-}\widehat{u}_h v_h^-|_{j+\frac12} -
c^+\widehat{u}_h v_h^+|_{j-\frac12} = &\;0,\\
\label{scheme:adv1d-2}
\intc{(\phi_h)_t\psi_h} + \intc{c\,\phi_h (\psi_h)_x} -
c^-\widehat{\phi}_h \psi_h^-|_{j+\frac12} +
c^+\widehat{\phi}_h \psi_h^+|_{j-\frac12} =&\; 0,
\end{align}
 \end{subequations}
holds for all $(v_h, \psi_h)\in V_h^k\times V_h^k$ and all $\jj.$
Applying Theorem \ref{thm:sys1d}, any energy energy conserving 
numerical fluxes $\widehat{u}_h$ and $\widehat{\phi}_h$ have the following form 
\begin{subequations}
\label{flux:2x2}
\begin{align}
 \widehat{u}_h|_{j-\frac12} =  &\; \mean {u_h} +\alpha_{j-\frac12}\jump {\phi_h},\\ 
 \widehat{\phi}_h|_{j-\frac12} = &\; \mean {\phi_h} +\alpha_{j-\frac12}\jump {u_h}.
\end{align}
\end{subequations}
with $\alpha_{j-\frac12}$ being any real constant.

We collect this result in the following Corollary.
\begin{corollary}\label{coro:adv1d}
The  energy 
\[
 E_h(t) = \int_I(c^{-1}(u_h)^2+c^{-1}(\phi_h^2))\mathrm{dx}
\]
is conserved by the semi-discrete scheme \eqref{scheme:adv1d} 
with the numerical flux \eqref{flux:2x2}
for all time.
\end{corollary}
\begin{remark}[Modified energy]
We specifically remark here that it is the total energy 
\[
  E_h(t) = \int_I(c^{-1}(u_h)^2+c^{-1}(\phi_h^2))\mathrm{dx}
\]
that is conserved, not the quantity $\int_I(c^{-1}(u_h)^2)\mathrm{dx}$.
The quantity $\phi_h$ is an approximation to the {\it zero} function, in general it will not be 
{\it zero} as long as $\alpha_{j-\frac12}\not=0$, due to the {\it coupling} in the numerical flux 
\eqref{flux:2x2}.
\end{remark}

Now, we turn to the error estimates of the scheme \eqref{scheme:adv1d-1}.
Clearly, taking $\alpha_{j-\frac12}=0$ decouples the two equations \eqref{scheme:adv1d-1} and 
\eqref{scheme:adv1d-2}, and we obtain the suboptimal DG method with central flux.
In the next result, we show that simply taking $\alpha_{j-\frac12}=\frac12$ for all $j$ results an optimal
convergence DG method with a clean proof. The resulting numerical fluxes are 
\begin{subequations}
\label{flux:opt}
\begin{align}
 \widehat{u}_h|_{j-\frac12} =  &\; \mean {u_h} +\frac12\jump {\phi_h},\\ 
 \widehat{\phi}_h|_{j-\frac12} = &\; \mean {\phi_h} +\frac12\jump {u_h}.
\end{align}
\end{subequations}
% In the rest of this subsection, we always assume $\alpha_{j-\frac12}=\frac12$ for all $j$.

We start by introducing a set of projections.
We shall use the following left and right Gauss-Radau projections $P_h^{\pm}$.
\begin{subequations}\label{proj1d}
\begin{align}
\intc{P_h^\pm u(x) v_h}  &= \intc{u(x) v_h} && \hspace{-2cm} \forall v_h \in
P^{k-1}(I_j), \\
 (P_h^\pm u)^\pm  &= u^\pm &&\hspace{-2cm} \text{at}~~ x_{j\mp\frac12},
\end{align}
the following approximation properties of $P_h^{\pm}$ is well-known
\begin{align}
\|P_h^\pm u - u\|_{I_j}\le Ch^{k+1}.
\end{align}
\end{subequations}

We shall also use the following coupled projection specifically designed for the DG scheme \eqref{scheme:adv1d}.
For any function $u, \phi\in H^1(I)$, 
we introduce the following coupled auxiliary projection 
 $(P_h^{1,\star} u, P_h^{2,\star} \phi)\in [V_h^k]^2$:
 \begin{subequations}
 \label{proj-adv}
 \begin{align}
 \label{proj-adv1}
  \intc{P_h^{1,\star} u(x) v_h}  &= \intc{u(x) v_h} &&  \forall v_h \in P^{k-1}(I_j),\\
 \label{proj-adv2}
  \intc{P_h^{2,\star} \phi(x) v_h}  &= \intc{\phi(x) v_h}  &&  \forall v_h \in P^{k-1}(I_j),\\
 \label{proj-adv3}
  (\mean {P_h^{1,\star} u_h} +{\frac12}\jump {P_h^{2,\star} \phi_h})\Big|_{j-\frac12} &=  u(x_{j-\frac12}),\\
 \label{proj-adv4}
 ( \mean {P_h^{2,\star} \phi_h} +{\frac12}\jump {P_h^{1,\star} u_h})\Big|_{j-\frac12} &=  \phi(x_{j-\frac12}),
 \end{align}
 \end{subequations}
 for all $j$.
 
 At a first glance, the projection \eqref{proj-adv} seems to be globally coupled. 
 The following Lemma shows that it is actually an optimal local projection.
 \begin{lemma}
 \label{lemma:proj}
  The projection \eqref{proj-adv} is well-defined, and it satisfies
  \begin{subequations}
  \label{transformX}
  \begin{align}
   P_h^{1,\star} u = \frac12(P_h^+(u+\phi)+P_h^-(u-\phi)),\\
   P_h^{2,\star} \phi = \frac12(P_h^+(u+\phi)-P_h^-(u-\phi)).
  \end{align}
  In particular, it satisfies
  \begin{align}
  \label{approx-proj}
  \| P_h^{1,\star} u -u\|_{I_j}\le Ch^{k+1}, \text{ and }
  \| P_h^{2,\star} \phi -\phi\|_{I_j}\le Ch^{k+1}.
  \end{align}
    \end{subequations}
 \end{lemma}
 \begin{proof}
 It is clear that the equations \eqref{proj-adv} form a square system, we only need to prove its existence.
  Adding equations \eqref{proj-adv1} and \eqref{proj-adv2}, we get 
   \[
     \intc{(P_h^{1,\star} u+P_h^{2,\star} \phi) v_h}  = \intc{(u+\phi) v_h},\quad   \forall v_h \in P^{k-1}(I_j).
   \]
     Adding equations \eqref{proj-adv3} and \eqref{proj-adv4}, we get 
   \[
     \underbrace{(\mean{P_h^{1,\star} u+P_h^{2,\star}\phi}+
     \frac12\jump{(P_h^{1,\star} u+P_h^{2,\star}\phi})}_{=
     (P_h^{1,\star} u+P_h^{2,\star}\phi)^+
     }\Big|_{j-\frac12}
     = u(x_{j-\frac12})+\phi(x_{j-\frac12}).
   \]
   This directly implies that $P_h^{1,\star} u+P_h^{2,\star}\phi = P_h^+(u+\phi)$
   by uniqueness of the projection $P_h^+$.
   Similar, we have 
   $
    P_h^{1,\star} u-P_h^{2,\star} \phi = P_h^-(u-\phi).
   $
   A simple calculation implies the identities in \eqref{transform}. The error estimates are then 
   direct consequences of the estimates in \eqref{proj1d} for $P_h^\pm$. 
 \end{proof}

Now, we are ready to state our main result on the error estimates.
\begin{theorem}\label{thm:adv1d:err}
Assume that the exact solution $u$ of \eqref{advection1d} is sufficiently smooth. 
Let $u_h$ be the numerical solution of the semi-discrete DG scheme \eqref{scheme:adv1d}
using the numerical flux \eqref{flux:opt}. 
Then for $T >0$ there holds the following
error estimate
\begin{equation} \label{thm:adv-est}
\normI {u(T) - u_h(T)} 
+\normI {\phi_h(T)}
\le C (1+T) \ho,
\end{equation}
where $C$ is independent of $h$.
\end{theorem}

\newcommand{\euu}{\varepsilon_u}
\newcommand{\duu}{\delta_u}
\newcommand{\epp}{\varepsilon_\phi}
\newcommand{\dpp}{\delta_\phi}
\newcommand{\puu}{P_h^{1,\star}}
\newcommand{\ppp}{P_h^{2,\star}}
\begin{proof}
The proof is a standard energy argument. We only give a sketch.
We denote 
\begin{align}
\label{split}
 \euu := \puu u - u_h, \quad 
 \duu := u - \puu u, \\
 \epp := \ppp \phi - \phi_h, \quad 
 \epp := \phi -\ppp\phi.\nonumber
\end{align}
Then, consistency the DG scheme \ref{scheme:adv1d} and definition of the projection \eqref{proj-adv}
directly implies that
% \begin{subequations}
\begin{align*}
\sum_{j=1}^N\left(\intc{(\euu)_tv_h-c\,\euu (v_h)_x}\right) -
\left.(\mean{\euu}+\frac12\jump{\epp}) \jump{cv_h}\right|_{j-\frac12}= &\;
\intid{(\duu)_tv_h},\\
\sum_{j=1}^N\left(\intc{(\epp)_t\psi_h+c\,\epp (\psi_h)_x}\right) +
\left.(\mean{\epp}+\frac12\jump{\euu}) \jump{c\phi_h}\right|_{j-\frac12}= &\;
\intid{(\dpp)_t\phi_h},
\end{align*}
% \end{subequations}
for all $(v_h,\phi_h)\in[V_h^k]^2$. 
Taking $v_h=c^{-1}\euu,\phi_h=c^{-1}\epp$ (recall that $c$ is a constant on each $I_j$) in the above error equations and adding, we get
the following energy identity
% \begin{subequations}
\begin{align*}
\intid{c^{-1}(\euu)_t\euu+
c^{-1}(\epp)_t\epp}
=
\intid{c^{-1}(\duu)_t\euu+
c^{-1}(\dpp)_t\epp}
\end{align*}
Finally, the error estimate in Theorem \ref{thm:adv1d:err} is obtained by 
applying the Cauchy-Schwarz inequality, and 
combing the approximation property of the projection in Lemma \ref{lemma:proj}, and an triangle inequality.
\end{proof}
\begin{remark}[$\phi_h$ approximates {\it zero}]
Note that $\phi_h$ is an order $k+1$ approximation to the {\it zero} function.
\end{remark}

% \begin{remark}
%  The numerical flux looks quite similar to the upwinding flux, in which the two jump terms shall be  of 
% \end{remark}

\begin{remark}[A natural extension to systems]
\label{rk:sys}
 This result can be directly used to obtain optimal convergent energy-conserving DG methods
for any constant-coefficient, linear symmetric hyperbolic systems \eqref{1d-symmetric-sys}, with 
 a doubling of the unknowns by introducing the auxiliary {\it zero} function $\bld \phi(x,t)$ that solves
 \begin{align}
 \label{aux-1d}
  \bld B_0\,{\bld \phi}_t - \bld B_1 \bld \phi_x =0,\quad\forall   (x,t)\in I\times (0,T]  
 \end{align}
 with {\it zero} initial condition.
 The resulting scheme reads as follows: Find, for any time $t\in(0,T]$,
 the unique functions $(\bld u_h, \bld\phi_h)\in [V_h^k]^m\times [V_h^k]^m$ such that
 \begin{subequations}
 \label{scheme:sysA}
 \begin{align} 
% \small
%  \label{scheme:sysX1}
\left(\intc{\bld B_0(\bld u_h)_t\cdot \bld v_h-\bld B_1\,\bld u_h\cdot (\bld v_h)_x}\right) +
\widehat{\bld B_1\bld u}_h\cdot \bld v_h^-|_{j+\frac12} -
\widehat{\bld B_1\bld u}_h \cdot\bld v_h^+|_{j-\frac12} = &\;0,\\
%  \label{scheme:sysX2}
\left(\intc{\bld B_0(\bld \phi_h)_t\cdot \bld \psi_h+\bld B_1\,\bld \phi_h\cdot (\bld \psi_h)_x}
\right)-
\widehat{\bld B_1\bld \phi}_h\cdot \bld \psi_h^-|_{j+\frac12} +
\widehat{\bld B_1\bld \phi}_h \cdot\bld \psi_h^+|_{j-\frac12} = &\;0,
\end{align}
with the numerical fluxes
\begin{align}
\label{flux-sysA}
 \widehat{\bld B_1\bld u}_h|_{j-\frac12} = 
\bld B_1\mean{\bld u_h}
+\frac12 \bld B_1\jump{\bld \phi_h},
\quad
 \widehat{\bld B_1\bld \phi}_h|_{j-\frac12} = 
\bld B_1\mean{\bld \phi_h}
+\frac12 \bld B_1\jump{\bld u_h}.
\end{align}
 \end{subequations}

However, doubling the unknowns might be computationally too expensive. 
For certain special and important cases, e.g. acoustics \cite{XingChouShu13}, there exists 
optimal convergence energy-conserving DG methods without the need of doubling the unknowns.
In the next subsection, we derive optimal energy-conserving DG methods for the acoustics equation 
 in a slightly more general form, c.f. \cite{LeVeque02}.
\end{remark}

% \begin{remark}
%  We can also work with variable velocity, with $c$ being piecewise smooth, 
%  optimal convergence of the method will still hold, we refer interested ready to \cite{CockburnDongGuzman10}.
% \end{remark}

\subsection{Optimal energy-conserving DG method for acoustics}
\label{sub:ac}
In this subsection, we consider the following acoustics equation
\begin{align}
\label{acoustics1d}
\left[\begin{tabular}{c}
 $p$\\
 $u$
\end{tabular}
\right]_t
+
% \underbrace{
\left[\begin{tabular}{cc}
 $u_0$&$K_0$\\
 $1/\rho_0$&$u_0$\\
\end{tabular}
\right]
% }_{:=A}
\left[\begin{tabular}{c}
 $p$\\
 $u$
\end{tabular}
\right]_x& =0, &&  (x,t)\in I\times (0,T].
\end{align}
with a smooth periodic initial condition and a periodic boundary condition.
% with a smooth periodic initial condition 
% $p(x,0) = p_0(x)$  and $u(x,0) = w_0(x)$ for $x\in I$.
Here $u$ is the velocity, $p$ is the pressure, and $u_0$ is the background velocity, 
$\rho_0$ is the background density, and $K_0$ is the {\it bulk modulus of compressibility} of the material, c.f. 
\cite{LeVeque02}.
The coefficients $\rho_0, u_0$ and $K_0$ are assumed to be positive constants.

Note that the equation \eqref{acoustics1d} is a $2\times 2$ linear symmetric hyperbolic system, which 
can be recast into the form \eqref{1d-symmetric-sys} with $\bld u=[p, u]'$, and the
coefficient matrices
\begin{align}
\label{acoustics-matrix}
\bld B_0 = 
\left[\begin{tabular}{cc}
 $1/K_0$&$0$\\
 $0$&$\rho_0$\\
\end{tabular}
\right]
, \text{ and }  
\bld B_1 = \left[\begin{tabular}{cc}
 $u_0/K_0$&$1$\\
 $1$&$u_0\rho_0$\\
\end{tabular}
\right].
\end{align}
We see that the energy 
\begin{align}
 \label{energy-acustics}
 E(t) = \int_I (p(x)^2/K_0+\rho_0u(x)^2)\mathrm{dx}
\end{align}
is conserved for the system \eqref{acoustics1d}.

Theorem \ref{thm:sys1d} implies that 
semi-discrete energy-conserving DG method for the resulting 
$2\times 2$ symmetric hyperbolic system 
shall be of the form \eqref{scheme:sys1d} with the coefficient matrices \eqref{acoustics-matrix}, and the  
following numerical flux 
\begin{align}
\label{acoustics-flux-matrix}
 \widehat{\bld B_1\bld u}_h|_{j-\frac12} = 
\bld B_1\mean{\bld u_h}
+ \alpha_{j-\frac12}
\left[\begin{tabular}{cc}
 $0$&$1$\\
 $-1$&$0$\\
\end{tabular}
\right]
\jump{\bld u_h},
\end{align}
with $\alpha_{j-1/2}$ a scalar constant for all $j$.

Translating this condition back to the non-symmetric system \eqref{acoustics1d}, 
we get the following equivalent formulation of the method.
Find, for any time $t \in (0, T]$,
the unique function $(p_h, u_h) = (p_h(t), u_h(t))
\in V_h^k\times V_h^k$ such that
 \begin{subequations}
 \label{scheme:acoustics1d}
\begin{align} \label{scheme:acoustics1d-1}
\intc{(p_h)_tq_h -(u_0\,p_h+K_0u_h) (q_h)_x} +
\widehat{f}_h q_h^-|_{j+\frac12} -
\widehat{f}_h q_h^+|_{j-\frac12} = &\;0,\\
\label{scheme:acoustics1d-2}
\intc{(u_h)_tv_h -(p_h/\rho_0+u_0u_h) (v_h)_x} +
\widehat{g}_h q_h^-|_{j+\frac12} -
\widehat{g}_h q_h^+|_{j-\frac12} = &\;0,
\end{align}
 for all $(q_h,v_h)\in V_h^k\times V_h^k$, for all $j$, where the numerical fluxes
 $\widehat f_h$ and $\widehat g_h$ are given by
 \begin{align}
 \label{acoustics-flux-1}
 \widehat f_h = &\; u_0\mean {p_h}+K_0\mean{u_h}+\alpha_{j-\frac12} K_0\jump{u_h},\\
 \label{acoustics-flux-2}
 \widehat g_h = &\; \mean {p_h}/\rho_0+u_0\mean{u_h}-\alpha_{j-\frac12} \jump{p_h}/\rho_0.
 \end{align}
 \end{subequations}
We state the energy-conservation property of this method in the following Corollary.
\begin{corollary}\label{coro:ac1d}
The  energy 
\[
 E_h(t) = \intid{(p_h(x)^2/K_0+\rho_0u_h(x)^2)}
\]
is conserved by the semi-discrete scheme \eqref{scheme:acoustics1d} 
for all time.
\end{corollary}
\begin{remark}[Alternating flux]
 In the special case when $u_0=0$, taking $\alpha_{j-\frac12}=1/2$ for all $j$, or 
 $\alpha_{j-\frac12}=-1/2$ for all $j$,
 results the 
 optimal convergent, energy conserving  DG method with an alternating flux considered in \cite{XingChouShu13}.
\end{remark}

Next, we turn to the error estimates of the scheme \eqref{scheme:acoustics1d} with a proper choice of the {\it stabilization}
parameter $\alpha_{j-\frac12}$.
It turns out the error estimates is drastically different, which
depends on whether the background velocity $u_0$ 
is subsonic ($u_0<c_0$) or supersonic ($u_0>c_0$), where $c_0:=\sqrt{K_0/\rho_0}$ is the speed of 
sound.
\subsubsection{Subsonic case ($u_0<c_0$)}
In this case, the matrix $\bld B_1$ has a positive eigenvalue
and a negative eigenvalue.
In particular, there exists an orthogonal matrix $\bld S$ with determinant $1$ such that
  \begin{align}
  \label{eig}
   \bld B_1 = \bld S\,\mathrm{diag}([\lambda_+, \lambda_-])\bld S^{-1},
    \end{align}
%   where 
%   \[
%    \lambda_{\pm} = \frac{(u_0/K_0+u_0\rho_0)\pm
%    \sqrt{(u_0/K_0-u_0\rho_0)^2+4}
%    }{2}.
%   \]
with $\lambda_{+}>0>\lambda_{-}$ being the two roots of the quadratic equation
\[
(\lambda-u_0/K_0)(\lambda-u_0\rho_0)-1 =0.
\]
  We have $\lambda_-\lambda_+ =\frac{u_0^2}{c_0^2}-1<0$.
%   Note that, up to a scaling of $-1$ for one row, we can ensure that the orthogonal matrix $\bld S$
%   satisfies
A simple calculation yields that, for any orthogonal matrix $\bld S\in \mathbb{R}^{2\times 2}$
with determinant $1$, there holds
  \begin{align}
  \label{oo}
   \bld S\left[\begin{tabular}{cc}
 $0$&$1$\\
 $-1$&$0$\\
\end{tabular}
\right]\bld S^{-1} = \left[\begin{tabular}{cc}
 $0$&$1$\\
 $-1$&$0$\\
\end{tabular}
\right],
    \end{align}
which will be used in the proof of Lemma \ref{lemma:proja} below.
%   and $\bld w=[w_1, w_2]' = \bld S \bld u$ is the characteristic variable.

We take the stabilization parameter 
$\alpha_{j-\frac12} = \frac12\sqrt{-\lambda_-\lambda_+}=\frac12\sqrt{1-\frac{u_0^2}{c_0^2}}$.  

To derive the optimal error estimate, 
we shall use the following coupled projection.
% specifically designed for the DG scheme \eqref{scheme:acoustics1d}.
We work with vector notation.
For any function $\bld u =(u_1, u_2)\in [H^1(I)]^2$, 
we introduce the following coupled auxiliary projection 
 $P_h^\star \bld u 
%  =(P_h^{1,\star} u_1, P_h^{2,\star} u_2)
 \in [V_h^k]^2$:
 \begin{subequations}
 \label{proj-acoustics}
 \begin{align}
 \label{proj-ac1}
  \intc{P_h^{\star}\bld u\cdot \bld v_h} = \intc{\bld u(x)\cdot\bld v_h} &\quad\quad  \forall \bld v_h \in [P^{k-1}(I_j)]^2,\\
 \label{proj-ac2}
 (\bld B_1\mean {P_h^{\star} \bld u_h} +
 \alpha
\left[\begin{tabular}{cc}
 $0$&$1$\\
 $-1$&$0$\\
\end{tabular}
\right]
 \jump {P_h^{\star} \bld u_h})\Big|_{j-\frac12} &=  \bld B_1\bld u(x_{j-\frac12}),
 \end{align}
 \end{subequations}
 for all $j$, where $\alpha =\frac12 \sqrt{1-\frac{u_0^2}{c_0^2}}$ is the stabilization parameter.
 
Similar to the advection case in Lemma \ref{lemma:proj}, 
the above projection is also an optimal local projection.
 \begin{lemma}
 \label{lemma:proja}
  The projection \eqref{proj-acoustics} is well-defined, and it satisfies
  \begin{subequations}
  \label{transform}
  \begin{align}
   P_h^{\star}\bld u = \bld S 
   \left[\begin{tabular}{c}
 $\Pi_h^{1,\star}w_1$\\[1ex]
 $\Pi_h^{2,\star}w_2$\\
\end{tabular}
\right],
  \end{align}
where $\bld w = (w_1,w_2)=\bld S^{-1}\bld u$ is the characteristic variable, and 
\begin{align}
 \label{pi1}
 \Pi_h^{1,\star}w_1 =&\; \frac12P_h^+\left(w_1+\sqrt{\frac{\lambda_+}{-\lambda_-}}w_2\right)+\frac12
 P_h^-\left(w_1-\sqrt{\frac{\lambda_+}{-\lambda_-}}w_2\right),\\
 \label{pi2}
 \Pi_h^{2,\star}w_2 = &\;\frac12P_h^+\left(w_2+\sqrt{\frac{-\lambda_-}{\lambda_+}}w_1\right)+\frac12
 P_h^-\left(w_2-\sqrt{\frac{-\lambda_-}{\lambda_+}}w_1\right),
\end{align}
In particular, it satisfies
  \begin{align}
  \label{approx-proj-ac}
  \| P_h^{\star} \bld u -\bld u\|_{I_j}\le Ch^{k+1}.
  \end{align}
    \end{subequations}
 \end{lemma}
 \begin{proof}
 The proof follows the lines for that for Lemma \ref{lemma:proj}. 
 We first turn to projection for the characteristic variable $\bld w$, and then transform back to the 
 primitive variable $\bld u$. 
Since $\bld B_1$ and $\bld S$ are constant matrices, we have $\bld S^{-1}P_h^\star\bld u=P_h^\star(\bld S^{-1}\bld u)=
P_h^\star\bld w$. Multiplying both sides of equation \eqref{proj-ac1} by $\bld S$, and 
both side of equations \eqref{proj-ac2} by $\Lambda^{-1}\bld S^{-1}$, and using the fact that 
$\alpha =\frac12 \sqrt{-\lambda_-\lambda_+}$ and using the equation \eqref{oo}, we get the following
projection for the characteristic variable $\bld w$:
\begin{align*}
  \intc{P_h^{\star}\bld w\cdot \bld v_h} = \intc{\bld w(x)\cdot\bld v_h} &\quad\quad  \forall \bld v_h \in [P^{k-1}(I_j)]^2,\\
 (\mean {P_h^{\star} \bld w} +
 \frac12
\left[\begin{tabular}{cc}
 $0$&$\sqrt{-\lambda_-/\lambda_+}$\\
 $\sqrt{-\lambda_+/\lambda_-}$&$0$\\
\end{tabular}
\right]
 \jump {P_h^{\star} \bld w_h})\Big|_{j-\frac12} &=  \bld w(x_{j-\frac12}).
\end{align*}
% where we used the identity \eqref{oo} to get the second equation.
A similar algebraic manipulation as that in the proof of Lemma \ref{lemma:proj} yields
\[
 \sqrt{\lambda_+}P_h^{1,\star} w_1 \pm
 \sqrt{\lambda_-}P_h^{1,\star} w_2 =
 P_h^\pm(
  \sqrt{\lambda_+}w_1
  +\sqrt{-\lambda_-}w_2),
\]
and the equalities and estimate \eqref{transform} in Lemma \ref{lemma:proja} 
follow directly.
 \end{proof}

 With the help of this projection, optimal error estimates follow 
directly. We skip the proof, which is 
 identical to the proof of Theorem \ref{thm:adv1d:err}.
 \begin{theorem}\label{thm:ac1d:err}
Assume that the exact solution $(p,u)$ of \eqref{acoustics1d} is sufficiently smooth. 
Let $(p_h, u_h)$ be the numerical solution of the semi-discrete DG scheme \eqref{scheme:acoustics1d}
with $\alpha_{j-\frac12}=\sqrt{1-u_0^2/c_0^2}$ in the numerical fluxes \eqref{acoustics-flux-1} and 
\eqref{acoustics-flux-2}. 
Then for $T >0$ there holds the following
error estimate
\begin{equation} \label{thm:ac-est}
\normI {u(T) - u_h(T)} 
+
\normI {p(T) - p_h(T)} 
\le C (1+T) \ho,
\end{equation}
where $C$ is independent of $h$.
\end{theorem}

\subsubsection{Supersonic case $(u_0>c_0)$}
\label{subsec:sup}
In this case, the eigenvalues of the matrix $\bld B_1$ are all positive, the construction of a local 
projection $P_h^\star$ in the previous section is no longer valid. 
Hence, we suggest the doubling the unknowns approach, c.f. Remark \ref{rk:sys}, to obtain an 
optimal convergent, energy-conserving semi-discrete DG scheme on general nonuniform meshes.

However, if we insist in working with the original system and 
use scheme \eqref{scheme:acoustics1d}, we can take the stabilization 
parameter $\alpha =\frac12 \sqrt{\frac{u_0^2}{c_0^2}-1}$.
The resulting method can be proven to be optimally convergent on {\it uniform} meshes
for all polynomial degree, but only suboptimal convergent on {\it nonuniform} meshes.
The optimal convergence of this method for all polynomial degree is numerically verified, and 
a loss of convergence order is also numerically observed on nonuniform randomly perturbed meshes.
These numerical tests are not reported in the paper to save space.

The corresponding error analysis is also more involved, which 
follows from similar arguments as in \cite{Bona13, ChenCockburnDong16}.
Without further going into details, 
% we claim that, following the similar arguments as in \cite{Bona13, ChenCockburnDong16}, 
we claim that we can prove
% , following the arguments in \cite{Bona13, ChenCockburnDong16}, 
the projection \eqref{proj-acoustics} is a well-defined {\it global} 
projection for all polynomial degree $k \ge 0$, 
in the supersonic case ($u_0>c_0$), 
which has the approximation property $\|P_h^\star\bld u-\bld u\|_{L^2(I)}\le Ch^{\widetilde k}$, 
where 
$\widetilde k = k+1$ on uniform meshes, and $\widetilde k = k$ on general nonuniform meshes.
% Details of the derivation is left out as an exercise. 
We specifically remark that the global projections
defined in \cite{Bona13, ChenCockburnDong16} require the polynomial degree to be even, 
otherwise is not well-defined.
But due to the coupling term \eqref{proj-ac2}, we do not have this polynomial degree restriction for 
well-possesses of the projection \eqref{proj-acoustics}.
In particular, we do obtain optimal convergence on uniform meshes for {\it any } polynomial degree.

\subsection{Optimal energy-conserving DG methods for linear symmetric hyperbolic systems}
\label{sub:syso}
Now, we turn back to the general, $m$-component,  linear symmetric hyperbolic systems \eqref{1d-symmetric-sys}
with a diagonal, piecewise constant, positive matrix $\bld B_0\in\mathbb{R}^{m\times m}$, and 
a {symmetric constant} matrix $\bld B_1\in\mathbb{R}^{m\times m}$. 

We shall consider the eigenvalue decomposition of $\bld B_1$.
Without loss of generality, we assume that the number of positive eigenvalues for $\bld B_1$
is always greater than or equal to the number of its negative eigenvalues.
Hence, we assume that $\bld B_1$ has $r+s$ positive eigenvalues $\{\lambda_{i}^+\}_{i=1}^{r+s}$, 
and $s$ negative eigenvalues, $\{\lambda_{r+s+1-i}^-\}_{i=1}^{s}$, 
with non-negative integers $r$ and  $s$ satisfying $r+2s\le m$. These {\it nonzero} eigenvalues are ordered such that 
\[
 \lambda_1^+\ge  \lambda_2^+\ge\cdots\lambda_{r+s}^+>0>\lambda_{r+s}^-\ge\cdots\ge\lambda_{r+1}^-.
\]
We denote the diagonal eigenvalue matrix of $\bld B_1$ as 
\begin{subequations}
\label{eig-B}
\begin{align}
\label{eig-B1}
\Lambda = \mathrm{diag}([\lambda_1^+,\cdots,\lambda_{r+s}^+,\underbrace{0,\cdots,0}_{m-r-2s \text{ zeros}},\lambda_{r+s}^-,\cdots,\lambda_{r+1}^-])
\in \mathbb{R}^{m\times m},
\end{align}
and the corresponds orthogonal eigenvalue decomposition 
\begin{align}
\label{eig-B2}
 \bld B_1 = \bld S\Lambda \bld S^{-1},\quad\quad \bld S\in \mathbb{R}^{m\times m} \text{ is orthogonal
 with determinant $1$}.
\end{align}
\end{subequations}
We denote the characteristic variable $\bld w=(w_1,\cdots, w_m)=\bld S^{-1}\bld u$, 
so the characteristic component $w_i$ has wave speed $\Lambda(i, i)$.

Based on the discussion in the previous two subsection,  for each positive integer $\mu\le s$, 
we shall pair the characteristic variables 
$w_{r+\mu}$, with wave speed $\lambda_{r+\mu}^+$, and $w_{m+1-\mu}$,
with wave speed $\lambda_{r+\mu}^-$, and 
% follow the discussion in subsection \ref{sub:ac} to 
consider the optimal energy-conserving numerical flux \eqref{acoustics-flux-matrix} 
for the pair $(w_{r+\mu}, w_{m+1-\mu})$.
And for the remaining $r$ variables, we shall follow the discussion in subsection \ref{sub:ad}
to introduce auxiliary {\it zero} variables that travel with the negative speed 
$\lambda_{\mu}^-:=-\lambda_{\mu}^+$ for $1\le \mu\le r$.
To be more precise, we consider the following $m+r$ component, augmented system
for the variable $\widetilde{\bld u} = [\bld u;\bld \phi]$:
\begin{subequations}
\label{sys:aug}
\begin{align}
  \widetilde{\bld B_0}\,{\widetilde{\bld u}}_t +   \widetilde{\bld B_1} \widetilde{\bld u}_x & =0, &&\hspace{-2.8cm}  (x,t)\in I\times (0,T], 
\end{align}
with 
\begin{align}
 \widetilde{\bld B_0} = \left[\begin{tabular}{cc}
                               $\bld B_0$ & 0\\
                               0 & $\bld I_{r}$\\                              
                              \end{tabular}
\right], \quad \text{and} \;\;
 \widetilde{\bld B_1} = \left[\begin{tabular}{cc}
                               $\bld B_1$ & 0\\
                               0 & $\mathrm{diag}([\lambda_{r}^-,\cdots,\lambda_1^-])$\\                              
                              \end{tabular}
\right],
\end{align}
\end{subequations}
with initial condition 
$\widetilde{\bld u}(x,0) = [\bld u(x);0]$.
Here $\bld I_r$ is the $r\times r$ identity matrix and $\bld \phi$ has $r$ components.
Note that the augmented matrix $\widetilde{\bld B_1}$ has the following eigenvalue decomposition
\begin{subequations}
\label{eig-BB}
\begin{align}
\label{eig-Bt1}
\widetilde{\bld B_1}= \widetilde{\bld S}\widetilde{\Lambda}\widetilde{\bld S}^{-1},
\end{align}
with 
\begin{align}
\label{eig-Bt2}
\widetilde{\Lambda} = 
\mathrm{diag}([\lambda_1^+,\cdots,\lambda_{r+s}^+,\underbrace{0,\cdots,0}_{m-r-2s \text{ zeros}},\lambda_{r+s}^-,\cdots,\lambda_{1}^-]),
\end{align}
and
\begin{align}
\label{eig-Bt3}
\widetilde{S}= \left[\begin{tabular}{cc}
                               $\bld S$ & 0\\
                               0 & $\bld I_{r}$\\                              
                              \end{tabular}
\right].
\end{align}

To further simplify notation, for each positive integer $\mu\le r+s$ 
we denote the anti-symmetric $(m+r)\times (m+r)$ matrices $\widetilde{\bld R_\mu}$ that only has non-vanishing
components on the $(\mu, m+r-\mu)$ and $(m+r-\mu,\mu)$ locations, with 
\[
 \widetilde{\bld R_\mu}(\mu, m+r-\mu) =1,\quad \text{ and\;\; }
 \widetilde{\bld R_\mu}(m+r-\mu,\mu) =-1.
\]
\end{subequations}

Finally, we are ready to state our main result on the optimal energy conserving semi-discrete DG method
for the augmented system \eqref{sys:aug}. The proof is omitted since it directly follows from the 
discussion in the previous two subsection.
\begin{theorem}
\label{thm:sys1dX}
 Assume that the exact solution $\widetilde{\bld u}$ of \eqref{sys:aug} is sufficiently smooth. 
Let $\widetilde{\bld u}_h\in [V_h^k]^{m+r}$ be the numerical solution of the following semi-discrete DG scheme:
\begin{subequations}
\label{scheme:sys1dX}
\begin{align} 
 \label{scheme:sys1dt}
\intc{\widetilde{\bld B_0}(\widetilde{\bld u}_h)_t\cdot \widetilde{\bld v}_h} - 
\intc{\widetilde{\bld B_1}\,\widetilde{\bld u}_h\cdot (\widetilde{\bld v}_h)_x} \quad\quad&\nonumber\\
+\widehat{\widetilde{\bld B_1}\widetilde{\bld u}_h}\cdot \widetilde{\bld v}_h^-|_{j+\frac12} -
\widehat{\widetilde{\bld B_1}\widetilde{\bld u}_h} \cdot\widetilde{\bld v}_h^+|_{j-\frac12} &= \;0,
\end{align}
for all $\widetilde{\bld v}_h\in [V_h^k]^{m+r}$ and all $\jj.$, with the 
numerical flux
\begin{align}\label{flux:sys1dt}
\widehat{\widetilde{\bld B_1}\widetilde{\bld u}_h}|_{j-\frac12} = 
\widetilde{\bld B_1}\mean{\widetilde{\bld u}_h}
+ 
\frac12\sum_{\mu=1}^{r+s}\sqrt{|\lambda_\mu^+\lambda_\mu^-|}\,
\widetilde{\bld S}\widetilde{\bld R}_\mu
\widetilde{\bld S}^{-1}
\jump{\widetilde{\bld u}_h},
\end{align}
\end{subequations}
Then, the total energy 
\[
 \widetilde{E}(t) =\intid{(\widetilde{\bld B_0}\widetilde{\bld u}_h)\cdot\widetilde{\bld u}_h}
\]
is conserved for all time.
Moreover, for $T >0$ there holds the following
error estimate
\begin{equation} \label{thm:sys-est}
\normI {\widetilde{\bld u}(T) - \widetilde{\bld u}_h(T)} 
\le C (1+T) \ho,
\end{equation}
where $C$ is independent of $h$.
\end{theorem}

% \begin{remark}[Implementation]
% The numerical flux \eqref{flux:sys1dt}
% can be obtained efficiently by working with the characteristic
%  variable 
% $\widetilde{\bld w}_h= [w_{1,h},\cdots, w_{m+r,h}]' = \widetilde{\bld S}^{-1}{\widetilde{\bld u}}_h
% $.
% Define the characteristic numerical flux $\widehat{\bld \Lambda\widetilde{\bld w}_h}$, 
% whose $i$-th component equals 
% $\lambda_i^+\mean{w_{i,h}} +\frac12\alpha_i\jump{w_{m+r+1-i,h}}$, and 
% $m+r+1-i$-th component equals 
% $\lambda_{i}^-\mean{w_{m+r+1-i,h}} -\frac12\alpha_i\jump{w_{i,h}}$, 
% for $1\le i \le r+s$, where $\alpha_i = \sqrt{|\lambda_i^+\lambda_i^-|}$, and those
% other components equals zero, i.e.,
% \begin{align*}
%  \widehat{\bld \Lambda\widetilde{\bld w}_h}
%  := \left[
%  \begin{tabular}{c}
%   $\lambda_1^+\mean{w_{1,h}} +\frac12\alpha_1\jump{w_{m+r,h}} $ \\[.5ex]
%   $\vdots$\\
%   $\lambda_{r+s}^+\mean{w_{r+s,h}} +\frac12\alpha_{r+s}\jump{w_{m+r+1-(r+s),h}} $ \\[.5ex]
%   $0$\\
%   $ \vdots$\\
%   $0$\\
%   $\lambda_{r+s}^-\mean{w_{m+r+1-(r+s),h}} -\frac12\alpha_{r+s}\jump{w_{r+s,h}} $ \\[.5ex]
%   $\vdots$\\
%   $\lambda_1^-\mean{w_{m+r,h}} -\frac12\alpha_1\jump{w_{1,h}} $ \\
%  \end{tabular}
%  \right].
% \end{align*}
% The numerical flux \eqref{flux:sys1dt} is then obtained by transform back to the 
% primative variable
% \[
% \widehat{\widetilde{\bld B_1}\widetilde{\bld u}_h}|_{j-\frac12} = 
% \widetilde{\bld S}\,\widehat{\bld \Lambda\widetilde{\bld w}_h}
% \]
% 
% \end{remark}

% for the following enlarged system
% Now, the idea is that
\begin{remark}[Doubling the unknowns]
If we simply double the unknowns,  
% and consider the coupled system \eqref{1d-symmetric-sys} and \eqref{aux-1d},
the scheme \eqref{scheme:sys1dX} applied to 
the resulting coupled system \eqref{1d-symmetric-sys} and \eqref{aux-1d} is slightly different from the scheme 
\eqref{scheme:sysA} introduced in Remark \ref{rk:sys}, with 
the only difference being
the numerical flux \eqref{flux-sysA}
replaced by the following characteristic-wise one:
\begin{align}
 \label{fluxAA}
  \widehat{\bld B_1\bld u}_h|_{j-\frac12} = 
\bld B_1\mean{\bld u_h}
+\frac12 |\bld B_1|\jump{\bld \phi_h},
\quad
 \widehat{\bld B_1\bld \phi}_h|_{j-\frac12} = 
\bld B_1\mean{\bld \phi_h}
+\frac12 |\bld B_1|\jump{\bld u_h},
\end{align}
where $|\bld B_1| = \bld S|\bld{\Lambda}|\bld S^{-1}$. 
Although this flux is slightly more expensive than the component-wise flux \eqref{flux-sysA} with both 
methods optimally convergent, 
the extension of the flux \eqref{fluxAA} to multi-dimensions 
on unstructured meshes is more promising than that for \eqref{flux-sysA}.
See also Remark \ref{rk:sys2d} below.
\end{remark}

\subsection{High-order energy-conserving Lax-Wendroff time discretization}
\label{sec:time}
In this section, we consider the temporal discretization
of the semi-discrete scheme \eqref{scheme:sys1dX}.
We introduce an 
explicit, high-order, energy-conserving Lax-Wendroff time integrator.

\newcommand{\wu}{\widetilde{\bld u}}
\newcommand{\wv}{\widetilde{\bld v}}
To simplify notation, we denote 
\begin{subequations}
\label{bilinear-forms}
\begin{align}
 M_h({\wu_h}, \wv_h):=&\;
 \sum_{j=1}^N\intc{\widetilde{\bld B_0}\,\widetilde{\bld u}_h\cdot \widetilde{\bld v}_h} \\
 B_h({\wu_h}, \wv_h):=&\; \sum_{j=1}^N\left(\intc{\widetilde{\bld B_1}\,\widetilde{\bld u}_h\cdot 
 (\widetilde{\bld v}_h)_x}
+\left.\widehat{\widetilde{\bld B_1}\widetilde{\bld u}_h} 
\cdot\jump{\wv_h}\right|_{j-\frac12} \right).
\end{align}
\end{subequations}
The semi-discrete scheme \eqref{scheme:sys1dX} is to
find $\wu_h(t)\in[V_h^k]^{m+d}$ such that
\begin{align}
 M_h(({\wu_h})_t, \wv_h) = 
  B_h({\wu_h}, \wv_h), \quad \forall\; \wv_h\in [V_h^k]^{m+d}.
\end{align}
Introducing a set of basis, e.g. orthogonal Legendre polynomials, 
for the DG space $V_h^k$, and denoting $[{\bld u}_h(t)]$ as the vector of 
degrees of freedom for $\widetilde{\bld u}_h(t)$, the above
semi-discrete scheme can be expressed as the following matrix-vector form:
\begin{align}
 \label{discrete-form}
[{\bld u}_h]_t =\bld M^{-1}\bld  A [\bld u_h],
\end{align}
where $\bld M$ is the ($\widetilde{\bld B_0}$-weighted) mass matrix, which is diagonal
if one choose the Legendre basis,
and $\bld A$ is the matrix corresponding to the spatial operator $B_h(\cdot,\cdot)$.
% Energy conservation proper of the scheme \eqref{scheme:sys1dX} ensure that the matrix $A$ is anti-symmetric.
A reformulation of the energy conservation property of the scheme \eqref{scheme:sys1dX}
in Theorem \ref{thm:sys1dX} in this 
matrix-vector notation is given below:
\begin{subequations}
\label{energy}
\begin{align}
\label{energy-1}
 (\bld M[{\bld u}_h(t)])\cdot [{\bld u}_h(t)]
 =&\;(\bld M[{\bld u}_h(0)])\cdot [{\bld u}_h(0)],\quad\quad \forall t>0,\\
\label{energy-2}
 \bld A \text{ is anti-symmetric}.
\end{align}
\end{subequations}

Now, we consider a class of Lax-Wendroff time discretization 
for the semi-discrete scheme \eqref{discrete-form}
that preserve a discrete version of the energy conservation identity \eqref{energy-1}.
The Lax-Wendroff time discretization \cite{LaxWendroff60} is a 
high-order method known as the Cauchy-Kowalewski type procedure
in the literature, which relies on converting
each time derivative in a truncated temporal Taylor expansion (with expected accuracy) 
of the solution into spatial derivatives by repeatedly using the 
underlying differential equation and its differentiated form. 
We directly work with the semi-discrete scheme \eqref{discrete-form} without going back to the 
 PDE \eqref{sys:aug}.

% Assuming $u: [0,T]\rightarrow \mathbb{R}$ is a sufficiently smooth function, by Taylor approximation, we have
% \begin{align}
%  \label{taylor}
% {u(t+\Delta t)- u(t-\Delta t)} = \sum_{i=0}^r \frac{2\Delta t^{2i+1}}{(2i+1)!} u^{(2i+1)}(t)+\mathcal{O}(\Delta t^{2r+2}).
% \end{align}
% The high-order time discretization of \eqref{discrete-form} is obtained using the above equation, and 
% repeatedly replacing the time derivative of $[\bld u_h]$ using the identity \eqref{discrete-form}. 
% 
Let $0=t_0< t_1<\cdots<t_N=T$ be a partition of the interval $[0,T]$ with time step $\Delta t = t_{n+1}-t_n$. Here uniform
time step $\Delta t$ is used. 
For, any non-negative integer $r$, a 
temporal ($2r+1$)-th stage, 
($2r+2$)-th order accurate fully discrete approximation $[\bld u_h^n]$ for \eqref{discrete-form}
are construction as follows: for $n=1,\cdots, N-1$, $[\bld u_h^{n+1}]$ is given by 
\begin{align}
 \label{fully-discrete}
 [\bld u_h^{n+1}]- [\bld u_h^{n-1}] = \sum_{i=0}^r \frac{2\Delta t^{2i+1}}{(2i+1)!} 
 (\bld M^{-1}\bld A)^{2i+1}[\bld u_h^n].
\end{align}
We specifically mention that 
the above time discretization is obtained by the following Taylor approximation and the 
Lax-Wendroff procedure of converting the time derivatives  into the discrete 
spatial operators using \eqref{discrete-form}, 
\begin{align*}
%  \label{taylor}
{u(t+\Delta t)- u(t-\Delta t)} = \sum_{i=0}^r \frac{2\Delta t^{2i+1}}{(2i+1)!} u^{(2i+1)}(t)+\mathcal{O}(\Delta t^{2r+2}).
\end{align*}
% Note that here we need two initialization stages $[\bld u_h^0]$, and $[\bld u_h^1]$.
Note that for $r=0$, we get the usual second-order accurate leap-frog method
\[
 [\bld u_h^{n+1}]- [\bld u_h^{n-1}] = {2\Delta t}\, \bld M^{-1}\bld A[\bld u_h^n].
\]

% Optimal $L^2$-convergence $\mathcal{O}({h^{k+1}+\Delta t^{2r+2}})$ of 
% this fully discrete scheme can be proven using the semi-discrete convergence result in Theorem 
% \ref{thm:sys1dX} with
% the standard elliptic projector approach \cite{Wheeler75}; we omit the details.
The energy conservation property of the fully discrete scheme is documented in the next theorem.
\begin{theorem}
 The fully discrete scheme \eqref{fully-discrete} satisfies the energy identity
 \[
   (\bld M[\bld u_h^{n+1}])\cdot[\bld u_h^n] =
      (\bld M[\bld u_h^{n}])\cdot[\bld u_h^{n-1}]
 \]
\end{theorem}
\begin{proof}
The equality is obtained by dotting the equation \eqref{fully-discrete} 
with $\bld M [\bld u_h^n]$, and taking into account the anti-symmetry of 
the matrix $\bld A$, \eqref{energy-2}.
\end{proof}

\begin{remark}[Runge-Kutta type time discretization]
Recall that the $r$-stage $r$-th order accurate explicit Runge-Kutta method 
for \eqref{discrete-form} can be 
write as the following Lax-Wendroff form, cf. \cite{SunShu17},
\begin{align}
 \label{rk}
  [\bld u_h^{n+1}]- [\bld u_h^{n}] = \sum_{i=1}^{r} \frac{\Delta t^{i}}{(i)!} 
 (\bld M^{-1}\bld A)^{i}[\bld u_h^n].
\end{align}
This time discretization is not energy-conserving.
% , but it enjoys a better stability property.
\end{remark}

\begin{remark}[Time-dependent source term]
\label{rk-source}
 The above time discretization \eqref{fully-discrete} and \eqref{rk} can be easily modified to 
treat a {\it linear} time-dependent source term without sacrificing its formal order of accuracy.
In particular, consider the follow system of ODEs:
\begin{align}
\label{source}
[{\bld u}_h]_t =\bld M^{-1}\Big({\bld  A} [\bld u_h] + [\bld f(t)]\Big), 
\end{align}
with $\bld f(t)$ takes into account possible {\it linear} 
boundary/volume source terms.
The energy-conserving Lax-Wendroff method then reads
\begin{subequations}
\label{source-t}
  \begin{align}
 \label{fully-discrete-s}
 [\bld u_h^{n+1}]- [\bld u_h^{n-1}] = \sum_{i=0}^r \left(\frac{2\Delta t^{2i+1}}{(2i+1)!} 
 [d^{2i+1}\bld u_h^n]\right),
\end{align}
where $[d^{0}\bld u_h] = [\bld u_h]$, and $[d^{s}\bld u_h]$, $s\ge 1$, 
is recursively defined through the following map:
\begin{align}
 \label{map-s}
 [d^{s}\bld u_h^n] =  &\;
 \bld M^{-1}\Big({\bld  A} [d^{s-1}\bld u_h] + [\bld f^{(s-1)}(t)]\Big)\quad s\ge 1,
\end{align}
with $\bld f^{(s-1)}(t)$ being the $(s-1)$-th derivative of $\bld f(t)$.
And the Runge-Kutta type Lax-Wendroff method reads
  \begin{align}
 \label{rk-s}
 [\bld u_h^{n+1}]- [\bld u_h^{n}] = \sum_{i=0}^r \left(\frac{\Delta t^{i+1}}{(i+1)!} 
 [d^{i+1}\bld u_h^n]\right),
\end{align}
\end{subequations}
We specifically mention the Lax-Wendroff method \eqref{rk-s}
is different from the classical Runge-Kutta method for the time-dependent source term treatment.
The Runge-Kutta method is well-known to suffer from the so-called order reduction when 
boundary source term were not properly adjusted, c.f. \cite{rkbd}.
But the Lax-Wendroff methods \eqref{fully-discrete-s} and 
\eqref{rk-s} do not suffer from such order reduction since all 
spatial derivatives are calculated on the same time level.
\end{remark}

\begin{remark}[Lax-Wendroff time discretization for nonlinear equations]
We shall point out that the Lax-Wendroff method is considerably 
more complex to derive for nonlinear equations; see \cite{LaxWendroff60,GuoQiuQiu15}, 
as one would need to take into account the time derivative of the matrix $\bld A$, 
that depends on the solution $\bld u_h$.
In this case, instead of the current method of lines approach (first spatial DG discretization, then 
temporal Lax-Wendroff discretization), 
we shall first discretize the PDE in time then 
apply a proper spatial DG discretization, which 
takes into account  higher order derivatives.
\end{remark}

% \begin{remark}[Stability of the fully-discrete scheme]
% Stability analysis of the fully discrete schemes 
% \eqref{rk} was recently for $r\le 4$ 
% \end{remark}

\subsection{Boundary treatment}
\label{sec:bdry}
For boundary value problems, special care need to 
be taken for the numerical fluxes at the boundary.
Here we discuss how to impose the inflow boundary conditions.

Consider the linear symmetric hyperbolic system \eqref{sys:aug}, 
where we suppressed the tilde notation for ease of presentation,
with initial condition 
$\bld u(x,0) = \bld u_0$, and {\it inflow} boundary condition 
\begin{subequations}
\begin{align}
\label{bdry-cd}
  \bld B_1^+\bld u(a, t) = \bld B_1^+\bld u_a(t),\quad \quad
  \bld B_1^-\bld u(b, t) = \bld B_1^-\bld u_b(t),
\end{align}
where
\begin{align}
 \bld B_1^+ = &\;\bld S\,\mathrm{diag}([\max(\lambda_1,0), \cdots,
 \max(\lambda_{m+r},0)])\bld S^{-1},\\
 \bld B_1^- = &\;\bld S\,\mathrm{diag}([\min(\lambda_1,0), \cdots,
 \min(\lambda_{m+r},0)])\bld S^{-1},
\end{align}
\end{subequations}
and $\bld B_1 = \bld S\,\mathrm{diag}([\lambda_1,\cdots,\lambda_{m+r}])\bld S^{-1}$ 
is an eigenvalue decomposition of $\bld B_1$.
We denote $|\bld B_1|:= \bld B_1^+- \bld B_1^-$.
We further assume that all eigenvalues of $\bld B_1$ are non-zero.
The PDE \eqref{sys:aug} (ignoring the tilde notation) with the boundary condition has the following 
energy identity:
\begin{align}
\label{energy-id2}
 \frac{d}{dt}\left(\frac12\intid{\bld B_0\bld u(t)\cdot \bld u(t)}\right)
 -\frac12\bld B_1^-\bld u(a,t)\cdot\bld u(a,t)^2 
 +\frac12\bld B_1^+\bld u(b,t)\cdot\bld u(b,t)^2 &\nonumber\\
 =
 \frac12\bld B_1^+\bld u_a(t)\cdot\bld u_a(t) 
 -\frac12\bld B_1^-\bld u_b(t)\cdot \bld u_b(t).&
\end{align}

% \subsubsection{Upwinding numerical flux on the boundary}
On the two end points of the interval $I$, we simply take the following 
upwinding numerical flux:
\begin{subequations}
 \label{bdry}
 \begin{align}
\widehat{\bld B_1 \bld u_h}|_{x=a} = 
\bld B_1^+\bld u_a + \bld B_1^-\bld u_h^+|_{x=a},\\
\widehat{\bld B_1 \bld u_h}|_{x=b} = 
\bld B_1^-\bld u_b + \bld B_1^+\bld u_h^+|_{x=b}.
 \end{align}
\end{subequations}

The resulting semi-discrete scheme enjoys a similar energy identity as \eqref{energy-id2}
and is optimal convergent. 
The proof is similar to the periodic case \eqref{thm:sys1dX}, and is omitted for 
simplicity.
% In particular, we  construct 
% a local projection 
% The semi-discrete scheme \eqref{scheme:sys1dX} for the initial boundary value problem
% \eqref{sys:aug} with
% boundary condition \eqref{bdry-cd},  using 
% the upwinding boundary numerical flux \eqref{bdry} enjoys the following energy identity.
\begin{theorem}
 Assume that the exact solution $\widetilde{\bld u}$ of \eqref{sys:aug} with boundary 
 condition \eqref{bdry-cd}
 is sufficiently smooth. 
Let $\widetilde{\bld u}_h\in [V_h^k]^{m+r}$ be the numerical solution of 
\eqref{scheme:sys1dX} with internal (energy-conserving) numerical flux \eqref{flux:sys1dt},
and
boundary (upwinding) numerical flux \eqref{bdry}.
Then, the following energy identity holds
\begin{align}
  \frac{d}{dt}\left(\frac12\intid{\bld B_0\bld u_h(t)\cdot \bld u_h(t)}\right)
 +\left.\frac12|\bld B_1|\bld u_h^+\cdot\bld u_h^+\right|_{x=a} 
 +\left.\frac12|\bld B_1|\bld u_h^-\cdot\bld u_h^-\right|_{x=b} &\nonumber\\
 =
\left.\bld B_1^-\bld u_a(t)\cdot\bld u_h^+\right|_{x=a} 
\left. -\bld B_1^-\bld u_b(t)\cdot \bld u_h^-\right|_{x=b}.&
\end{align}
Moreover, for $T >0$ there holds the following
error estimate
\begin{equation} \label{thm:sys-est-bd}
\normI {{\bld u}(T) - {\bld u}_h(T)} 
\le C (1+T) \ho,
\end{equation}
where $C$ is independent of $h$.
\end{theorem}

% \subsubsection{Time discretization}
% The resulting 
% semi-discrete scheme \eqref{scheme:sys1dX} for the initial boundary value problem 
% \eqref{sys:aug}
% can be, again, written in the following matrix-vector form:
%  \begin{align}
%  \label{discrete-form-B}
% [{\bld u}_h]_t =\bld M^{-1}\Big({\bld  A} [\bld u_h] + [\bld f(t)]\Big),
% \end{align}
% where $\bld f(t)$ takes into account boundary condition, and 
% the matrix ${\bld A}$ now satisfies ${\bld A}+{\bld A}^T\ge 0$.
% We may still apply the Lax-Wendroff procedure for \eqref{discrete-form-B}, and obtain 
% the following fully discrete scheme (assuming uniform time step $\Delta t$):
% \begin{align}
%  \label{fully-discrete-B}
%  [\bld u_h^{n+1}]- [\bld u_h^{n-1}] = \sum_{i=0}^r \left(\frac{2\Delta t^{2i+1}}{(2i+1)!} 
%  [d^{2i+1}\bld u_h^n]\right),
% \end{align}
% where $[d^{0}\bld u_h] = [\bld u_h]$, and $[d^{s}\bld u_h]$, $s\ge 1$, 
% is recursively defined through the following map:
% \begin{align}
%  \label{map}
%  [d^{s}\bld u_h^n] =  &\;
%  \bld M^{-1}\Big({\bld  A} [d^{s-1}\bld u_h] + [\bld f^{(s-1)}(t)]\Big)\quad s\ge 1.
% \end{align}
\begin{remark}[Time discretization, stability issue]
 The semi-discrete DG scheme for the boundary value problem naturally 
 leads to the ODE system \eqref{source},
 where $\bld f(t)$ takes into account the boundary condition. 
 We can simply apply the time-discretization \eqref{fully-discrete-s} or \eqref{rk-s}
 as discussed in Remark \ref{rk-source}.

However, our numerical results, not reported in this paper,  showed that the resulting 
fully discrete scheme using the time discretization \eqref{fully-discrete-s}
is {\it unconditionally unstable}.
Such instability was not observed for the Runge-Kutta type time discretization \eqref{rk-s}.

Similar boundary-driven instability was documented in the literature for 
energy-conserving schemes such as the finite difference leap-frog method,
c.f. \cite{AbarbanelGottlieb79}, which is identical to the
lowest-order $P^0$-DG method with a central flux and 
a leap-frog time stepping on uniform meshes.
One remedy to cure this instability for the leap-frog method, c.f. \cite{AbarbanelGottlieb79}, 
was to simply modify the leap-frog time-stepping on cells that touch the boundary 
to be a forward Euler time stepping.
% We extended this idea to the higher-oder methods, and 
% propose to modify the time integration 
% for the two cells near 
% the boundary using the Runge-Kutta
% type scheme \eqref{rk}:
% \begin{align}
%  \label{fully-discrete-B}
%  [\bld u_{h,bdry}^{n+1}]- [\bld u_{h,bdry}^{n}] = \sum_{i=0}^{2r} \left(\frac{\Delta t^{i+1}}{(i+1)!} 
%  [d^{i+1}\bld u_{h,bdry}^n]\right),
% \end{align}
% where $ [\bld u_{h,bdry}]$  corresponds to the vector of degrees of freedom on the two 
% cells near the boundary. 

We can extend this idea to the higher-order Lax-Wendroff methods as follows:
for cells not touching the boundary, use the ($2r+1$)-stage energy-conserving 
Lax-Wendroff method \eqref{fully-discrete-s}, 
and for cells that touch the boundary, use the ($2r+1$)-stage 
Runge-Kutta type Lax-Wendroff method \eqref{rk-s}.
The resulting scheme is numerically shown, with results not reported in 
this paper to save space,
to be high-order accurate and 
conditionally stable, although a detailed stability analysis is missing.
% For the lowest-order case, 
% we conjecture that stability  the 
% We refer to the GKS theory \cite{Gustafsson72} 
% for a discussion of stability issues for initial-boundary value problems.
\end{remark}

% \begin{remark}
% It is well-known that stability issues usually arise energy-conserving schemes 
% such as the finite difference leap-frog method 
% is unconditionally unstable, c.f. \cite{AbarbanelGottlieb79}.
% \end{remark}
% \end{remark}

\section{Energy-conserving DG methods for the multidimensional case}\label{sec:2d}
In this section, we present the  energy-conserving DG methods
for the multidimensional symmetric linear hyperbolic systems. 
Without loss of generality, we describe our DG scheme
in two dimensions $(d = 2)$; all the arguments can be easily
extended to the more general cases $d >2$.

\newcommand{\B}{{\bld B}}
\newcommand{\Bn}{{\bld B_{\bld n}}}
\newcommand{\Bnp}{{\bld B_{\bld n^+}}}
\newcommand{\Bnm}{{\bld B_{\bld n^-}}}
We shall restrict ourselves mainly to the following two-dimensional
system of linear symmetric hyperbolic conservation laws
problem\begin{subequations}\label{sys2d}
\begin{align}
\B_0\,\bld u_t + \B_1\, \bld u_x + \B_2\,\bld u_{y} & =0, &&  
\hspace{-1.8cm}  (x,y,t)\in \Omega \times (0,T], \\
             \bld u(x,y,0)  & = \bld u_0(x,y), &&  \hspace{-1.8cm}  (x,y) \in \Omega,
\end{align}
\end{subequations}
where $\bld B_0: \Omega\rightarrow \mathbb{R}^{m\times m}$ is a positive, 
piecewise-constant, diagonal matrix,
$\bld B_1, \B_2\in \mathbb{R}^{m\times m}$ are two symmetric matrices.
For the sake of simplicity, we consider only the periodic boundary conditions.
% The domain $\Omega$ is a periodic box.

Given any direction field $\bld n=(n_x,n_y)$, we denote the matrix 
\begin{align}
 \label{Bn}
 \bld B_{\bld n} : = n_x \bld B_1 +  n_y \bld B_2.
\end{align}
Based on the one-dimensional results, 
we shall first derive an energy-conserving DG methods for \eqref{sys2d} in
the case when the matrix
$\bld B_{\bld n }$ 
has the same number 
of positive and negative eigenvalues for any direction $\bld n$, denoted 
as ${r_{\bld n}\ge 0}$. The number of positive eigenvalues
$r_{\bld n}$ may be different for different direction $\bld n$.
We call such system a {\it linear symmetric hyperbolic system with paired eigenvalues}.
We denote the (orthogonal) eigenvalue decomposition of $\B_{\bld n}$ as
\begin{subequations}
 \label{eig-d}
 \begin{align}
  \B_{\bld n} = \bld S_{\bld n}\bld\Lambda_{\bld n}\bld S_{\bld n}^{-1},
 \end{align}
 with the eigenvalues in the diagonal matrix in descending order
 \begin{align}
   \bld\Lambda_{\bld n}:= \mathrm{diag}([\lambda_{\bld n,1}^{+},\cdots,
   \lambda_{\bld n,r_{\bld n}}^{+},0,\cdots,0,
  \lambda_{\bld n,r_{\bld n}}^{-},\cdots,\lambda_{\bld n, 1}^{-}]).
 \end{align}
\end{subequations}

We then give examples including the advection, acoustics, aeroacoustics, electromagnetism, and
elastodynamics that shall fit into the framework. 
The key idea follows from
the one-dimensional case by adding auxiliary {\it zero} equations to the system so that 
we get a system with paired eigenvalues.

\subsection{Notation and definitions in the two-dimensional case}
% \subsubsection{The meshes}
Let $\oh = \{K\}$ denote a conforming triangulation of $\Omega$ with shape-regular triangular/rectangular 
elements $K$,
and set $\poh = \{\pk : K \in \oh\}$ where 
$\pk$ is the boundary of the element $K$.
Denote $\mathcal{E}_h$ be the collection of edges in the mesh $\oh$.
For each $K \in \oh$, we
denote by $h_K$ the diameter of $K$ and set, as usual,
$h = \max_{K \in \oh} h_K$. The finite element space associated with
the mesh $\oh$ is of the form
$$
V_h^k : = \{  v \in L^2(\Omega) : v|_K \in V_k(K)  \quad \forall K \in \oh \},
$$
where 
\begin{align*}
 V_k(K) =\left\{
 \begin{tabular}{ll}
  $P^k(K)$ & if $K$ is a triangle,\\[.1cm]
  $Q^k(K)$ & if $K$ is a rectangle,
 \end{tabular}
 \right.
\end{align*}
and $P^k(K)$ is the space of 
polynomials of degrees at most $k$ on $K$, and 
$Q^k(K)$
is the tensor product of polynomials of degrees
at most $k$ in each variable.

% \subsubsection{Jump and average}
We would like to adopt the following notation for the average and jumps of any function $\phi_h$ in the 
DG space $V_h^k$.
Let $E\in \mathcal{E}_h$ be an edge shared by two elements $K^+$ and $K^-$. Let $\bld n^\pm = (n_x^\pm, n_y^\pm)$ be the normal 
direction on $E$ from $K^\pm$. 
We select the unique element $K^-$ such that the direction $\bld n^- =(n_x^-,n_y^-)$ satisfies
\begin{align}
 \label{normal-p}
\bld v_{\mathrm{ref}}\cdot \bld n^-\ge 0,
 \text{ (when $\bld v_{ref}\cdot \bld n^-=0$ we take $\bld n^-$ such that $n_x^->0$)},
\end{align}
where $\bld v_{\mathrm{ref}} = (1,1)$ is an 
{\it artificial (velocity) vector} used to single out the unique $K^-$. See an illustration in 
Figure \ref{fig:E}.
\begin{figure}
\caption{Illustration of the choice of direction for an edge $E$ shared by 
two triangles $K^-$ and $K^+$}
\includegraphics[width=.6\textwidth]{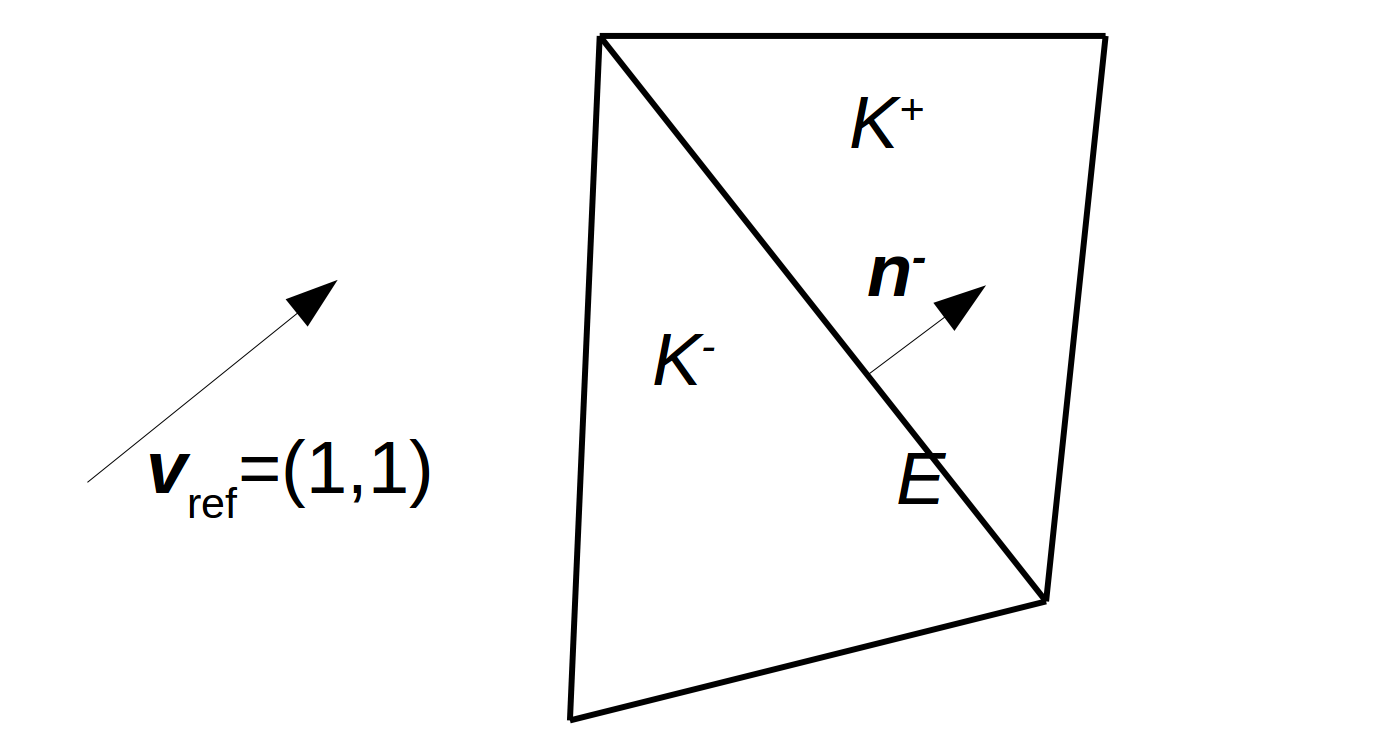}
 \label{fig:E}
\end{figure}
Let $(\phi_h)^\pm =\left.(\phi_h)\right|_{K^\pm}$. 
We use
\begin{align}
 \label{avg-jmp2d}
\jump {\phi_h}|_E  = \phi_h^+- \phi_h^-,\quad \quad
\mean {\phi_h}|_E  = \frac12(\phi_h^++ \phi_h^-)
\end{align}
to denote the jump and the average of $\phi_h$ on the edge $E$.
We shall always take $\bld n^-$ as the normal direction of the edge $E$.

% \subsubsection{Function spaces and norms}
% Denote by $\normk {\cdot}$  the standard $L^2$ norms on $K$.
% 
% For any integer $ l \ge 0$, the standard Sobolev $l$ norm on $K$ is denoted by $\norm \cdot_{H^l(K)}$.
% Furthermore, the norms of the broken Sobolev spaces
% $W^{l,p}(\oh) = \{ v \in L^2 (\Omega) : v|_K \in W^{l,p}(K) \quad \forall K \in \oh \}$ with $p = 2, \infty$
% are given by
% $$
%  \normhdl v = \left( \sumk {\normhkls v}\right)^{\frac12},
%  \quad \norm v_{W^{l,\infty}(\oh)} = \max_{K\in \oh} \norm v_{W^{l,\infty}(K)}.
% $$
% In the case $l = 0$, we denote $\normd v = \norm v_{H^0(\oh)}$.

\subsection{Energy-conserving DG methods for linear symmetric hyperbolic systems with 
paired eigenvalues}
\label{sec:sys2d}
Following the one-dimensional case \eqref{scheme:sys1dX}, the energy-conserving 
semi-discrete DG methods for the linear symmetric hyperbolic systems \eqref{sys2d}
with paired eigenvalues
is given as follows.
Find, for any time $t \in (0, T]$,
the unique function $\bld u_h = \bld u_h(t)
\in [V_h^k]^m$ such that
\begin{align}
 \intk {\B_0(\bld u_h)_t\cdot\bld v_h}  - \intk {\B_1\bld u_h\cdot(\bld  v_h)_x} 
 - \intk {\B_2\bld u_h \cdot(\bld v_h)_y} &\nonumber\\
 + 
 \intdk {\widehat{\Bn\bld u_h}\cdot \bld v_h} & =\;0
                    \label{scheme:2d}
\end{align}
holds for all $\bld v_h \in [V_h^k]^m$ and all $ K \in \oh$. 
% where $\intk {(\cdot)}$ stands for $\intij {(\cdot)}$. 
Here, the numerical flux, which is similar to the one-dimensional case
\eqref{flux:sys1dt}, is given as follows:
\begin{subequations}\label{flux:2d}
\begin{alignat}{2}
\widehat{{\Bnm}{\bld u}_h}|_{E} = &
{\Bnm}\mean{{\bld u}_h}
+ 
\frac12\sum_{\mu=1}^{r_{\bld n}}
\sqrt{|\lambda_{\bld n^-, \mu}^{+}\lambda_{\bld n^-, \mu}^{-}|}\,
{\bld S_{\bld n^-}}{\bld R}_\mu
{\bld S_{\bld n^-}^{-1}}
\jump{{\bld u}_h},\\
\widehat{\Bnp{\bld u}_h}|_{E} = &
-\widehat{\Bnm{\bld u}_h}|_{E},
\end{alignat}
\end{subequations}
where $\bld n^-$ is the direction of the edge $E$ that satisfy \eqref{normal-p}, 
and $\bld n^+=-\bld n^-$ is the direction with opposite sign.
% , and 
% $s_{\mu, \bld n^-}\in \{-1, 1\}$ is  a direction to be determined.

Recall that the matrix $\bld R_\mu\in \mathbb{R}^{m\times m}$ is the 
anti-symmetric matrix that only has non-vanishing components on the 
$(\mu,m-\mu)$ and $(m-\mu, \mu)$ locations, with 
\[
 \bld R_\mu(\mu, m-\mu) = 1,\quad 
  \bld R_\mu(m-\mu,\mu) = -1.
\]
Note that the above choice of numerical flux is consistent and conservative.

\begin{remark}[On the numerical flux]
 Recall that the (dissipative) upwinding numerical flux is given by 
\begin{alignat}{2}
\label{flux:upwind}
\widehat{{\Bnm}{\bld u}_h}|_{E} = &
{\Bnm}\mean{{\bld u}_h}
- 
\frac12 
{\bld S_{\bld n^-}}|\bld \Lambda_{\bld n^-}\!|\,
{\bld S_{\bld n^-}^{-1}}
\jump{{\bld u}_h},
\end{alignat}
where 
\[
|\bld \Lambda_{\bld n^+}| 
= \mathrm{diag}([|\lambda_{\bld n^-,1}^{+}|,\cdots,
   |\lambda_{\bld n^-,r_{\bld n}}^{+}|,0,\cdots,0,
  |\lambda_{\bld n^-,r_{\bld n}}^{-}|,\cdots,|\lambda_{\bld n^-, 1}^{-}|]).
\]
It is obtained by solving the Riemann problem along the normal direction.
The cheaper-to-implement, more dissipative  Lax-Friedrichs flux is given by 
\begin{alignat}{2}
\label{flux:lf}
\widehat{{\Bnm}{\bld u}_h}|_{E} = &
{\Bnm}\mean{{\bld u}_h}
- 
\frac12 \max\{|\lambda_{\bld n^-,1}^{+}|, |\lambda_{\bld n^-,1}^{-}|\}
\jump{{\bld u}_h},
\end{alignat}
and the central flux is given by 
\begin{alignat}{2}
\label{flux:central}
\widehat{{\Bnm}{\bld u}_h}|_{E} = &
{\Bnm}\mean{{\bld u}_h}.
\end{alignat}
The only difference among these numerical fluxes is on the choice of the {\it stabilization} term 
involving $\jump{\bld u_h}$.
\end{remark}

We have energy-conservation of the method \eqref{scheme:2d}, just as 
the one-dimensional case. The proof is identical, and is omitted.
\begin{theorem}
\label{thm:sys2dX}
%  Assume that the exact solution ${\bld u}$ of \eqref{sys2d} is sufficiently smooth. 
Let ${\bld u}_h\in [V_h^k]^{m}$ be the numerical solution of the  
semi-discrete DG scheme \eqref{scheme:2d}.
Then, the total energy 
\[
 {E}(t) =\int_{\Omega}{({\bld B_0}{\bld u}_h)\cdot{\bld u}_h}\mathrm{d\bld x}
\]
is conserved for all time.
% Moreover, for $T >0$ there holds the following
% error estimate
% \begin{equation} \label{thm:sys-est}
% \normI {{\bld u}(T) -{\bld u}_h(T)} 
% \le C (1+T) \ho,
% \end{equation}
% where $C$ is independent of $h$.
\end{theorem}
% \begin{proof}
%  The proof of energy conservation is identical to the one-dimensional case in 
%  Theorem \ref{thm:sys1d}.
% The main ingredient of the proof of optimal convergence 
% is a tensor-product version of the one-dimensional
% local projection \eqref{proj-adv}, and a crucial superconvergence result 
% of this projection on tensor-product meshes; see \cite[Lemma 3.6]{CockburnCartesian01}.
% We leave out the details.
%  \end{proof}

 \begin{remark}[Doubling the unknowns]
 \label{rk:sys2d}
  Any linear symmetric hyperbolic system \eqref{sys2d} can be modified to be a 
  system with paired eigenvalues, essentially following the doubling the unknowns approach in section \ref{sub:ad}.
In particular, we shall consider the following augmented system:
\begin{subequations}
\label{sys:2daux}
\begin{align}
 \B_0\,\bld u_t + \B_1\, \bld u_x + \B_2\,\bld u_{y} & =0, \\
 \B_0\,\bld \phi_t - \B_1\, \bld \phi_x - \B_2\,\bld \phi_{y} & =0,
\end{align}
\end{subequations}
where the auxiliary {\it zero} variable $\bld \phi(x,y,t)$ has a {\it zero} initial condition.
It is easy to observe that the above system is a system with paired eigenvalues.
Taking into account its block anti-symmetric structure, 
the scheme \eqref{scheme:2d} applied to the equations \eqref{sys:2daux}
has the following form.
Find, for any time $t \in (0, T]$,
the unique function $(\bld u_h, \bld \phi_h)\in [V_h^k]^m\times [V_h^k]^m$ such that
\begin{subequations}
\label{scheme:2d-aux}
\begin{align}
 \intk {\B_0(\bld u_h)_t\cdot\bld v_h}  - \intk {\B_1\bld u_h\cdot(\bld  v_h)_x} 
 - \intk {\B_2\bld u_h \cdot(\bld v_h)_y} &\nonumber\\
 + 
 \intdk {\widehat{\Bn\bld u_h}\cdot \bld v_h} & =\;0,\\
 \intk {\B_0(\bld \phi_h)_t\cdot\bld \psi_h}  + \intk {\B_1\bld \phi_h\cdot(\bld  \psi_h)_x} 
 + \intk {\B_2\bld \phi_h \cdot(\bld \psi_h)_y} &\nonumber\\
 - 
 \intdk {\widehat{\Bn\bld \phi_h}\cdot \bld \psi_h} & =\;0,
\end{align}
holds for all $(\bld v_h, \bld \psi_h) \in [V_h^k]^m\times [V_h^k]^m$ and all $ K \in \oh$, 
with 
the numerical fluxes given as follows:
\begin{alignat}{2}
\widehat{{\Bnm}{\bld u}_h}|_{E} = &
{\Bnm}\mean{{\bld u}_h}
+
\frac12{|\Bnm\!|}\jump{{\bld \phi}_h}, \\
\widehat{{\Bnm}{\bld \phi}_h}|_{E} = &{\Bnm}\mean{{\bld \phi}_h}
+
\frac12{|\Bnm\!|}\jump{{\bld u}_h},\\
\widehat{\Bnp{\bld u}_h}|_{E} = &
-\widehat{\Bnm{\bld u}_h}|_{E},\quad
\widehat{\Bnp{\bld \phi}_h}|_{E} = 
-\widehat{\Bnm{\bld \phi}_h}|_{E},
\end{alignat}
\end{subequations}
where $|\bld B_{\bld n^-}\!|  = \bld S_{\bld n^-}|\bld \Lambda_{\bld n^-}|\bld S_{\bld n^{-}}^{-1}$.
It is interesting to see the similarity of this numerical flux with the upwinding flux \eqref{flux:upwind}.
Unlike the upwinding case, the jump term in the above numerical flux do not contribute 
to dissipation, but to the {\it coupling} of the primal variables $\bld u_h$ and 
the auxiliary variables $\bld \phi_h$.
Note also that the above numerical flux 
is different from the one dimensional case in Remark \ref{rk:sys} as we need the 
eigenvalue decomposition of $\bld B_{\bld n^{-}}$ for the numerical flux.
We numerically observed that on triangular meshes, such eigenvalue decomposition is crucial for the 
method to be optimally convergent.

Finally we point out that doubling the unknowns essentially leads to a doubling of the computational 
cost when explicit time-stepping schemes, see section \eqref{sec:time}, are used.
 \end{remark}

\begin{remark}[Error estimates]
\label{rk-error}
The error analysis of the method \eqref{scheme:2d} is more involved than
the 1D case. 
Suboptimal convergence order of $k$ can be proven using a standard $L^2$-projection on general mesh.
Optimal convergence order of $k+1$ for all the variables $\bld u_h$ 
on rectangular meshes
% , 
% with $s_{\mu,\bld n} = -1$ 
% for all $\mu$,  or  $s_{\mu,\bld n} = 1$ for all $\mu$ in the numerical flux \eqref{flux:2d}, 
can be proven 
by using the  superconvergence result of Lesaint and Raviart \cite{LasaintRaviart74} 
of the tensor-product Gauss-Radau projection, see also \cite[Lemma 3.6]{CockburnCartesian01}.
However, the method is numerically observed to be suboptimal for certain hyperbolic systems 
on general triangular meshes  including acoustics with zero background velocity (Example 4.10 in section \ref{sec:num}) 
and elastodynamics (Example 4.11 in section \ref{sec:num}). 
It was also numerically observed in  \cite{LiShiShu18} to be suboptimal for the 
DG method with an alternating numerical flux for the time domain Maxwell's equation on triangular meshes, 
which is equivalent to the method \eqref{scheme:2d} directly applied to the Maxwell's equations (see section \ref{sec:em2d}).
We note that the aforementioned equations are by themselves systems 
with paired eigenvalues.
% On the other hand, the method \eqref{scheme:2d} on triangular meshes is numerically observed to be optimal
% for the advection equation  (Example 4.8 in section \ref{sec:num}) 
% and acoustic equation with non-zero background velocity (Example 4.9 in section \ref{sec:num}).
On the other hand, the doubling unknowns approach \eqref{scheme:2d-aux} 
in Remark \ref{rk:sys2d} applied to the augmented system 
is numerically observed to be optimally convergent for the aforementioned equations.
Of course, we have also doubled the computational cost.
Further study needs to be conducted to understand the convergence behavior of this method on triangular meshes.
\end{remark}

 \begin{remark}[Time discretization, source term, and boundary conditions]
 The same high-order energy-conserving Lax-Wendroff time discretization \eqref{fully-discrete} can be used for 
 \eqref{scheme:2d} to get a fully discrete energy-conserving DG method.
 We can also use the Runge-Kutta type Lax-Wendroff time discretization \eqref{rk}.
Source terms and boundary conditions can be easily incorporated into the scheme \eqref{scheme:2d}.
% We can also impose upwinding boundary conditions for 
% boundary value problems with inflow boundary condition.
We refer details to the discussion in section \ref{sec:time} and section \ref{sec:bdry}.
 \end{remark}

 \subsection{Practical examples}
 Now, we consider the application of Theorem \ref{thm:sys2dX} and Remark \ref{rk:sys2d} for a large class of symmetric 
 linear hyperbolic system of equations.
 \subsubsection{Advection}
 \label{sec:adv2d}
We consider the advection equation 
\begin{align}
  u_t + b_0 u_x + b_1 u_y = 0,
 \end{align}
 with $b_0^2+b_1^2\not=0$.
 Following Remark \ref{rk:sys2d},
we convert it to a system with paired eigenvalues by introducing the auxiliary {\it zero}
 function $\phi(x,y,t)$ that solve the equation
 \[
  \phi_t - b_0 \phi_x - b_1 \phi_y = 0.
 \]
%  Taking $s_{\mu, \bld n} = \mathrm{sign}(b_n)$, 
The energy-conserving numerical flux \eqref{flux:2d} 
for the
resulting system 
on the edge $E$ with normal direction $\bld n = (n_x, n_y)$
is given by 
\begin{align}
 \label{flux-adv2d}
 \widehat{{\Bn}{\bld u}_h}|_{E} = &
 \left[
 \begin{tabular}{l}
  $b_n \mean{u_h}
$\\[.2cm]
  $-b_n\mean{\phi_h}$
 \end{tabular}
 \right]+\frac12
 \left[
 \begin{tabular}{l}
  $
 |b_n|
 \jump{\phi_h}$\\[.2cm]
  $-
  |b_n|\jump{u_h}$
 \end{tabular}
 \right],
\end{align}
where $b_n = b_0n_x + b_1n_y$ is the normal velocity.
We mention in particular that the above numerical flux is 
independent of the artificial direction $\bld v_{\mathrm{ref}}=(1,1)$ 
used to determine the unique direction of the edge $E$ in \eqref{normal-p}.

\subsubsection{Acoustics}
 \label{sec:ac2d}
We consider the acoustics equations
\begin{align}
 \left[
 \begin{tabular}{c}
  $p$\\
  $u$\\
  $v$
 \end{tabular}
 \right]_t
 +
 \left[
 \begin{tabular}{ccc}
  $u_0$ &  $K_0$ & $0$\\
  $1/\rho_0$ & $u_0$ & $0$\\
  $0$ & $0$ & $u_0$
 \end{tabular}
 \right]
  \left[
 \begin{tabular}{c}
  $p$\\
  $u$\\
  $v$
 \end{tabular}
 \right]_x
+ 
\left[
 \begin{tabular}{ccc}
  $v_0$ & $0$ & $K_0$\\
  $0$ & $v_0$ & $0$\\
  $1/\rho_0$ & $0$ & $v_0$
 \end{tabular}
 \right]
\left[
 \begin{tabular}{c}
  $p$\\
  $u$\\
  $v$
 \end{tabular}
 \right]_y=0,
 \end{align}
where $p(x,y,t)$ is the pressure, and $\vec{u} = (u(x,y,t), v(x,y,t))$ is the velocity vector,
and for the constants, $\vec{u_0}=(u_0,v_0)$ is the velocity 
for a background flow, $K_0>0$ is the bulk modulus of compressibility 
and $\rho_0>0$ is the density.

Similar to the one-dimensional case in Section \ref{sub:ac}, 
the system can be symmetrized to the following form:
 \begin{align}
 \label{ac2d}
\underbrace{\left[
 \begin{tabular}{ccc}
  $1/K_0$ &  $0$ & $0$\\
  $0$&$\rho_0$ & $0$\\
  $0$ & $0$ & $\rho_0$
 \end{tabular}
 \right]}_{:=\B_0}
\left[
 \begin{tabular}{c}
  $p$\\
  $u$\\
  $v$
 \end{tabular}
 \right]_t
 +
\underbrace{\left[
 \begin{tabular}{ccc}
  $u_0/K_0$ &  $1$ & $0$\\
  $1$ & $u_0\rho_0$ & $0$\\
  $0$ & $0$ & $u_0\rho_0$
 \end{tabular}
 \right]}_{:=\B_1}
  \left[
 \begin{tabular}{c}
  $p$\\
  $u$\\
  $v$
 \end{tabular}
 \right]_x&\nonumber\\
+ 
\underbrace{\left[
 \begin{tabular}{ccc}
  $v_0/K_0$ & $0$ & $1$\\
  $0$ & $v_0\rho_0$ & $0$\\
  $1$ & $0$ & $v_0\rho_0$
 \end{tabular}
 \right]}_{:=\B_2}
\left[
 \begin{tabular}{c}
  $p$\\
  $u$\\
  $v$
 \end{tabular}
 \right]_y&=0,
 \end{align}
 
The doubling unknowns approach in Remark \ref{rk:sys2d}
shall be used for the system \eqref{ac2d} on general triangular meshes. 
However, we can save the computational cost by looking into the eigenvalue structure of the 
matrix $\Bn = n_x\bld B_1+n_y\bld B_2$. 
The following discussion is similar to the 
one dimensional case in section \ref{sub:ac}.
We mention that the following simplification shall be done on 
Cartesian meshes, as we numerically observe suboptimal convergence of this simplified 
method on general triangular meshes.

The matrix $\Bn$ is given by 
 \begin{align*}
\Bn:=  \left[
 \begin{tabular}{ccc}
  $v_n/K_0$ & $n_x$ & $n_y$\\
  $n_x$ & $v_n\rho_0$ & $0$\\
  $n_y$ & $0$ & $v_n\rho_0$
 \end{tabular}
 \right],
 \end{align*}
where $v_n = u_0n_x+v_0n_y$. 
It has an eigenvalue $\lambda_1 = v_n\rho_0$, and a pair of eigenvalues $\lambda_{2}^\pm$
that are the two roots of the following quadratic equation 
\[
 \lambda^2 - (v_n/K_0+v_n/K_0)\lambda + v_n^2/c_0^2-1 = 0
\]
where $c_0 = \sqrt{K_0/\rho_0}$ is the speed of sound.
We shall distinguish with the following three cases.
\subsection*{Zero background velocity ($\vec{u_0} = 0$)}
In this case ($u_0=v_0=0$), the three eigenvalues of $\Bn$ are $0, \pm1$.
The system \eqref{ac2d} by itself
is a linear symmetric hyperbolic system with paired eigenvalues. 
% We do not need to 
% modify the equations.
We can direct apply the method \eqref{scheme:2d} to the equations \eqref{ac2d}.
The numerical flux of the method  
is noting but the alternating numerical flux considered in 
\cite{XingChouShu14}:
\begin{align}
\label{ac-flux0}
\widehat{{\Bn}{\bld u}_h}|_{E} = 
\left[
 \begin{tabular}{c}
  $u_h^+n_x+v_h^+n_y$\\
  $p_h^-n_x$\\
  $p_h^-n_y$
 \end{tabular}
 \right].
\end{align}
This method is numerically observed to be suboptimal on general triangular meshes.
% where $\bld n = (n_x,n_y)$ is the normal direction of the edge $E$.
\subsection*{Subsonic case ($0<u_0^2+v_0^2<c_0^2$)}
In this case, the magnitude of the 
normal velocity $v_n=u_0n_x+v_0n_y$ on any edge $E$ 
is less than $c_0$, and in general is not equal to zero. 
In this case, the matrix $\Bn$ either has 2 positive eigenvalues and 1 negative eigenvalue ($v_n>0$)
or has 1 positive eigenvalues and 2 negative eigenvalue ($v_n<0$).
We can convert the system \eqref{ac2d} to a 4-component linear symmetric system with
paired eigenvalues by introducing the {\it zero} function 
$\phi(x,y,t)$ that solves 
\[
\rho_0 \phi_t - \rho_0u_0 \phi_x-\rho_0v_0\phi_y = 0.
\]
The resulting system reads

\begin{align}
 \label{ac2dX}
\underbrace{\left[
 \begin{tabular}{cccc}
  $1/K_0$ &  $0$ & $0$&$0$\\
  $0$&$\rho_0$ & $0$&$0$\\
  $0$ & $0$ & $\rho_0$ &$0$\\
  $0$ & $0$ & $0$ &$\rho_0$
 \end{tabular}
 \right]}_{:=\widetilde{\B_0}}
\left[
 \begin{tabular}{c}
  $p$\\
  $u$\\
  $v$\\
  $\phi$
 \end{tabular}
 \right]_t
 +
\underbrace{\left[
 \begin{tabular}{cccc}
  $u_0/K_0$ &  $1$ & $0$& $0$\\
  $1$ & $u_0\rho_0$ & $0$& $0$\\
  $0$ & $0$ & $u_0\rho_0$& $0$\\
  $0$ & $0$ & $0$&$-u_0\rho_0$\\
 \end{tabular}
 \right]}_{:=\widetilde{\B_1}}
  \left[
 \begin{tabular}{c}
  $p$\\
  $u$\\
  $v$\\
  $\phi$
 \end{tabular}
 \right]_x&\nonumber\\
+ 
\underbrace{\left[
 \begin{tabular}{cccc}
  $v_0/K_0$ & $0$ & $1$& $0$\\
  $0$ & $v_0\rho_0$ & $0$& $0$\\
  $1$ & $0$ & $v_0\rho_0$& $0$\\
  $0$ & $0$ & $0$& $-v_0\rho_0$\\
 \end{tabular}
 \right]}_{:=\widetilde{\B_2}}
\left[
 \begin{tabular}{c}
  $p$\\
  $u$\\
  $v$\\
  $\phi$
 \end{tabular}
 \right]_y&=0,
 \end{align}
 It is easy to verify that the matrix $\widetilde{\Bn}=
 \widetilde{\B_1}n_x+\widetilde{\B_2}n_y
 $ always has paired eigenvalues for the above 
 system for any normal direction $\bld n$.
 
Denoting the 4-component vector $\widetilde{\bld u}_h=[p_h, u_h, v_h,\phi_h]'$, 
the numerical flux on the edge $E$ 
for the system \eqref{ac2dX} reads
\begin{align}
 \label{ac-flux-sub}
 \widehat{\widetilde{\Bn}\widetilde{\bld u}_h}|_{E} = 
\left[
 \begin{tabular}{c}
  $v_n/K_0\mean {p_h}+\mean{\!u_h\!}n_x+
 \mean{\!v_h\!}n_y $\\[.2ex]
  $\mean {p_h}n_x
  +v_n\rho_0\mean{u_h}
  $\\[.2ex]
  $\mean {p_h}n_y
  +v_n\rho_0(\mean{v_h}$\\[.2ex]
  $-v_n\rho_0\mean{\phi_h}$\\
 \end{tabular}
 \right]
 +
\frac12 \left[
 \begin{tabular}{c}
  $
% {\sqrt{1-v_n^2/c_0^2}}
\alpha_n
\,(\jump{u_h}n_x + \jump{v_h}n_y) $\\[.2ex]
  $-
  \alpha_n
\jump{p_h}n_x
  -\beta_n\jump{\phi_h}n_y
  $\\[.2ex]
  $-\alpha_n\,\jump{p_h}n_y
  +\beta_n\jump{\phi_h}n_x$\\[.2ex]
  $\beta_n(\jump{u_h}n_y-\jump{v_h}n_x)$\\
 \end{tabular}
 \right],
\end{align}
where $\alpha_n = \sqrt{1-v_n^2/c_0^2}$ and 
$\beta_n = |v_n|$.
Different from the zero background velocity case, 
this method is numerically observed to be optimal on general triangular meshes.

\subsection*{Supersonic case $u_0^2+v_0^2\ge c_0^2$}
In this case, there exists direction $\bld n$ such that
all eigenvalues of the matrix $\Bn$ are positive.
We shall use the doubling the unknowns approach in Remark \ref{rk:sys2d}, and consider the
augmented 6-component system.
% 
% We shall double the equations to obtain a symmetric 
% system with paired eigenvalues, by introducing 
% the 3-component {\it zero} function $\bld \phi(x,y,t)$ that satisfy 
% \[
%  \B_0 \bld \phi_t - \B_1 \bld \phi_x - \B_2 \bld \phi_y = 0.
% \]
% It is easy to verify that the $6\times 6$ matrix $\widetilde{\Bn}$ for the resulting system 
% always have paired eigenvalues.
% The numerical fluxes is similar to the scalar case, and reads
% \begin{subequations}
% \label{flux2d-sup}
% \begin{align}
%   \widehat{{\Bn}{\bld u}_h}|_{E} = &\;
%  {{\Bn}\mean{\bld u_h}+
%  \frac12
% |\Bn|
%  \jump{\bld\phi_h}},\\
%   \widehat{{\Bn}{\bld \phi}_h}|_{E} = &\;
%  {{\Bn}\mean{\bld \phi_h}+\frac12 |\Bn|\jump{\bld u_h}},
% \end{align}
% \end{subequations}
% where $|\Bn|=\bld S_{\bld n}|\bld \Lambda_{\bld n}|\bld S_{\bld n}^{-1}$.
% We note that on Cartesian meshes, we can simply use the following cheaper component-wise 
% numerical flux, similar to the 1D case in Remark \ref{rk:sys}:
% \begin{align*}
%   \widehat{{\Bn}{\bld u}_h}|_{E} = &\;
%  {{\Bn}(\mean{\bld u_h}+
%  \frac12 \jump{\bld\phi_h}}),\\
%   \widehat{{\Bn}{\bld \phi}_h}|_{E} = &\;
%  {{\Bn}(\mean{\bld \phi_h}+\frac12 \jump{\bld u_h}}).
% \end{align*}
% \begin{remark}[Convergence on triangle meshes]
% Numerical results, see Example 4.8/4.9 in section \ref{sec:num}, show that 
% for the above method 
% the acoustic equation with zero background velocity, the method with 
% the alternting flux \eqref{ac-flux0} loss {\it one} 
% convergence order in the velocity variables $(u_h, v_h)$.
%  
% \end{remark}

\subsubsection{Linearized Euler equations}
 \label{sec:le2d}
We consider the linearized Euler equations in dimensionless form
\begin{align}
\label{lee}
 \left[
 \begin{tabular}{c}
  $\rho$\\
  $u$\\
  $v$\\
  $p$
 \end{tabular}
 \right]_t
 +
 \left[
 \begin{tabular}{cccc}
  $M_x$ &  $1$ & $0$ & $0$\\
  $0$ & $M_x$ & $0$& $1$\\
  $0$ & $0$ & $M_x$& $0$\\
  $0$ & $1$ & $0$& $M_x$\\
  \end{tabular}
 \right]
  \left[
 \begin{tabular}{c}
  $\rho$\\
  $u$\\
  $v$\\
  $p$ 
 \end{tabular}
 \right]_x
+ 
\left[
 \begin{tabular}{cccc}
  $M_y$ &  $0$ & $1$ & $0$\\
  $0$ & $M_y$ & $0$& $0$\\
  $0$ & $0$ & $M_y$& $1$\\
  $0$ & $0$ & $1$& $M_y$\\
  \end{tabular}
 \right]
\left[
 \begin{tabular}{c}
   $\rho$\\
  $u$\\
  $v$\\
  $p$ 
 \end{tabular}
 \right]_y=0,
 \end{align}
 where $M_x, M_y>=0$ with are the constant mean flow Mach number in the 
$x$- and $y$-direction, respectively.
 Subtracting the first equation by the fourth, one obtain the following 
linear symmetric system for the unknown vector 
$\bld u:=[\rho-p, u, v, p]'$:
\begin{align}
\label{lee-2}
 \left[
 \begin{tabular}{c}
  $\rho-p$\\
  $u$\\
  $v$\\
  $p$
 \end{tabular}
 \right]_t
 +
 \underbrace{\left[
 \begin{tabular}{cccc}
  $M_x$ &  $0$ & $0$ & $0$\\
  $0$ & $M_x$ & $0$& $1$\\
  $0$ & $0$ & $M_x$& $0$\\
  $0$ & $1$ & $0$& $M_x$\\
  \end{tabular}
 \right]}_{:=\B_1}
  \left[
 \begin{tabular}{c}
  $\rho-p$\\
  $u$\\
  $v$\\
  $p$ 
 \end{tabular}
 \right]_x\nonumber\\
+ 
\underbrace{\left[
 \begin{tabular}{cccc}
  $M_y$ &  $0$ & $0$ & $0$\\
  $0$ & $M_y$ & $0$& $0$\\
  $0$ & $0$ & $M_y$& $1$\\
  $0$ & $0$ & $1$& $M_y$\\
  \end{tabular}
 \right]}_{:=\B_2}
\left[
 \begin{tabular}{c}
   $\rho-p$\\
  $u$\\
  $v$\\
  $p$ 
 \end{tabular}
 \right]_y=0.
 \end{align}
 Note that the above equation is simply the combination of 
 the acoustic equations for $[u,v,p]$, and the advection equation for $\rho-p$.
 We can just follow the discussion on the previous two subsections to obtain the energy-conserving method.
 We leave out the details.
%  In general , these components talk to each other via proper boundary conditions.
%  But in our periodic setting, they are evolving independently.
%  
%  Similarly to the acoustics case, we shall consider the following three cases.
%  \subsection*{Zero background velocity ($M_x=M_y = 0$)}
%  This reduces to the acoustics equations for $[u,v,p]$ in the non-dimensional form (with zero
%  background mean flow). We do not need to modify the equations. 
% 
% \subsection*{Subsonic case ($0<M_x^2+M_y^2<1$)}
% In this case,  
% % $\Bn$ 
% % has 3 positive eigenvalues, and one negative eigenvalue.
% We shall convert the system \eqref{lee-2} to a 6-component linear symmetric system with
% paired eigenvalues by introducing the {\it zero} functions
% $\phi(x,y,t)$ and $\psi(x,y,t)$ that solves 
% \[
%  \phi_t - M_x \phi_x-M_y\phi_y = 0, \quad 
% \psi_t - M_x \psi_x-M_y\psi_y = 0,
% \]
% with $\phi$ pairing for the advection equation $\rho-p$, and 
% $\psi$ pairing for the acoustic equation for $[u,v,p]$. We leave out detailed formulation
% of the corresponding numerical flux.
% 
% \subsection*{Supersonic case $M_x^2+M_y^2>=1$}
% % In this case all eigenvalues of $\Bn$ are positive.
% In this case, we double the equations to obtain an 8-component symmetric system.

\subsubsection{Electromagnetism}
 \label{sec:em2d}
We consider the two-dimensional time-domain Maxwell equations 
in transverse magnetic form (TM) in a heterogeneous media
\begin{align}
\label{elec}
 \left[
 \begin{tabular}{c}
  $\mu H^x$\\
  $\mu H^y$\\
  $\epsilon E^z$
 \end{tabular}
 \right]_t
 +
 \left[
 \begin{tabular}{ccc}
  $0$ &  $0$ & $0$\\
  $0$ &  $0$ & $-1$\\
  $0$ & $-1$ & $0$
 \end{tabular}
 \right]
  \left[
 \begin{tabular}{c}
  $H^x$\\
  $H^y$\\
  $E^z$
 \end{tabular}
 \right]_x
+ 
 \left[
 \begin{tabular}{ccc}
  $0$ &  $0$ & $1$\\
  $0$ &  $0$ & $0$\\
  $1$ & $0$ & $0$
 \end{tabular}
 \right]
  \left[
 \begin{tabular}{c}
  $H^x$\\
  $H^y$\\
  $E^z$
 \end{tabular}
 \right]_y
=0,
 \end{align}
where $(H^x(x,y,t), H^y(x,y,t))$ is the magnetic fields, and $E^z(x,y,t)$ is the electric field,
and $\mu, \epsilon:\Omega\rightarrow \mathbb{R}$ is the magnetic permeability, and electric permittivity,
respectively.

The system is similar to the acoustic case with {\it zero} background mean flow.
On general triangular meshes, we propose to double the unknowns and 
obtain the DG method \eqref{scheme:2d-aux} for the 6-components augmented system. 
On Cartesian meshes, we can also directly apply the method \eqref{scheme:2d} 
to \eqref{elec} to obtain the DG method with 
an an alternating numerical flux. 
The alternating flux DG method, which is 
optimal on Cartesian meshes but
suboptimal on triangular meshes,  was discussed in details recently 
\cite{LiShiShu17,LiShiShu18}, where 
special focus was made on Maxwell's equations in Drude metamaterials.

\subsubsection{Elastodynamics}
 \label{sec:es2d}
We consider the elastodynamics equations in a heterogeneous, isotropic media,
written in stress-velocity form
\begin{align}
 \label{elas}
 \left[
 \begin{tabular}{c}
  $\sigma_{xx}$\\
  $\sigma_{yy}$\\
  $\sigma_{xy}$\\
  $v$\\
  $w$
 \end{tabular}
 \right]_t
 +
 \left[
 \begin{tabular}{ccccc}
  $0$ & $0$ &$0$ & $-(\lambda+2\mu)$ & $0$\\
  $0$ & $0$ &$0$ & $-\lambda$ & $0$\\
  $0$ & $0$ &$0$ & $0$ & $-\mu$\\
  $-1/\rho$ & $0$ &$0$ & $0$ & $0$\\
  $0$ & $0$ &$-1/\rho$ & $0$ & $0$\\
  \end{tabular}
 \right]
 \left[
 \begin{tabular}{c}
  $\sigma_{xx}$\\
  $\sigma_{yy}$\\
  $\sigma_{xy}$\\
  $v$\\
  $w$
 \end{tabular}
 \right]_x\nonumber\\[.2ex]
+ 
\left[
 \begin{tabular}{ccccc}
  $0$ & $0$ &$0$ & $0$ & $-\lambda$\\
  $0$ & $0$ &$0$ & $0$ & $-(\lambda+2\mu)$\\
  $0$ & $0$ &$0$ & $-\mu$ & $0$\\
  $0$ & $0$ & $-1/\rho$ &$0$ & $0$\\
  $0$  &$-1/\rho$ & $0$& $0$ & $0$\\
  \end{tabular}
 \right]\left[
 \begin{tabular}{c}
  $\sigma_{xx}$\\
  $\sigma_{yy}$\\
  $\sigma_{xy}$\\
  $v$\\
  $w$
 \end{tabular}
 \right]_y&=0,
\end{align}
where $\bld \sigma = \left[\begin{tabular}{cc}
                            $\sigma_{xx}$& $\sigma_{xy}$\\
                            $\sigma_{xy}$& $\sigma_{yy}$                            
                           \end{tabular}
\right]$ is the stress field, and 
$\bld v = [v,w]'$ is the velocity field, and 
$\lambda$ and $\mu$ are the Lam\'e constants, and $\rho$ is the density functions.
We assume $\lambda,\mu,\rho$ are piecewise constants with discontinuity aligned with the 
mesh.
The system \eqref{elas} can be transformed into a symmetric hyperbolic 
system by left multiplying 
the equation with the symmetric positive definite matrix 
\[
 \bld B_0 : = \left[
 \begin{tabular}{c c c c c}
  $\frac{\lambda+2\mu}{4\mu(\mu+\lambda)}$ &
  $\frac{-\lambda}{4\mu(\mu+\lambda)}$ &
  $0$&$0$&$0$\\[.2ex]
  $\frac{-\lambda}{4\mu(\mu+\lambda)}$ &$\frac{\lambda+2\mu}{4\mu(\mu+\lambda)}$ &
  $0$&$0$&$0$\\
  $0$ &
  $0$ &
  $\frac{1}{\mu}$&$0$&$0$\\
  $0$ &
  $0$ &
  $0$&$\rho$&$0$\\
  $0$ &
  $0$ &
  $0$&$0$&$\rho$
 \end{tabular}
 \right].
\]
Denoting the 5-component vector $\bld u = [\sigma_{xx}, \sigma_{yy}, \sigma_{xy}, v, w]'$, 
and the matrices
\begin{align*}
 \bld B_1 = 
 \left[
 \begin{tabular}{ccccc}
  $0$ & $0$ &$0$ & $-1$ & $0$\\
  $0$ & $0$ &$0$ & $0$ & $0$\\
  $0$ & $0$ &$0$ & $0$ & $-1$\\
  $-1$ & $0$ & $0$ &$0$ & $0$\\
  $0$  &$0$ & $-1$& $0$ & $0$\\
  \end{tabular}
 \right],\quad 
 \bld B_2 = 
 \left[
 \begin{tabular}{ccccc}
  $0$ & $0$ &$0$ & $0$ & $0$\\
  $0$ & $0$ &$0$ & $0$ & $-1$\\
  $0$ & $0$ &$0$ & $-1$ & $0$\\
  $0$ & $0$ & $-1$ &$0$ & $0$\\
  $0$  &$-1$ & $0$& $0$ & $0$\\
  \end{tabular}
 \right]
\end{align*}

We have 
\begin{align}
\label{es:sym}
 \bld B_0 \bld u_t + \bld B_1\bld u_x + \bld B_2 \bld u_y = 0. 
\end{align}
The matrix $\Bn = \bld B_1 n_x + \bld B_2 n_y$
has two positive eigenvalues $\sqrt{1\pm n_x n_y}$, and two negative eigenvalues
$-\sqrt{1\pm n_x n_y}$, and a zero eigenvalue. Hence, the system \eqref{elas} is already a symmetric system with paired eigenvalues.

Similar to the electromagnetism case in section \ref{sec:em2d}, 
we propose to double the unknowns and 
obtain the DG method \eqref{scheme:2d-aux} for the 10-component augmented system
on general triangular meshes. 
On Cartesian meshes, we directly apply the method \eqref{scheme:2d} to the equations 
\eqref{es:sym} to obtain the DG method with 
% the following alternating numerical flux: 
% The resulting numerical flux
% % , written in terms of the primative variable $\bld u$,
% is the following alternating numerical flux:
\begin{align*}
\widehat{\Bn\bld u_h}|_E = 
\left[
\begin{tabular}{c}
 $-n_x v_h^+$\\[.2cm]
 $-n_y w_h^+$\\[.2cm]
 $-n_y v_h^+ - n_x w_h^+$\\[.2cm]
 $-n_x \sigma_{xx,h}^--n_y \sigma_{xy,h}^-$\\[.2cm]
 $-n_x \sigma_{xy,h}^--n_y \sigma_{yy,h}^-$
\end{tabular}
\right].
%  \widehat{\sigma}_{xx,h}|_{j-1/2} = \sigma_{xx,h}^+,\quad
%  \widehat{\sigma}_{yy,h}|_{j-1/2} = \sigma_{yy,h}^+,\quad&
%  \widehat{\sigma}_{xy,h}|_{j-1/2} = \sigma_{xy,h}^+,\\
%  \widehat{\sigma}_{xx,h}|_{i-1/2} = \sigma_{xx,h}^+,\quad
%  \widehat{\sigma}_{yy,h}|_{i-1/2} = \sigma_{yy,h}^+,\quad&
%  \widehat{\sigma}_{xy,h}|_{i-1/2} = \sigma_{xy,h}^+,\\
%  \widehat v_h|_{j-1/2} = v_h^-,\quad &
%  \widehat w_h|_{j-1/2} = w_h^-,\\
%  \widehat v_h|_{i-1/2} = v_h^-,\quad &
%  \widehat w_h|_{i-1/2} = w_h^-.
\end{align*}
% Note that the resulting semi-discrete DG scheme preserves the energy 
% \[
%  E(t) = \int_{\Omega} \B_0\bld u_h\cdot\bld u_h\,\mathrm{dxdy}
%  =
%  \int_{\Omega} \mathcal{A}\bld\sigma_h\cdot\bld\sigma_h
%  +\rho\bld v\cdot\bld v\,\mathrm{dxdy},
% \]
% where $\mathcal{A}$ is the following compliance tensor
% \[
%  \mathcal{A}\bld\sigma := \frac{1}{2\mu}(\bld \sigma-\frac{\lambda}{2\mu+2\lambda}
%  (\sigma_{xx}+\sigma_{yy})\bld I_2).
% \]

For an efficient time integration,
we prefer to work directly with the (equivalent) original stress-velocity form \eqref{elas}. 
This leads to a diagonal mass matrix for the whole system if the orthogonal basis is used. 

% \begin{remark}[Alternating numerical flux in 3D]
% Due to the two-way wave nature for the second-order wave equations,
% the above results on acoustics with {\it zero} background mean flow velocity, 
%  electromagnetism and elastodynamics extends directly to
%  three-dimensions using an alternating numerical flux.
% \end{remark}

\section{Numerical results}
\label{sec:num}

We present extensive numerical results to 
assess the performance of the proposed energy-conserving DG method.
We also compare results with the 
(dissipative) upwinding DG methods, and the (energy-conserving) DG methods with a central flux.
All numerical simulation are performed using the open-source finite-element software {\sf NGSolve} 
\cite{Schoberl16}, \url{https://ngsolve.org/}.

For all the accuracy tests, we restrict ourselves to the spatial error, and 
take the 6-stage 6th order ($r=6$) Lax-Wendroff time stepping \eqref{rk} 
with a small enough time step size so that the temporal error can be neglected.
% and take 
% $\Delta t = C\!F\!L\,h$ with a small enough CFL number $C\!F\!L$.

\subsection*{Example 4.1: 1D advection with periodic boundary condition}
We consider the following advection equation
\begin{align}
\label{eq1}
 u_t + u_x = 0
\end{align}
on a unit interval $I=[0,1]$ with 
initial condition $u(x,0) = \sin(2\pi x)$, and a periodic boundary condition.
The exact solution is 
\begin{align*}
u(x, t) = \sin(2\pi (x-t)).
\end{align*}

We present numerical results with the following three DG methods:
\begin{itemize}
 \item [(U)] the DG method for \eqref{eq1} with an upwinding numerical flux.
%  , 
%  Temporal approximation: $k+1$-th stage (Runge-Kutta type) Lax-Wendroff time stepping \eqref{rk}.
 \item [(C)]the DG method for \eqref{eq1} with a central numerical flux.
%  ,
%  Temporal approximation: $k+1$-th stage energy-conserving Lax-Wendroff time stepping \eqref{fully-discrete}.
 \item [(A)] the DG method \eqref{scheme:adv1d}
 with numerical flux \eqref{flux:opt}  for the augmented
 system \eqref{aux-adv}. 
\end{itemize}
% In order to reduce time error, we take the 
% we take the 6-stage 6th order Lax-Wendroff time stepping \eqref{rk} with 
% $r = 6$.
Table \ref{table:adv1d} lists the numerical errors and their orders for 
the above three DG methods at $T=0.5$.
We use $P^k$ polynomials with $0\le k\le 4$ on a nonuniform mesh which is a $10\%$ random perturbation of the 
uniform mesh.
% and take the time step size to 
% be $\Delta t = C\!F\!L\,h$ with  $C\!F\!L = 0.5$ for $P^0$, 
% $C\!F\!L = 0.2$ for $P^1$, $C\!F\!L = 0.1$ for $P^2$, $C\!F\!L = 0.06$ for $P^3$, $C\!F\!L = 0.04$ for $P^4$.

From the table we conclude that, 
one can always observe optimal $(k+1)$th order of accuracy for both the variable $u_h$ (which approximate the solution $u$)
and $\phi_h$ (which approximate the {\it zero} function)
for the new energy-conserving DG method \eqref{scheme:adv1d}. This validates 
our convergence result in Theorem \ref{thm:adv1d:err}.
Moreover, the absolute value of the error is slightly smaller than
the optimal-convergent upwinding DG method for all polynomial degrees.
We also observe suboptimal convergence for the (energy-conserving) DG method with a central flux for all polynomial degree.
We specifically point out that while 
optimal convergence for the central DG method
has been proven for even polynomial degrees on uniform meshes \cite{CockburnShu98},
Table \ref{table:adv1d}  shows that such optimality no longer holds on nonuniform meshes, regardless of the polynomial degree.

\begin{table}[htbp]
\caption{\label{table:adv1d} 
The $L^2$-errors
% $\|u-u_h\|_I$ (and $\|\phi_h\|$) 
% at time $T=0.5$ 
% $e_u(T) = \|u-u_h\|_{L^2(I)}$ and $e_\phi=\|\phi_h\|_{L^2(I)}$) 
and orders for Example 4.1 for the upwinding DG method (U), the central DG method (C), and the 
new DG method (A) on a random mesh of $N$ cells. 
$T =0.5$. 
} \centering
\bigskip
\begin{tabular}{|c|c|cc|cc|cc|cc|}
\hline
 & &\multicolumn{2}{|c|}{(U)}&\multicolumn{2}{c|}{(C)}&\multicolumn{4}{c|}{(A)}\\
 \hline
%  \cline{3-4}   \cline{6-7} \cline{9-12}
 &  {$N$}      & $\|u-u_h\|$ & Order & $\|u-u_h\|$ & Order &  $\|u-u_h\|$ & Order 
 &  $\|\phi_h\|$  & Order\\
\hline
\multirow{4}{*}{$P^0$}
 & 10 & 5.22e-01  &  -0.00 & 3.16e-01  &  -0.00 & 1.40e-01  &  -0.00 & 2.06e-01  &  -0.00 \\ 
 & 20 & 3.11e-01  &  0.75 & 1.70e-01  &  0.89 & 6.61e-02  &  1.08 & 1.10e-01  &  0.90 \\ 
 & 40 & 1.74e-01  &  0.84 & 1.59e-01  &  0.10 & 3.25e-02  &  1.02 & 5.52e-02  &  1.00 \\ 
 & 80 & 9.28e-02  &  0.90 & 1.09e-01  &  0.54 & 1.62e-02  &  1.01 & 2.76e-02  &  1.00 \\ 
 & 160 & 4.76e-02  &  0.96 & 3.76e-02  &  1.54 & 8.10e-03  &  1.00 & 1.39e-02  &  0.99 \\ 
\hline
\multirow{4}{*}{$P^1$}
 & 10 & 1.88e-02  &  -0.00 & 4.72e-02  &  -0.00 & 1.06e-02  &  -0.00 & 1.47e-02  &  -0.00 \\ 
 & 20 & 4.58e-03  &  2.04 & 2.21e-02  &  1.09 & 2.76e-03  &  1.94 & 3.64e-03  &  2.01 \\ 
 & 40 & 1.11e-03  &  2.05 & 1.08e-02  &  1.03 & 6.75e-04  &  2.03 & 8.73e-04  &  2.06 \\ 
 & 80 & 2.77e-04  &  2.00 & 5.40e-03  &  1.01 & 1.69e-04  &  1.99 & 2.18e-04  &  2.00 \\ 
 & 160 & 6.89e-05  &  2.01 & 2.69e-03  &  1.00 & 4.22e-05  &  2.01 & 5.44e-05  &  2.00 \\ 
\hline
\multirow{4}{*}{$P^2$}
 & 10 & 9.22e-04  &  -0.00 & 1.05e-02  &  -0.00 & 6.03e-04  &  -0.00 & 7.04e-04  &  -0.00 \\ 
 & 20 & 1.18e-04  &  2.97 & 1.49e-03  &  2.81 & 7.60e-05  &  2.99 & 8.94e-05  &  2.98 \\ 
 & 40 & 1.43e-05  &  3.04 & 3.92e-04  &  1.92 & 9.25e-06  &  3.04 & 1.09e-05  &  3.03 \\ 
 & 80 & 1.77e-06  &  3.02 & 6.90e-05  &  2.51 & 1.14e-06  &  3.02 & 1.35e-06  &  3.02 \\ 
 & 160 & 2.23e-07  &  2.99 & 4.93e-06  &  3.81 & 1.44e-07  &  2.99 & 1.71e-07  &  2.98 \\ 
\hline
\multirow{4}{*}{$P^3$}
 & 10 & 3.64e-05  &  -0.00 & 1.93e-04  &  -0.00 & 2.37e-05  &  -0.00 & 2.66e-05  &  -0.00 \\ 
 & 20 & 2.51e-06  &  3.86 & 1.05e-05  &  4.20 & 1.66e-06  &  3.84 & 1.87e-06  &  3.83 \\ 
 & 40 & 1.47e-07  &  4.10 & 1.69e-06  &  2.63 & 9.71e-08  &  4.10 & 1.10e-07  &  4.09 \\ 
 & 80 & 9.24e-09  &  3.99 & 1.63e-07  &  3.38 & 6.11e-09  &  3.99 & 6.92e-09  &  3.99 \\ 
 & 160 & 5.71e-10  &  4.02 & 2.02e-08  &  3.01 & 3.78e-10  &  4.02 & 4.28e-10  &  4.01 \\ 
\hline
\multirow{4}{*}{$P^4$}
 & 10 & 1.22e-06  &  -0.00 & 3.76e-05  &  -0.00 & 8.29e-07  &  -0.00 & 9.20e-07  &  -0.00 \\ 
 & 20 & 4.10e-08  &  4.90 & 1.34e-06  &  4.82 & 2.75e-08  &  4.91 & 3.04e-08  &  4.92 \\ 
 & 40 & 1.20e-09  &  5.10 & 8.83e-08  &  3.92 & 8.04e-10  &  5.10 & 8.89e-10  &  5.10 \\ 
 & 80 & 3.63e-11  &  5.04 & 3.67e-09  &  4.59 & 2.44e-11  &  5.04 & 2.70e-11  &  5.04 \\ 
 & 160 & 1.17e-12  &  4.96 & 7.66e-11  &  5.58 & 7.87e-13  &  4.95 & 8.62e-13  &  4.97 \\ 
\hline
\end{tabular}
\end{table}

\subsection*{Example 4.2: 1D advection with inflow boundary condition}
We consider the same problem as in Example 4.1, but with the inflow boundary condition at the left end
\[
 u(0,t) = \sin(-2\pi t).
\]
We use the three DG methods considered in Example 4.1, again on a 
nonuniform mesh which is a $10\%$ random perturbation of the  uniform mesh.
For the boundary treatment, the upwinding boundary numerical flux \eqref{bdry} 
is used for all three DG methods.
% And 
The time integration takes into account the boundary source term; see \eqref{rk-s}.
Table \ref{table:adv1d-D} lists the numerical errors and their orders with the three DG methods at $T=0.5$.
We observe similar convergence results as that for Example 4.1. 
We specifically mention that the Lax-Wendroff time integration \eqref{rk-s} 
do not leads to order reduction, which is typically 
observed for Runge-Kutta methods.
% We specifically remark that  the boundary treatment \qref{fully-}
% We use $P^k$ polynomials with $0\le k\le 4$ on a nonuniform mesh which is a $10\%$ random perturbation of the 
% uniform mesh, and take the time step size to 
% be $\Delta t = C\!F\!L\,h$ with  $C\!F\!L = 0.5$ for $P^0$, 
% $C\!F\!L = 0.2$ for $P^1$, $C\!F\!L = 0.1$ for $P^2$, $C\!F\!L = 0.06$ for $P^3$, $C\!F\!L = 0.04$ for $P^4$.
% From the table we conclude that, 
% one can always observe optimal $(k+1)$th order of accuracy for both the variable $u_h$ (which approximate the solution $u$)
% and $\phi_h$ (which approximate the {\it zero} function)
% for the new energy-conserving DG method \eqref{scheme:adv1d}. This validate our convergence result in Theorem \ref{thm:adv1d:err}.
% Moreover, the the absolute value of the error is always slightly smaller than
% the optimal upwinding DG method.
% We also observe suboptimal convergence for the energy-conserving DG method with a central flux for all polynomial degree.

\begin{table}[htbp]
\caption{\label{table:adv1d-D} 
The $L^2$-errors
% $\|u-u_h\|_I$ (and $\|\phi_h\|$) 
% at time $T=0.5$ 
% $e_u(T) = \|u-u_h\|_{L^2(I)}$ and $e_\phi=\|\phi_h\|_{L^2(I)}$) 
and orders for Example 4.2 for the upwinding DG method (U), the central DG method (C), and the 
new DG method (A) on a random mesh of $N$ cells. 
$T =0.5$. 
} \centering
\bigskip
\begin{tabular}{|c|c|cc|cc|cc|cc|}
\hline
 & &\multicolumn{2}{|c|}{(U)}&\multicolumn{2}{c|}{(C)}&\multicolumn{4}{c|}{(A)}\\
 \hline
%  \hline
%  \cline{3-4}   \cline{6-7} \cline{9-12}
 &  {$N$}      & $\|u-u_h\|$ & Order & $\|u-u_h\|$ & Order &  $\|u-u_h\|$ & Order 
 &  $\|\phi_h\|$  & Order\\
\hline
\multirow{4}{*}{$P^0$}
 & 10 & 4.54e-01  &  -0.00 & 3.27e-01  &  -0.00 & 1.89e-01  &  -0.00 & 2.45e-01  &  -0.00 \\ 
 & 20 & 2.61e-01  &  0.80 & 1.58e-01  &  1.05 & 8.16e-02  &  1.21 & 1.15e-01  &  1.09 \\ 
 & 40 & 1.46e-01  &  0.84 & 7.20e-02  &  1.13 & 3.63e-02  &  1.17 & 5.76e-02  &  1.00 \\ 
 & 80 & 7.79e-02  &  0.90 & 4.81e-02  &  0.58 & 1.75e-02  &  1.05 & 2.86e-02  &  1.01 \\ 
 & 160 & 3.99e-02  &  0.96 & 5.29e-02  &  -0.14 & 8.45e-03  &  1.05 & 1.41e-02  &  1.02 \\ 
\hline
\multirow{4}{*}{$P^1$}
 & 10 & 1.78e-02  &  -0.00 & 4.75e-02  &  -0.00 & 1.14e-02  &  -0.00 & 1.49e-02  &  -0.00 \\ 
 & 20 & 4.38e-03  &  2.02 & 2.36e-02  &  1.01 & 2.73e-03  &  2.07 & 3.51e-03  &  2.09 \\ 
 & 40 & 1.11e-03  &  1.98 & 1.14e-02  &  1.05 & 6.73e-04  &  2.02 & 8.67e-04  &  2.02 \\ 
 & 80 & 2.76e-04  &  2.01 & 5.73e-03  &  1.00 & 1.68e-04  &  2.00 & 2.17e-04  &  2.00 \\ 
 & 160 & 6.86e-05  &  2.01 & 2.87e-03  &  1.00 & 4.20e-05  &  2.00 & 5.43e-05  &  2.00 \\ 
\hline
\multirow{4}{*}{$P^2$}
 & 10 & 8.87e-04  &  -0.00 & 1.76e-03  &  -0.00 & 6.58e-04  &  -0.00 & 7.57e-04  &  -0.00 \\ 
 & 20 & 1.20e-04  &  2.89 & 1.33e-04  &  3.72 & 8.04e-05  &  3.03 & 9.55e-05  &  2.99 \\ 
 & 40 & 1.42e-05  &  3.08 & 1.66e-05  &  3.01 & 9.83e-06  &  3.03 & 1.17e-05  &  3.03 \\ 
 & 80 & 1.78e-06  &  2.99 & 5.52e-06  &  1.59 & 1.20e-06  &  3.03 & 1.44e-06  &  3.03 \\ 
 & 160 & 2.23e-07  &  2.99 & 4.34e-07  &  3.67 & 1.48e-07  &  3.02 & 1.73e-07  &  3.06 \\ 
\hline
\multirow{4}{*}{$P^3$}
 & 10 & 4.04e-05  &  -0.00 & 1.13e-04  &  -0.00 & 2.27e-05  &  -0.00 & 2.44e-05  &  -0.00 \\ 
 & 20 & 2.13e-06  &  4.24 & 1.55e-05  &  2.86 & 1.47e-06  &  3.95 & 1.66e-06  &  3.88 \\ 
 & 40 & 1.47e-07  &  3.86 & 1.84e-06  &  3.08 & 9.43e-08  &  3.96 & 1.07e-07  &  3.95 \\ 
 & 80 & 8.84e-09  &  4.05 & 2.28e-07  &  3.01 & 5.85e-09  &  4.01 & 6.64e-09  &  4.01 \\ 
 & 160 & 5.64e-10  &  3.97 & 2.81e-08  &  3.02 & 3.70e-10  &  3.98 & 4.20e-10  &  3.98 \\ 
\hline
\multirow{4}{*}{$P^4$}
 & 10 & 1.04e-06  &  -0.00 & 2.76e-06  &  -0.00 & 1.03e-06  &  -0.00 & 1.05e-06  &  -0.00 \\ 
 & 20 & 3.65e-08  &  4.84 & 5.33e-08  &  5.70 & 2.52e-08  &  5.35 & 2.63e-08  &  5.32 \\ 
 & 40 & 1.20e-09  &  4.93 & 3.65e-09  &  3.87 & 8.79e-10  &  4.84 & 9.32e-10  &  4.82 \\ 
 & 80 & 3.69e-11  &  5.02 & 1.84e-10  &  4.31 & 2.52e-11  &  5.13 & 2.76e-11  &  5.08 \\ 
 & 160 & 1.13e-12  &  5.02 & 4.04e-12  &  5.51 & 7.81e-13  &  5.01 & 8.64e-13  &  5.00 \\ 
\hline
\end{tabular}
\end{table}

\subsection*{Example 4.3: 1D acoustics with periodic boundary condition}
We consider the acoustic equation \eqref{acoustics1d} with coefficients 
$\rho_0=K_0=1, u_0 = 0.5$, i.e.,
\begin{align*}
 p_t + .5p_x + u_x = &0,\\
 u_t + p_x + .5u_x = &0.
\end{align*}
The domain is a unit interval $I=[0,1]$.
The initial condition \[p(x,0) = \sin(2\pi x), u(x,0) = 0,\] and a periodic boundary condition is used.
The exact solution is
\begin{align*}
p(x, t) = \frac12\sin(2\pi (x-1.5 t))+\frac12\sin(2\pi (x+.5 t)),\\[.2ex]
u(x, t) = \frac12\sin(2\pi (x-1.5 t))-\frac12\sin(2\pi (x+.5 t)).
\end{align*}
We are in the subsonic regime, 
Theorem \ref{thm:ac1d:err} indicates the energy-conserving DG method \eqref{scheme:acoustics1d}
with the numerical flux \eqref{acoustics-flux-1} and \eqref{acoustics-flux-2} 
using $\alpha_{j-\frac12} = \frac12\sqrt{0.75}$ is optimally convergent. We label this method as (A). 
Again, we also consider the numerical results for 
the DG method with the upwinding flux, labeled as (U),  and with the central flux, labeled as (C).

Table \ref{table:ac1d} lists the numerical errors and their orders with the three DG methods at $T=0.5$.
Again, we observe optimal convergence for the DG method with upwinding flux (U) and with the 
new energy-conserving flux (A), but suboptimal convergence for the DG method with central flux (C).

Similar numerical results, not reported here to save space,
are also obtained for the supersonic case where the new method (A) shall solve an 
augmented system with 4 components.

\begin{table}[htbp]
\caption{\label{table:ac1d} 
The $L^2$-error $(\|u-u_h\|^2+\|p-p_h\|^2)^{1/2}$
and orders for Example 4.3 for the upwinding DG method (U), the central DG method (C), and the 
new DG method (A) on a random mesh of $N$ cells. 
$T =0.5$. 
} \centering
\bigskip
\begin{tabular}{|c|c|cc|cc|cc|}
\hline
 & &\multicolumn{2}{|c|}{(U)}&\multicolumn{2}{c|}{(C)}&\multicolumn{2}{c|}{(A)}\\
%  \cline{3-4}   \cline{6-7} \cline{9-12}
 &  {$N$}      & Error & Order & Error & Order &  Error & Order \\
\hline
\multirow{4}{*}{$P^0$}
 & 10 & 5.04e-01  &  -0.00 & 3.64e-01  &  -0.00 & 3.08e-01  &  -0.00 \\ 
 & 20 & 3.23e-01  &  0.64 & 2.26e-01  &  0.69 & 1.50e-01  &  1.04 \\ 
 & 40 & 1.83e-01  &  0.82 & 1.52e-01  &  0.57 & 7.46e-02  &  1.01 \\ 
 & 80 & 9.93e-02  &  0.88 & 1.07e-01  &  0.51 & 3.78e-02  &  0.98 \\ 
 & 160 & 5.15e-02  &  0.95 & 3.47e-02  &  1.62 & 1.90e-02  &  0.99 \\ 
\hline
\multirow{4}{*}{$P^1$}
 & 10 & 1.94e-02  &  -0.00 & 5.00e-02  &  -0.00 & 2.56e-02  &  -0.00 \\ 
 & 20 & 4.56e-03  &  2.09 & 2.22e-02  &  1.17 & 5.45e-03  &  2.23 \\ 
 & 40 & 1.10e-03  &  2.05 & 1.08e-02  &  1.04 & 1.31e-03  &  2.06 \\ 
 & 80 & 2.73e-04  &  2.02 & 5.35e-03  &  1.01 & 3.25e-04  &  2.01 \\ 
 & 160 & 6.86e-05  &  1.99 & 2.68e-03  &  1.00 & 8.16e-05  &  2.00 \\ 
\hline
\multirow{4}{*}{$P^2$}
 & 10 & 8.96e-04  &  -0.00 & 1.18e-02  &  -0.00 & 1.26e-03  &  -0.00 \\ 
 & 20 & 1.20e-04  &  2.90 & 2.14e-03  &  2.46 & 1.51e-04  &  3.06 \\ 
 & 40 & 1.44e-05  &  3.06 & 3.71e-04  &  2.53 & 1.66e-05  &  3.19 \\ 
 & 80 & 1.81e-06  &  2.99 & 6.52e-05  &  2.51 & 2.12e-06  &  2.97 \\ 
 & 160 & 2.25e-07  &  3.01 & 4.10e-06  &  3.99 & 2.67e-07  &  2.99 \\ 
\hline
\multirow{4}{*}{$P^3$}
 & 10 & 3.76e-05  &  -0.00 & 2.01e-04  &  -0.00 & 4.79e-05  &  -0.00 \\ 
 & 20 & 2.43e-06  &  3.95 & 1.02e-05  &  4.30 & 2.85e-06  &  4.07 \\ 
 & 40 & 1.45e-07  &  4.07 & 1.67e-06  &  2.62 & 1.71e-07  &  4.06 \\ 
 & 80 & 8.89e-09  &  4.03 & 1.59e-07  &  3.39 & 1.04e-08  &  4.03 \\ 
 & 160 & 5.65e-10  &  3.98 & 2.01e-08  &  2.98 & 6.62e-10  &  3.98 \\ 
\hline
\multirow{4}{*}{$P^4$}
 & 10 & 1.18e-06  &  -0.00 & 4.10e-05  &  -0.00 & 1.36e-06  &  -0.00 \\ 
 & 20 & 4.30e-08  &  4.78 & 1.97e-06  &  4.38 & 5.12e-08  &  4.73 \\ 
 & 40 & 1.22e-09  &  5.14 & 8.35e-08  &  4.56 & 1.43e-09  &  5.16 \\ 
 & 80 & 3.83e-11  &  4.99 & 3.46e-09  &  4.59 & 4.48e-11  &  5.00 \\ 
 & 160 & 1.18e-12  &  5.02 & 6.52e-11  &  5.73 & 1.38e-12  &  5.02 \\ 
\hline
\end{tabular}
\end{table}

\subsection*{Example 4.4: long time simulation: advection of a plane wave}
We consider the advection equation \eqref{eq1} on the unit interval with periodic boundary condition and 
initial condition $u(x,0) = \sin(6\pi x)$. The exact solution is 
\[
 u(x,t) = \sin(6\pi(x-t)).
\]
We use the above mentioned three DG methods using quadratic polynomials $k=2$.
Again, we denote the upwinding flux as (U), 
the central flux as (C), and the new method as (A).
% numerical flux \eqref{flux:opt} for the augmented 
% system \eqref{scheme:adv1d} as (A).
It is known that all three methods have optimally third-order convergence on uniform meshes.
We take a uniform mesh with $N = 10$ cells, so there are 10 degrees of freedom per wavelength.

For the time integration, we use the 3-stage, 3-rd order ($r=3$) 
Lax-Wendroff time stepping \eqref{rk}, denoted as RK3. This is identical
to the SSP-RK3 method, which is known to be dissipative. 
We also use  the 3-stage, 4-th order ($r=2$) 
energy-conserving, Lax-Wendroff time stepping \eqref{fully-discrete}, denoted as LF4,
for the energy-conserving DG methods (C) and (A). The CFL number for all cases is taken to be $0.1$.

The numerical results at time $T=10$ (wave propagated 30 cycles) 
% and $T=50$
% (wave propagated 150 cycles) 
of the three DG methods using RK3 time stepping
are shown in  Figure \ref{fig:pw1d1}.
% and Figure \ref{fig:pw1d2}.
From this figure, 
we observe that the upwinding method (U) is very dissipative,
% but has excellent phase accuracy, 
the central method (C) is less dissipative
% , with the dissipation comes from RK3 time integration, 
but has a large phase error, while the new method (A) provides 
excellent results in terms of 
dissipation error and phase accuracy.

\begin{figure}[ht!]
 \caption{Numerical solution at $T=10$ for Example 4.4. RK3 time stepping.
 Top left: method (U). Top right: method (C).
 Bottom left: primal variable $u_h$ for method (A). Bottom right: 
 auxiliary variable $\phi_h$ for method (A). 
 Solid line: numerical solution. Dashed line: exact solution.
 DG-$P^2$ space,  $10$ cells.}
 \label{fig:pw1d1}
 \includegraphics[width=.45\textwidth]{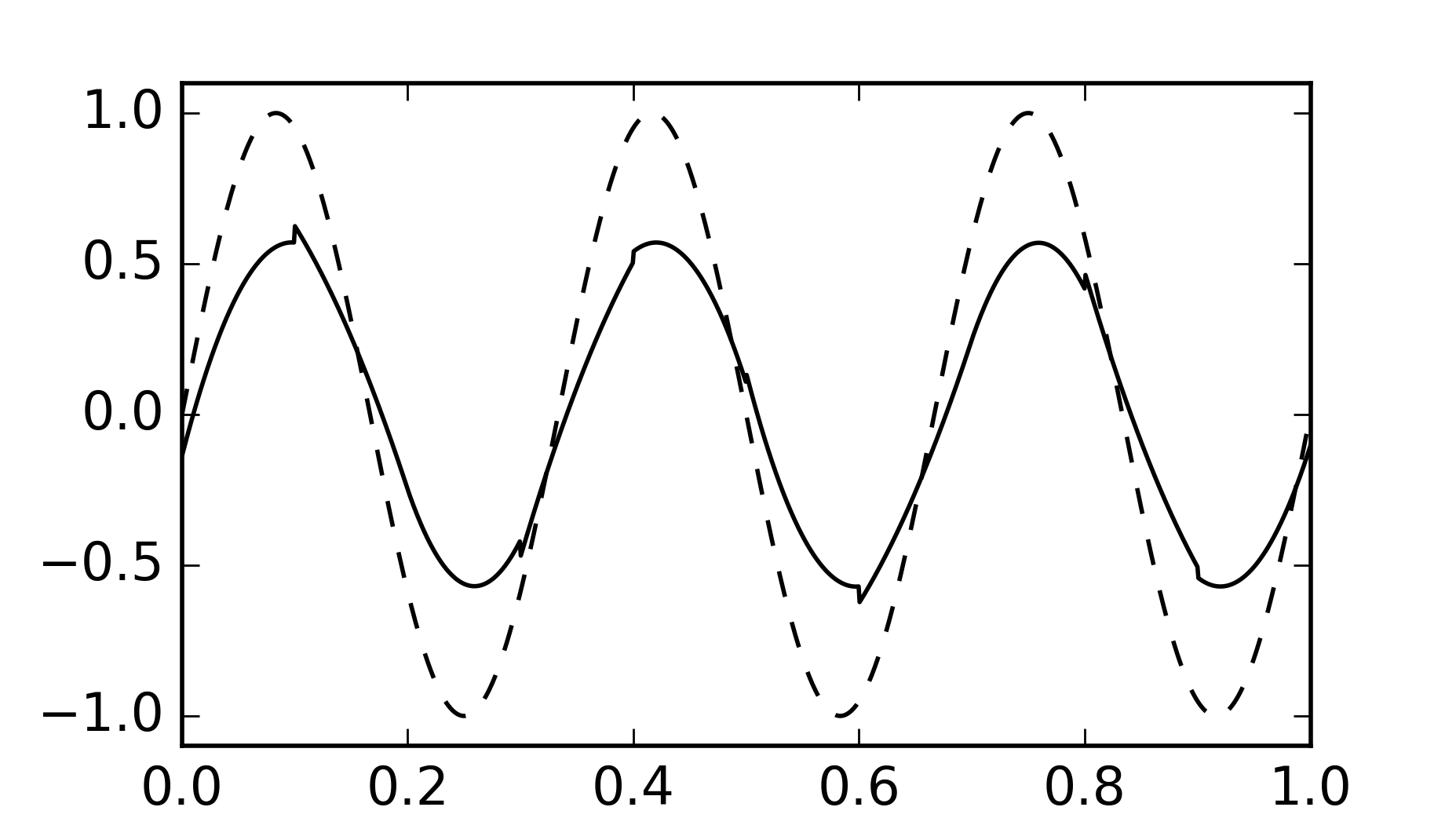}
 \includegraphics[width=.45\textwidth]{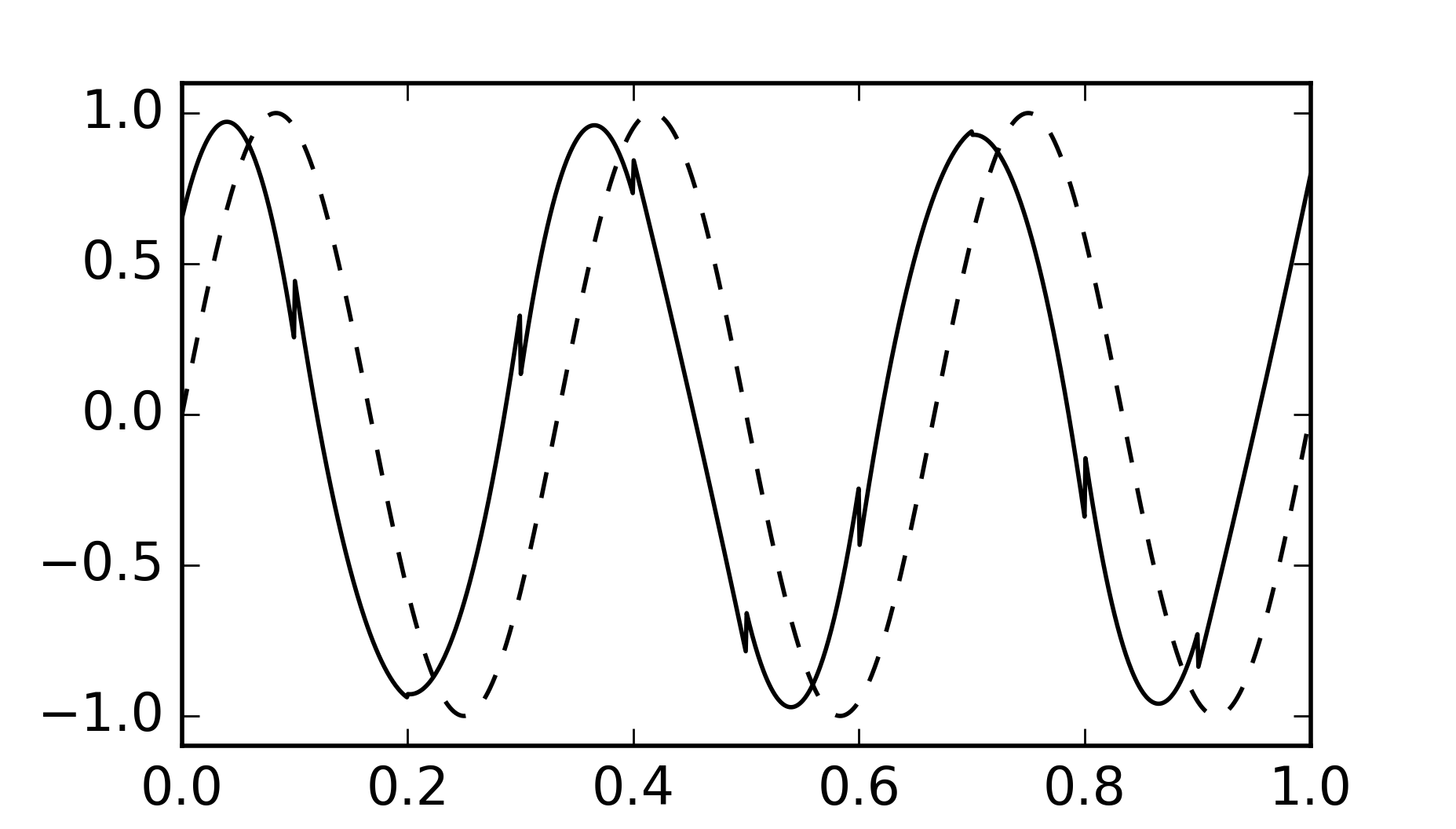}
 \includegraphics[width=.45\textwidth]{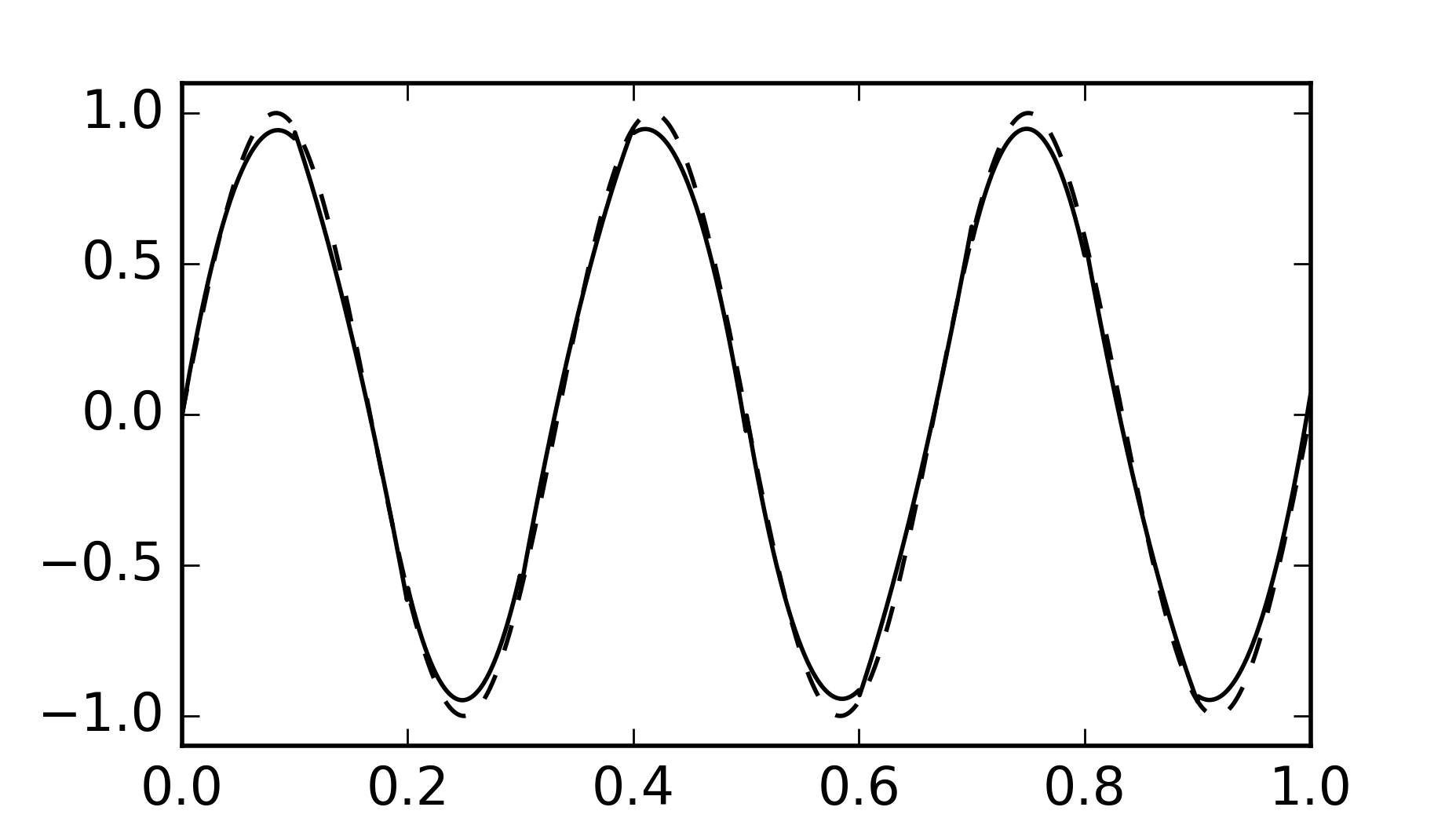}
 \includegraphics[width=.45\textwidth]{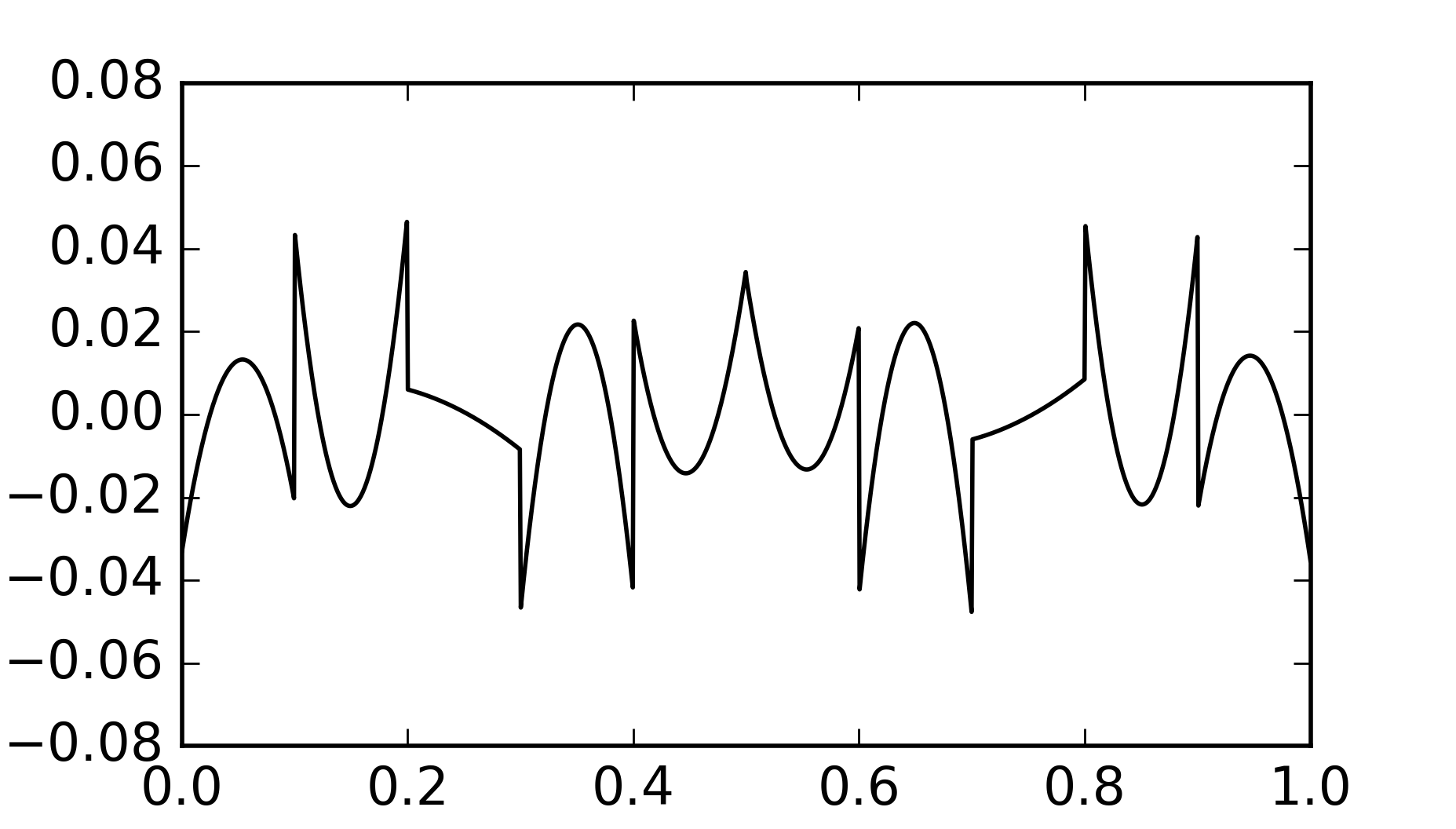}
\end{figure}

% \begin{figure}[ht!]
%  \caption{Numerical solution at $T=50$ for Example 4.5.RK3 time stepping.
%  Top left: method (U). Top right: method (C).
%  Bottom left: primal variable $u_h$ for method (A). Bottom right: 
%  auxiliary variable $\phi_h$ for method (A). 10 cells. }
%  \label{fig:pw1d2}
%  \includegraphics[width=.45\textwidth]{data/pw1d/upwindingrku5.png}
%  \includegraphics[width=.45\textwidth]{data/pw1d/centralrku5.png}
%  \includegraphics[width=.45\textwidth]{data/pw1d/auxrku5.png}
%  \includegraphics[width=.45\textwidth]{data/pw1d/auxrkp5.png}
% \end{figure}

% Since the spatial discretization for the methods (C) and (A) are energy-conserving, 
We also present the 
numerical results in Figure \ref{fig:pw1d3} 
% ($T=10$) and Figure \ref{fig:pw1d4} ($T=50$)
for two energy-conserving methods (C) and (A) 
using the energy-conserving LF4 time integration. 
% We note that the upwind method (U) is 
% unstable for the LF4 time integration.
Numerical dissipation is not visible from the figures.
But again, we observe large phase error for the central method (C), and small 
phase error for the new method (A).
\begin{figure}[ht!]
 \caption{Numerical solution at $T=10$ for Example 4.4. LF4 time stepping.
 Left : method (C).  Right: method (A)
%  Bottom left: primal variable $u_h$ for method (A). Bottom right: 
%  auxiliary variable $\phi_h$ for method (A). 
 Solid line: numerical solution. Dashed line: exact solution.
  DG-$P^2$ space,  $10$ cells.}
 \label{fig:pw1d3}
 \includegraphics[width=.45\textwidth]{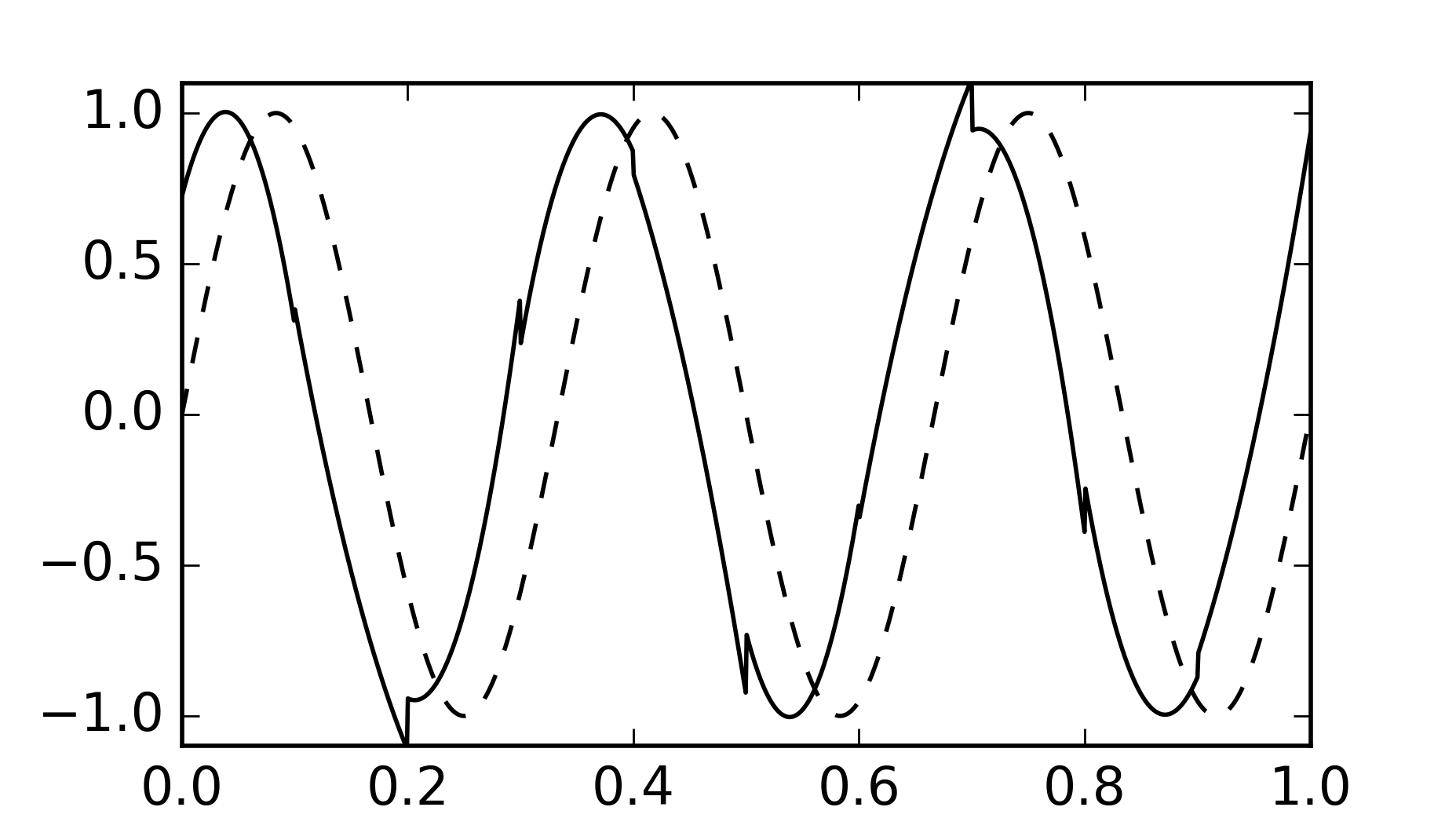}
 \includegraphics[width=.45\textwidth]{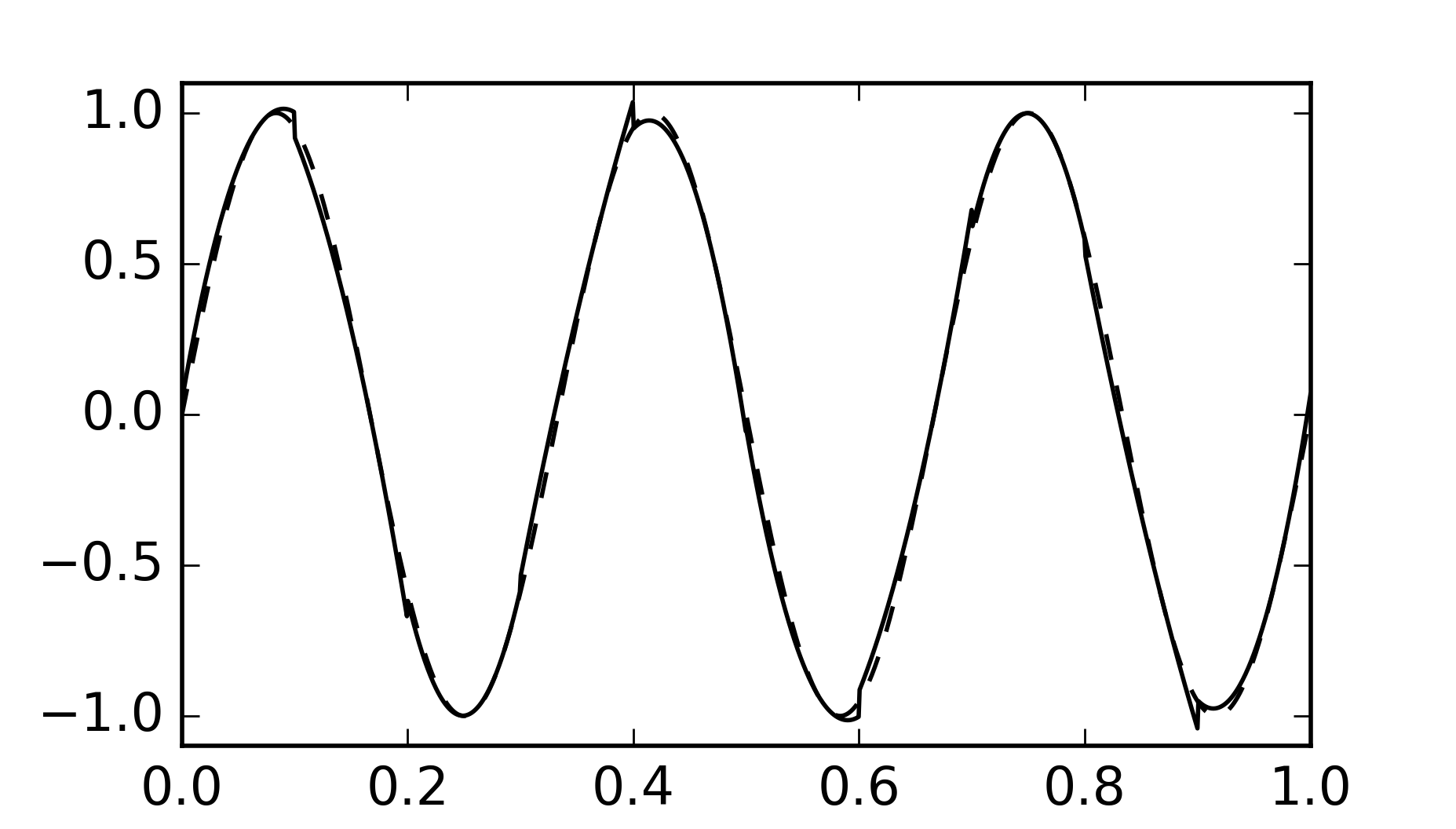}
\end{figure}

% \begin{figure}[ht!]
%  \caption{Numerical solution at $T=50$ for Example 4.4. LF4 time stepping.
%   Left : method (C).  Right: method (A)
% %  Top : method (C).
% %  Bottom left: primal variable $u_h$ for method (A). Bottom right: 
% %  auxiliary variable $\phi_h$ for method (A). 
%   Solid line: numerical solution. Dashed line: exact solution.
% 10 cells. }
%  \label{fig:pw1d4}
%  \includegraphics[width=.45\textwidth]{data/pw1d/centralleapfrogu5.png}
%  \includegraphics[width=.45\textwidth]{data/pw1d/auxleapfrogu5.png}
% % \includegraphics[width=.45\textwidth]{data/pw1d/auxleapfrogp5.png}
% \end{figure}

\subsection*{Example 4.5: long time simulation: advection of a Gaussian pulse}
We consider the advection equation \eqref{eq1} on the unit interval with periodic boundary condition and 
initial condition $u(x,0) = \exp(-200(x-.5)^2)$. 
% The exact solution is 
% \[
%  u(x,t) = \left\{
%  \begin{tabular}{cc}
% $\exp(-200(x-.5-t)^2)$ & if $x>t$,\\[.2ex]
% $\exp(-200(x+.5-t)^2)$ & if $x<t$,\\
%  \end{tabular}
% \]

Again, we use the above mentioned three DG methods with quadratic polynomial space ($k=2$).
% Again, we use quadratic $P^2$-DG methods with the above mentioned three numerical fluxes.
% Again, we denote the upwinding flux as (U), 
% the central flux as (C), and the numerical flux \eqref{flux:opt} for the augmented 
% system \eqref{scheme:adv1d} as (A).
% It is known that all three methods are optimally convergent on uniform meshes.
We take a uniform mesh with $N = 20$ cells, so there are a total of 60 
degrees of freedom, which can roughly resolve 
waves frequency up to $k = 24\pi$.

The numerical results at time $T=40$ (wave propagated 40 cycles)
of the three DG methods using RK3 time stepping
are shown in  Figure \ref{fig:gs1d1}.
From this figure, 
we observe large dissipation error for the upwinding method (U), 
large dispersion error for the central method (C), 
and relatively the smallest dissipation and dispersion errors for the new method (A).

\begin{figure}[ht!]
 \caption{Numerical solution at $T=40$ for Example 4.5. RK3 time stepping.
 Top left: method (U). Top right: method (C).
 Bottom left: primal variable $u_h$ for method (A). Bottom right: 
 auxiliary variable $\phi_h$ for method (A). 
  Solid line: numerical solution. Dashed line: exact solution.
  DG-$P^2$ space,  $20$ cells.}
 \label{fig:gs1d1}
 \includegraphics[width=.45\textwidth]{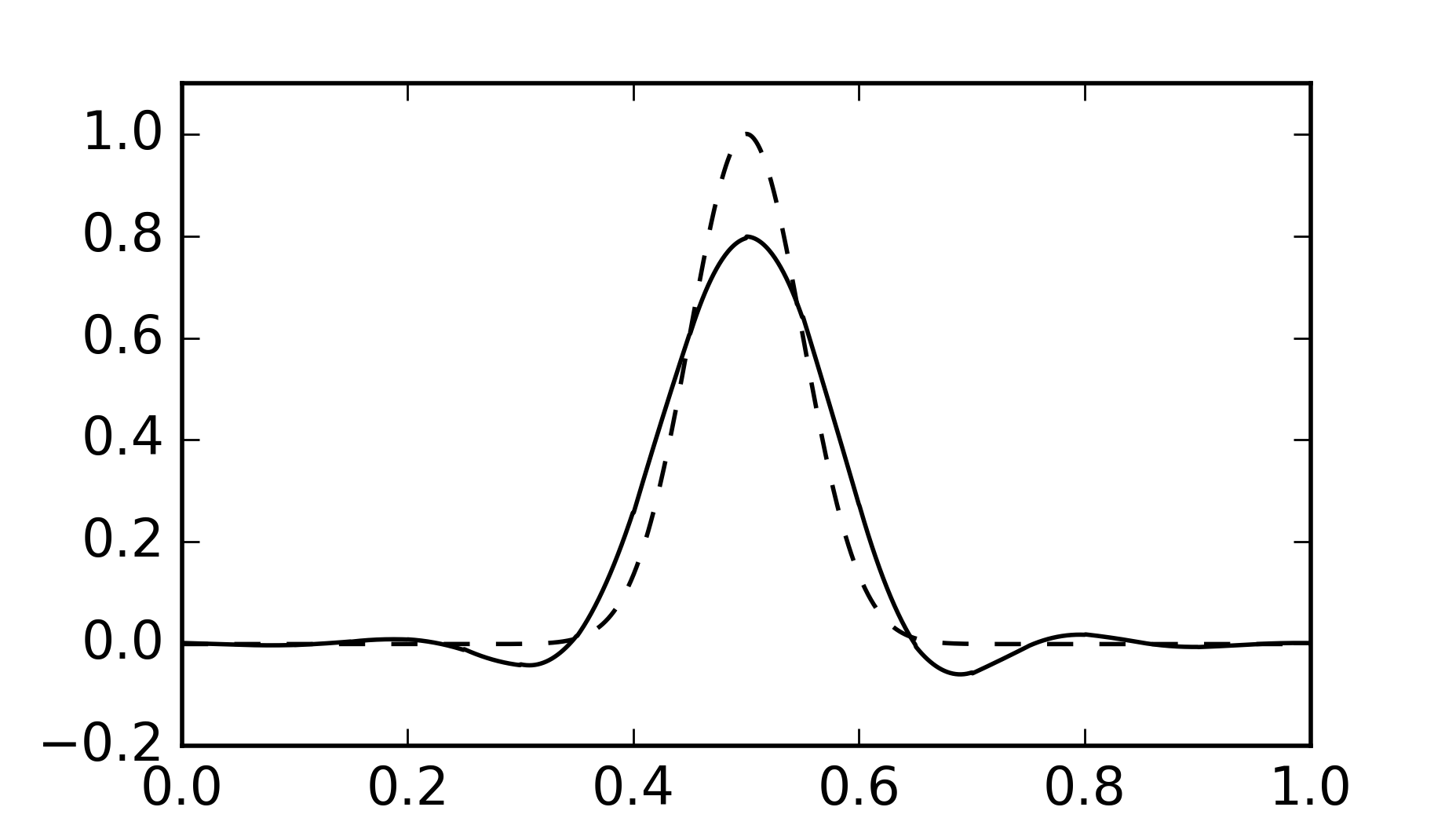}
 \includegraphics[width=.45\textwidth]{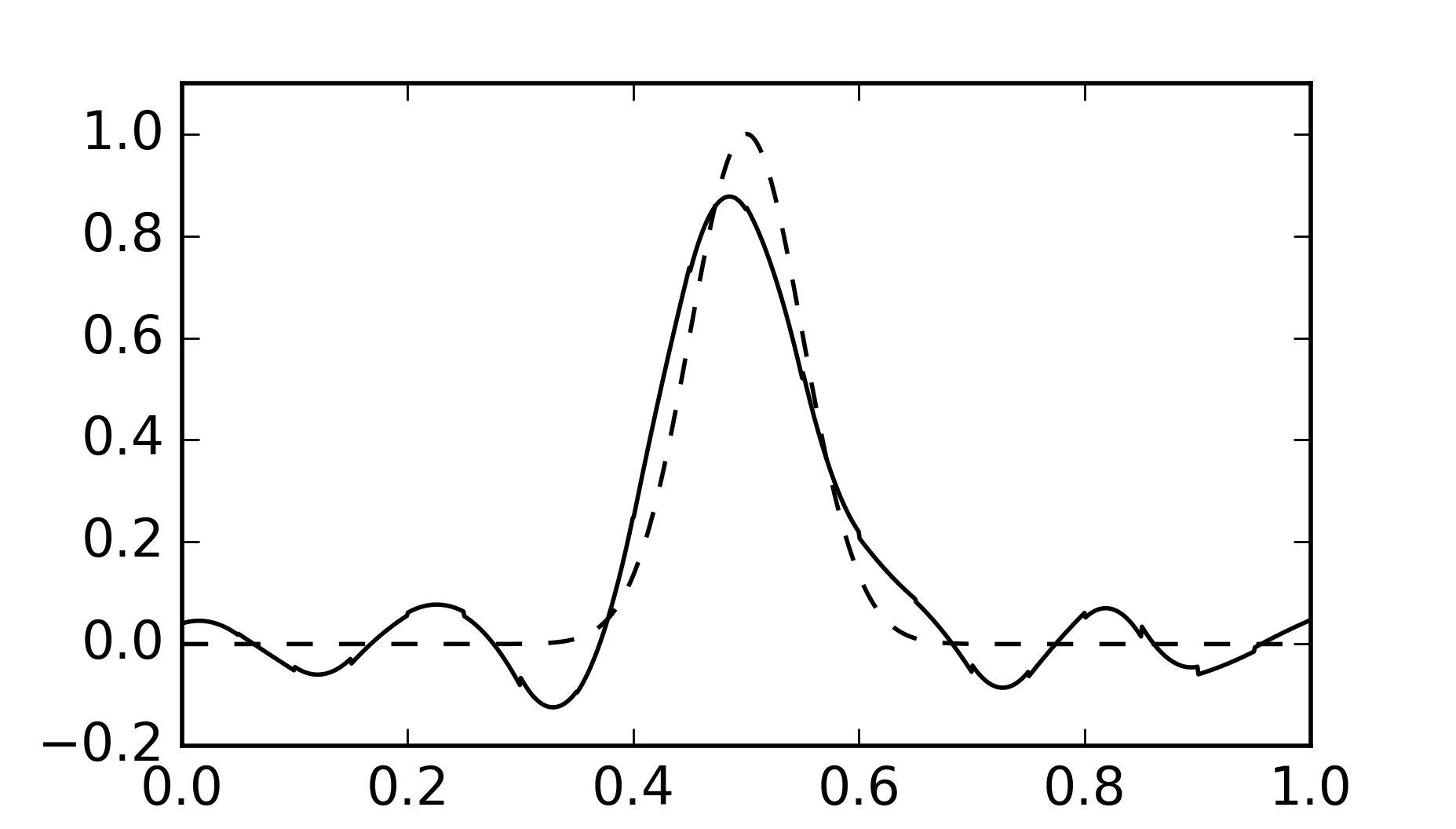}
 \includegraphics[width=.45\textwidth]{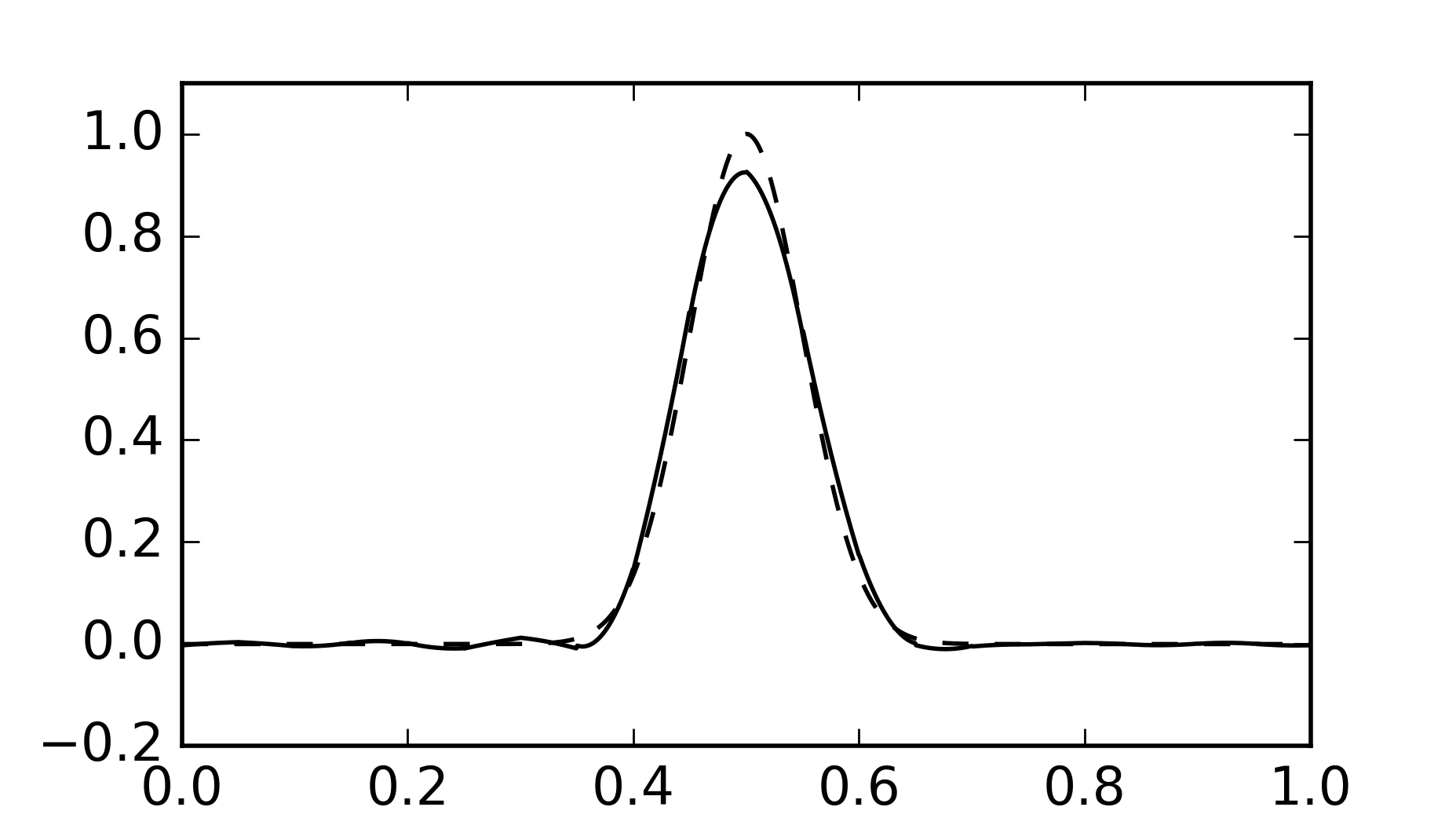}
 \includegraphics[width=.45\textwidth]{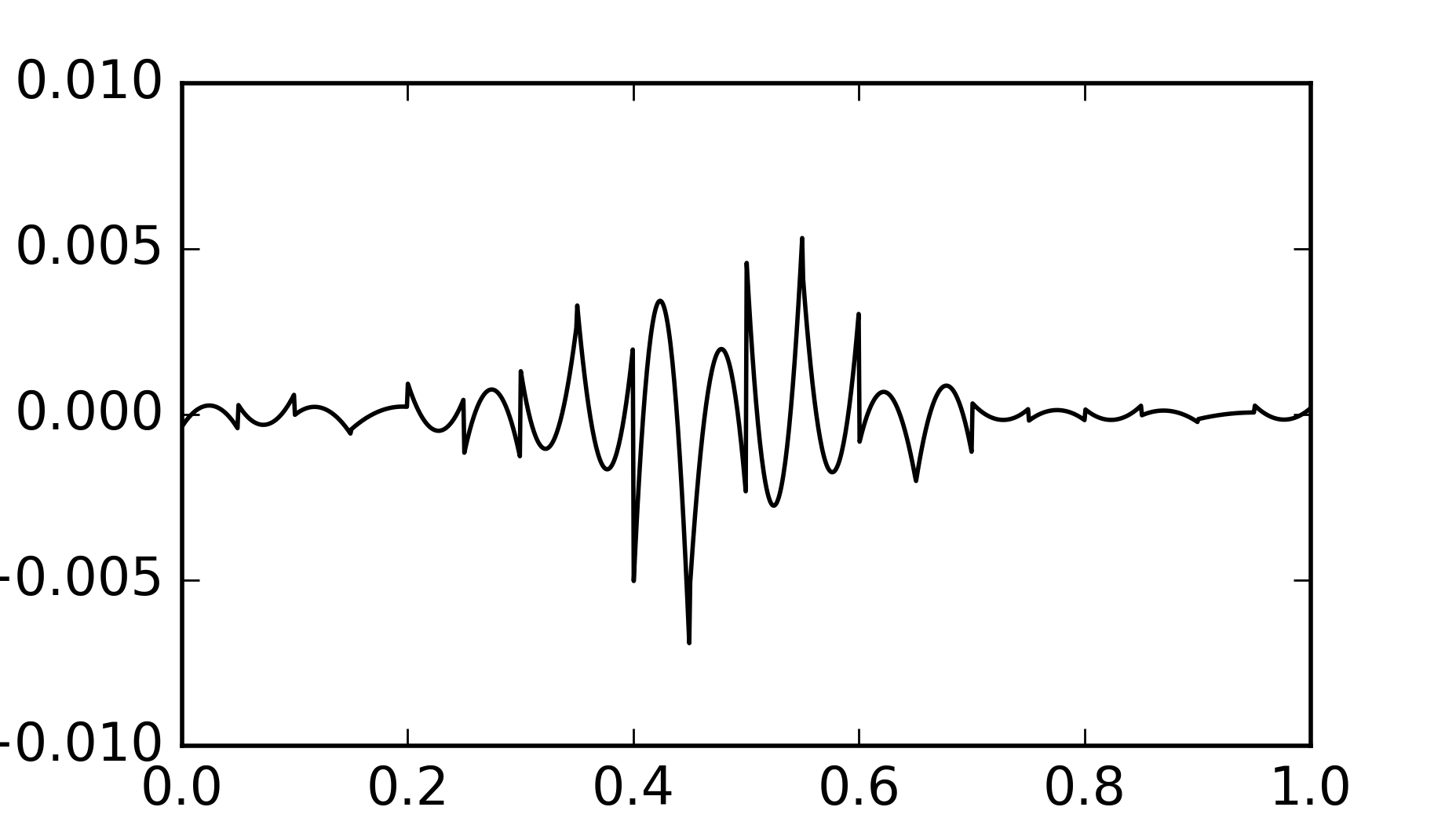}
\end{figure}

% Since the spatial discretization for the methods (C) and (A) are energy-conserving, 
We also present the 
numerical results at time $T=40$ in Figure \ref{fig:pw1d3} 
for the two energy-conserving methods (C) and (A) 
using the energy-conserving LF4 time integration.
% Numerical dissipation is not visible from the figures.
This time, we observe a larger 
dispersion error for both methods, with the dissipation error for (A) sightly reduced. 
% with a larger error 
% for the central method (C) than for the new method (A).
\begin{figure}[ht!]
 \caption{Numerical solution at $T=40$ for Example 4.6. LF4 time stepping.
   Left : method (C).  Right: method (A)
%  Top : method (C).
%  Bottom left: primal variable $u_h$ for method (A). Bottom right: 
%  auxiliary variable $\phi_h$ for method (A). 
  Solid line: numerical solution. Dashed line: exact solution.
  DG-$P^2$ space,  $20$ cells.}
 \label{fig:gs1d3}
 \includegraphics[width=.45\textwidth]{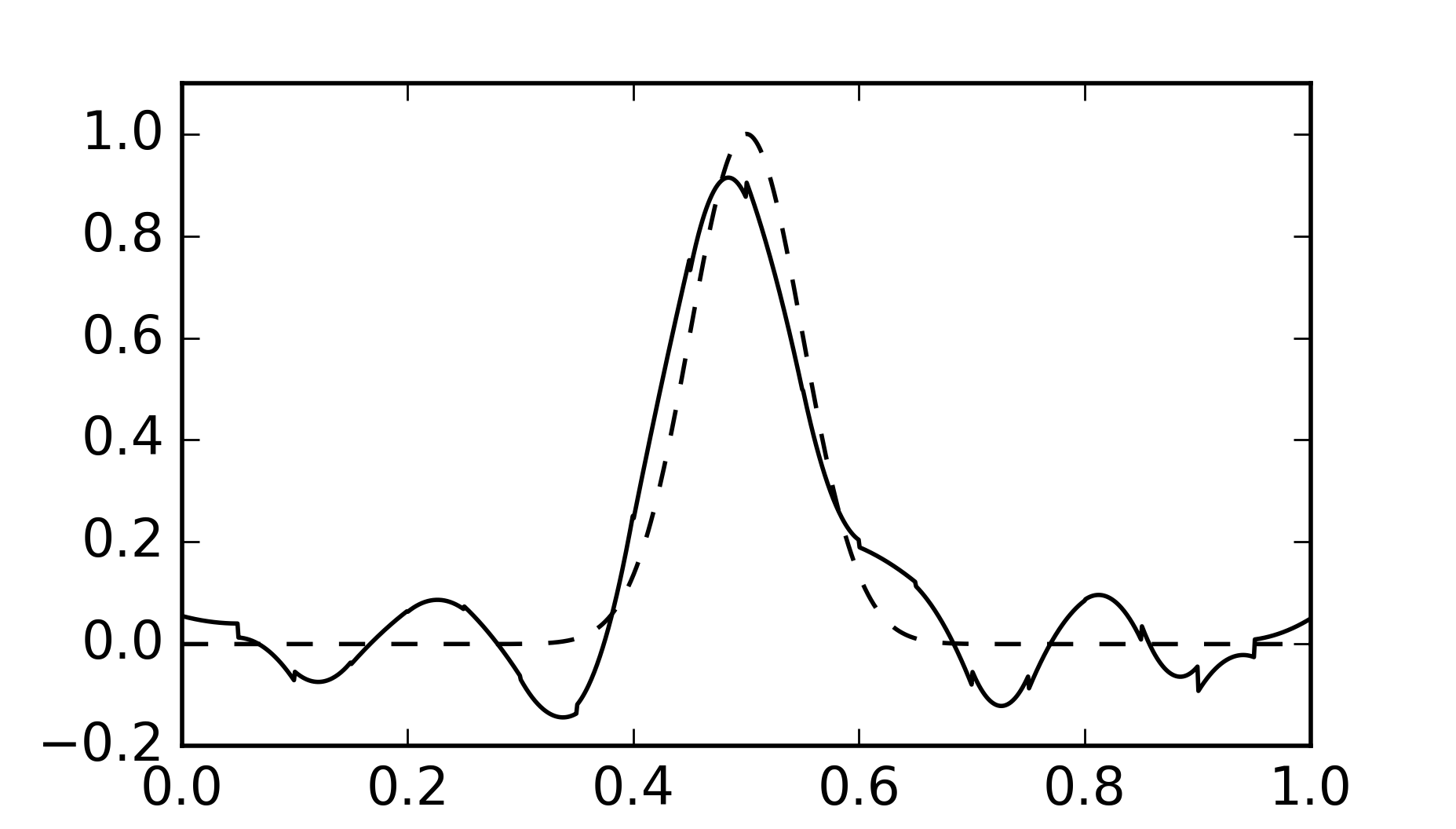}
 \includegraphics[width=.45\textwidth]{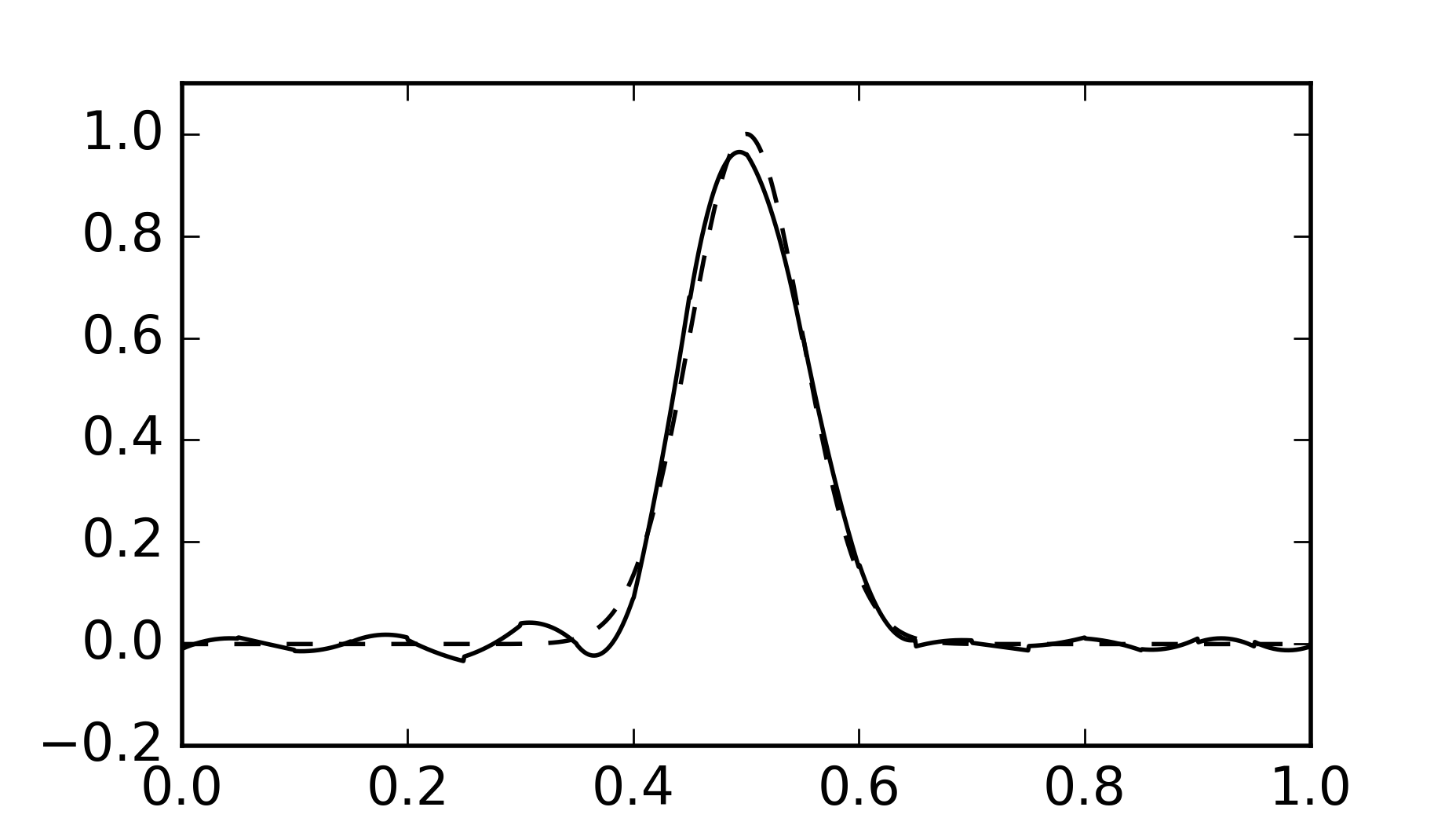}
\end{figure}

\subsection*{Example 4.6: long time simulation: spherical wave problem}
This is our our last one-dimensional example. 
We consider the following spherical wave problem, which was one of the 
benchmark problems proposed in the first computational aeroacoustics workshop \cite{CAA1},
\[
 u_t + u/r + u_r = 0
\]
over the domain $5\le r\le 450$, with initial condition $u(x,0) = 0$. The boundary condition at 
$r=5$ is 
\[
 u(5, t) = \sin(\omega t),\quad \text{ with $\omega = \pi/3$}.
\]
% with $\omega = \pi/3$.
Exact solution is 
\begin{align*}
 u(r,t) = \left\{
 \begin{tabular}{cc}
  $0$ & $r> t+5$\\[.2ex]
  $\frac5r[\sin(\omega(t-r+5))]$ & $r\le t+5.$  
 \end{tabular}
 \right.
\end{align*}

Again, we use the above mentioned three DG methods with quadratic polynomial space ($k=2$).
Here we mention that although there is a source term in this equation,
the auxiliary {\it zero} variable for the method (A) still solve the equation 
\[
 \phi_t - \phi_r = 0.
\]
We take a uniform mesh with $N = 250$ cells, so there are about 10 degrees of freedom per wavelength.

The numerical results at time $T=400$ (wave propagated 50 cycles)
along the segment $350\le r\le 430$
of the three DG methods using RK3 time stepping
are shown in  Figure \ref{fig:sp1d1}.
From this figure, 
we observe large dissipation error for the upwinding method (U), 
large dissipation error and phase shift for the central method (C), 
and relatively the smallest dissipation error and phase shift for the new method (A).

\begin{figure}[ht!]
 \caption{Numerical solution for  $350\le r\le 430$
 at $T=400$ for Example 4.6. RK3 time stepping.
 Top: method (U). 
 Middle: method (C).
 Bottom: method (A).
  Solid line: numerical solution. Dashed line: exact solution.
DG-$P^2$ space,  $250$ cells.
}
 \label{fig:sp1d1}
 \includegraphics[width=0.8\textwidth]{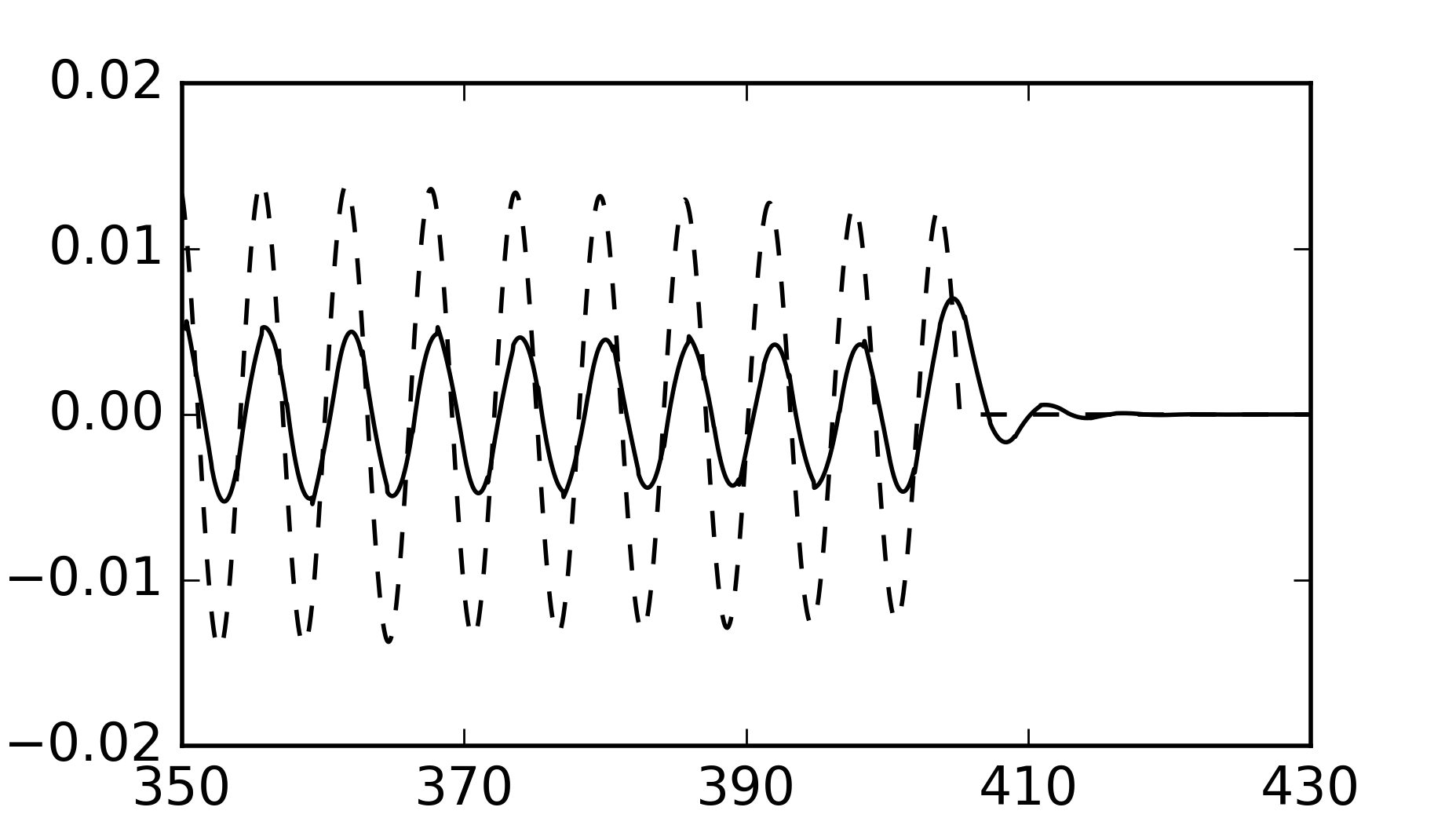}
 \includegraphics[width=0.8\textwidth]{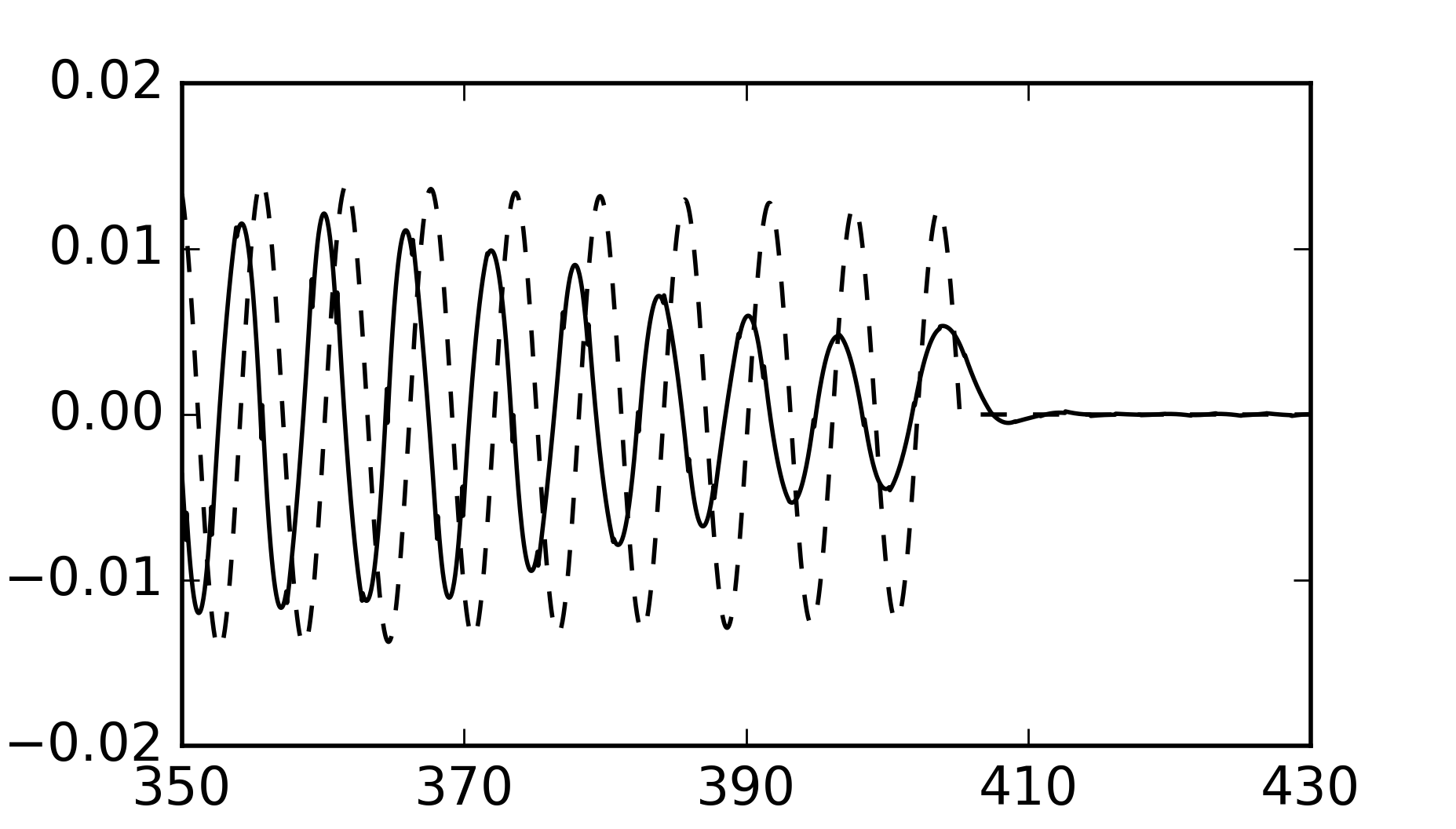}
 \includegraphics[width=0.8\textwidth]{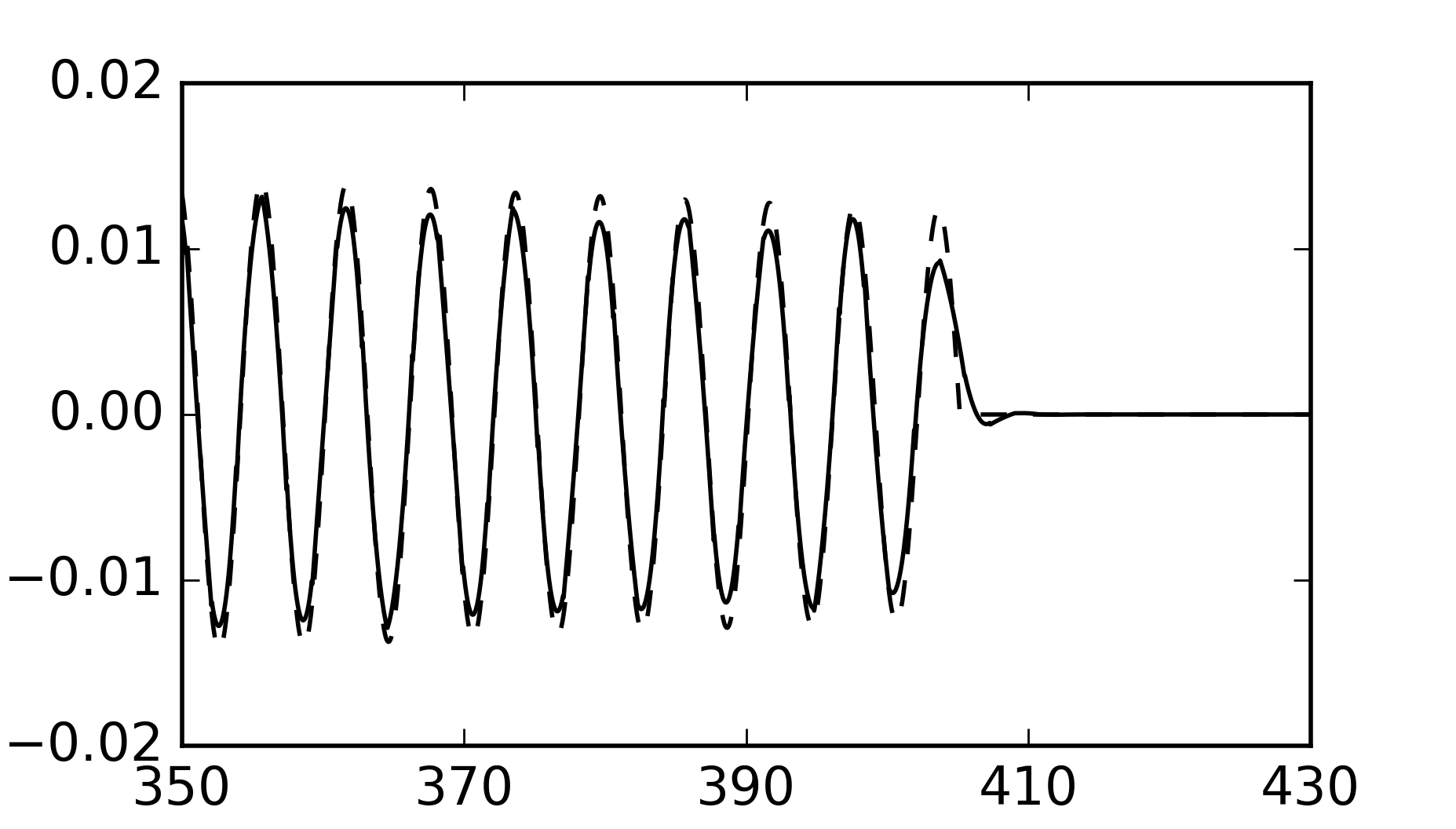}
\end{figure}

\subsection*{Example 4.7: 2D advection with periodic boundary condition}
We consider the following advection equation
\begin{align}
\label{eq-ad2}
 u_t + u_x +u_y = 0
\end{align}
on a unit square $\Omega=[0,1]\times [0,1]$ with 
initial condition $u(x,0) = \sin(2\pi (x+y))$, and a periodic boundary condition.
The exact solution is 
\begin{align*}
u(x, t) = \sin(2\pi (x+y-2t)).
\end{align*}

We use the three DG methods, (U) for upwinding flux, (C) for central flux, and (A) for the 
new method in section \ref{sec:adv2d}.
% In order to reduce time error,
% we take the 6-stage 6th order Lax-Wendroff time stepping \eqref{rk} with 
% $r = 6$.
% The time step size is taken be to 
% be $\Delta t = C\!F\!L\,h$ with  
% $C\!F\!L = 0.1$ for $P^1$, $C\!F\!L = 0.05$ for $P^2$, $C\!F\!L = 0.03$ for $P^3$.
The methods are tested on both nonuniform rectangular meshes and 
unstructured triangular meshes; see Figure \ref{fig:mesh2d} for a coarse mesh.
% \begin{itemize}
%  \item [(M1)] A nonuniform rectangular mesh which is a $10\%$ random perturbation of the 
% uniform square mesh.
%  \item [(M2)] A non-rectangular quadrilateral mesh which is a $10\%$ random perturbation of the 
% uniform square mesh.
%  \item [(M3)] An unstructured triangular mesh.
% \end{itemize}
\begin{figure}
 \caption{The coarse meshes. Left:  nonuniform rectangular mesh.
 Right: unstructured triangular mesh.}
\label{fig:mesh2d}
  \includegraphics[width=.49\textwidth]{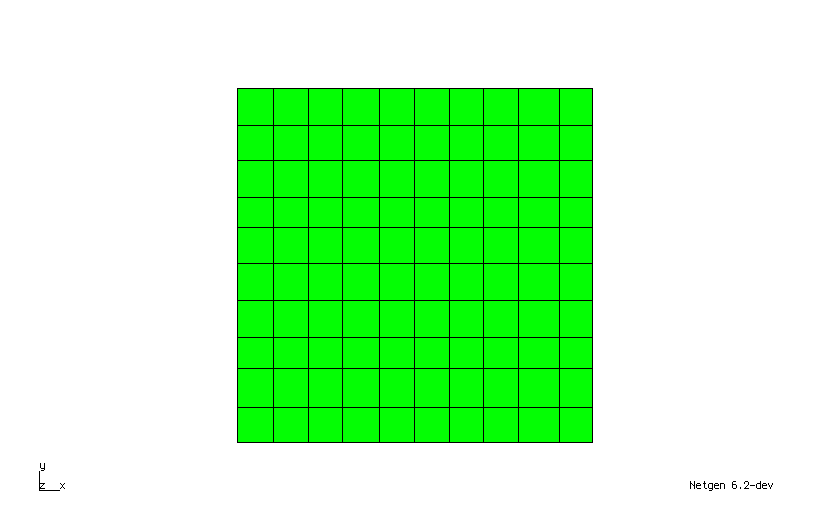}
 \includegraphics[width=.49\textwidth]{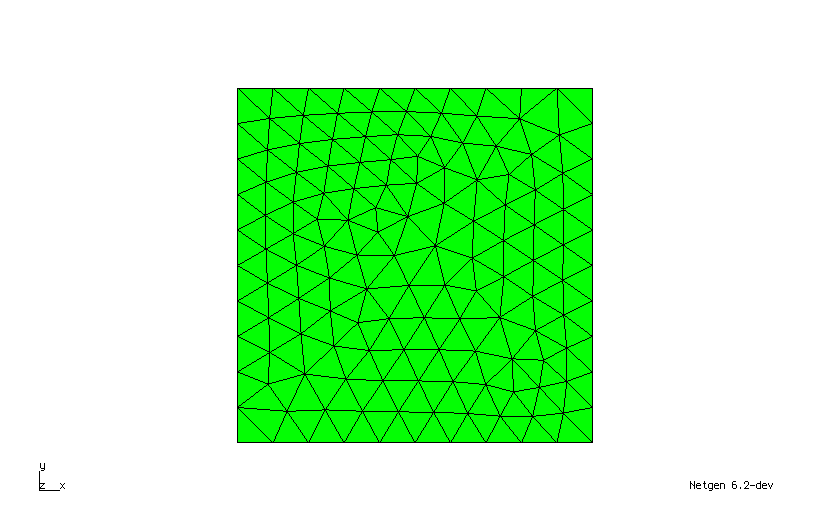}
\end{figure}
Table \ref{table:adv2d1}
lists the numerical errors and their orders for the above three DG methods at $T=0.1$
on the rectangular meshes. And Table \ref{table:adv2d3}
lists the errors and orders on the triangular meshes.
We use $Q^k/P^k$ polynomials with $1\le k\le 3$ for rectangular/triangular meshes.

From the tables, we observe 
optimal $(k+1)$th order of accuracy for both the variable $u_h$ (which approximate the solution $u$)
and $\phi_h$ (which approximate the {\it zero} function)
for the new energy-conserving DG method (A). 
The optimal convergence on rectangular meshes is 
understood; see Remark \ref{rk-error}. 
% supported by 
% our convergence result in Theorem \ref{thm:sys2dX}. 
But we do not have a theoretical proof 
for such optimality on the triangular meshes.
Similar to the 1D case in Example 4.1, 
the absolute value of the error for method (A) is slightly smaller than
the optimal-convergent upwinding DG method (U) 
for all polynomial degree.
We observe suboptimal convergence of order $k$ for the  method (C) for odd polynomial 
degree on both meshes. 
We also observe optimal convergence rate $3$ for the method (C) using $P^2$ space on the triangular meshes, 
but a suboptimal convergence rate of about $2.5$
on the nonuniform rectangular mesh.

\begin{table}[htbp]
\caption{\label{table:adv2d1} 
The $L^2$-errors
and orders for Example 4.7 for the upwinding DG method (U), the central DG method (C), and the 
new DG method (A) on a non-uniform rectangular mesh of $N\times N$ cells. 
$T =0.1$. 
} \centering
\bigskip
\begin{tabular}{|c|c|cc|cc|cc|cc|}
\hline
 & &\multicolumn{2}{|c|}{(U)}&\multicolumn{2}{c|}{(C)}&\multicolumn{4}{c|}{(A)}\\
 \hline
%  \cline{3-4}   \cline{6-7} \cline{9-12}
 &  {$N$}      & $\|u-u_h\|$ & Order & $\|u-u_h\|$ & Order &  $\|u-u_h\|$ & Order 
 &  $\|\phi_h\|$  & Order\\
\hline
\multirow{4}{*}{$P^1$}
 & 10 & 2.40e-02  &  -0.00 & 4.29e-02  &  -0.00 & 1.62e-02  &  -0.00 & 2.39e-02  &  -0.00 \\ 
 & 20 & 6.15e-03  &  1.97 & 1.86e-02  &  1.21 & 3.85e-03  &  2.07 & 5.35e-03  &  2.16 \\ 
 & 40 & 1.55e-03  &  1.99 & 8.88e-03  &  1.06 & 9.52e-04  &  2.02 & 1.26e-03  &  2.09 \\ 
\hline
\multirow{4}{*}{$P^2$}
 & 10 & 1.28e-03  &  -0.00 & 1.94e-03  &  -0.00 & 8.47e-04  &  -0.00 & 1.30e-03  &  -0.00 \\ 
 & 20 & 1.61e-04  &  2.99 & 4.48e-04  &  2.12 & 1.05e-04  &  3.02 & 1.35e-04  &  3.27 \\ 
 & 40 & 2.02e-05  &  3.00 & 8.04e-05  &  2.48 & 1.30e-05  &  3.01 & 1.53e-05  &  3.14 \\ 
\hline
\multirow{3}{*}{$P^3$}
 & 10 & 5.10e-05  &  -0.00 & 1.13e-04  &  -0.00 & 3.41e-05  &  -0.00 & 4.11e-05  &  -0.00 \\ 
 & 20 & 3.25e-06  &  3.97 & 1.43e-05  &  2.98 & 2.16e-06  &  3.98 & 2.51e-06  &  4.03 \\ 
 & 40 & 2.03e-07  &  4.00 & 1.77e-06  &  3.01 & 1.35e-07  &  4.00 & 1.55e-07  &  4.02 \\ 
\hline
\end{tabular}
\end{table}

\begin{table}[htbp]
\caption{\label{table:adv2d3} 
The $L^2$-errors
and orders for Example 4.7 for the upwinding DG method (U), the central DG method (C), and the 
new DG method (A) on a unstructured triangular mesh with mesh size $h = 1/N$. 
$T =0.1$. 
} \centering
\bigskip
\begin{tabular}{|c|c|cc|cc|cc|cc|}
\hline
 & &\multicolumn{2}{|c|}{(U)}&\multicolumn{2}{c|}{(C)}&\multicolumn{4}{c|}{(A)}\\
 \hline
%  \cline{3-4}   \cline{6-7} \cline{9-12}
 &  {$N$}      & $\|u-u_h\|$ & Order & $\|u-u_h\|$ & Order &  $\|u-u_h\|$ & Order 
 &  $\|\phi_h\|$  & Order\\
% \hline
% \multirow{4}{*}{$P^0$}
%  & 10 & 3.53e-01  &  -0.00 & 3.79e-01  &  -0.00 & 1.85e-01  &  -0.00 & 9.24e-02  &  -0.00 \\ 
%  & 20 & 1.91e-01  &  0.88 & 1.59e-01  &  1.26 & 9.18e-02  &  1.01 & 4.57e-02  &  1.02 \\ 
%  & 40 & 1.05e-01  &  0.86 & 1.01e-01  &  0.65 & 4.59e-02  &  1.00 & 1.43e-02  &  1.67 \\ 
%  & 80 & 5.08e-02  &  1.05 & 5.32e-02  &  0.93 & 2.29e-02  &  1.00 & 1.10e-02  &  0.38 \\ 
\hline
\multirow{4}{*}{$P^1$}
 & 10 & 1.78e-02  &  -0.00 & 2.47e-01  &  -0.00 & 1.26e-02  &  -0.00 & 1.36e-02  &  -0.00 \\ 
 & 20 & 4.63e-03  &  1.94 & 1.15e-01  &  1.10 & 3.16e-03  &  2.00 & 3.52e-03  &  1.95 \\ 
 & 40 & 1.14e-03  &  2.02 & 5.20e-02  &  1.15 & 7.76e-04  &  2.02 & 8.46e-04  &  2.06 \\ 
%  & 80 & 3.93e-04  &  2.00 & 3.00e-02  &  0.59 & 2.45e-04  &  2.01 & 3.15e-04  &  2.03 \\ 
\hline
\multirow{4}{*}{$P^2$}
 & 10 & 1.21e-03  &  -0.00 & 4.71e-03  &  -0.00 & 8.35e-04  &  -0.00 & 9.12e-04  &  -0.00 \\ 
 & 20 & 1.50e-04  &  3.01 & 5.19e-04  &  3.18 & 1.02e-04  &  3.03 & 1.09e-04  &  3.06 \\ 
 & 40 & 1.83e-05  &  3.04 & 6.24e-05  &  3.05 & 1.24e-05  &  3.04 & 1.34e-05  &  3.03 \\ 
%  & 80 & 2.65e-06  &  3.03 & 2.36e-05  &  2.57 & 1.75e-06  &  3.03 & 2.10e-06  &  3.00 \\ 
\hline
\multirow{3}{*}{$P^3$}
 & 10 & 6.45e-05  &  -0.00 & 1.06e-03  &  -0.00 & 4.40e-05  &  -0.00 & 4.67e-05  &  -0.00 \\ 
 & 20 & 4.27e-06  &  3.92 & 1.02e-04  &  3.37 & 2.93e-06  &  3.91 & 3.13e-06  &  3.90 \\ 
 & 40 & 2.36e-07  &  4.18 & 1.17e-05  &  3.13 & 1.62e-07  &  4.17 & 1.72e-07  &  4.18 \\ 
\hline
\end{tabular}
\end{table}

\subsection*{Example 4.8: 2D acoustics with periodic boundary condition, subsonic background velocity}
We consider the acoustic equations \eqref{ac2d} with  coefficients 
$\rho_0=K_0=1$, and $u_0=.5$, $v_0 = 0$, i.e.,
\begin{align*}
% \label{eq2-ac}
 p_t + .5p_x + u_x +v_y =& 0,\\
 u_t + .5u_x + p_x =& 0,\\
 v_t + .5v_x + p_y =& 0,
\end{align*}
on a unit square $\Omega=[0,1]\times [0,1]$ with 
initial condition $p(x,y, 0) = \sin(2\pi (x+y)), u(x,0)=v(x,0) = 0$, and a periodic boundary condition.
The exact solution is the following plane wave solution
\begin{align*}
p(x, y, t) = &\frac12\sin(2\pi (x+y-(\sqrt{2}+.5)t))+\frac12\sin(2\pi (x+y+(\sqrt{2}-.5)t)),\\
u(x, y, t) = &\frac{\sqrt{2}}{4}\sin(2\pi (x+y-(\sqrt{2}+.5)t))-\frac{\sqrt{2}}{4}\sin(2\pi (x+y+(\sqrt{2}-.5)t)),\\
v(x, y, t) = &\frac{\sqrt{2}}{4}\sin(2\pi (x+y-(\sqrt{2}+.5)t))-\frac{\sqrt{2}}{4}\sin(2\pi (x+y+(\sqrt{2}-.5)t)).
\end{align*}

We present numerical results with the upwinding DG method (U), central DG method (C), 
the DG method (A) for the augmented 4-components system \eqref{ac2dX}
in section \ref{sec:ac2d}, and also the
DG method (A-Double) for the augmented 6-components system \eqref{sys:2daux}.
The method (A-Double) is only tested for triangular meshes.
% Note that the method (A) augment the 
% above equations with the auxiliary {\it zero} function that solve 
% \[
%  \phi_t -.5\phi_x = 0.
% \]
% On triangular meshes, we also present the numerical results using the 
Table \ref{table:ac2d1} and Table \ref{table:ac2d3}
lists the numerical errors and their orders for the above DG methods at $T=0.1$
on the non-uniform rectangular meshes, and triangular meshes, respectively.
% And Table \ref{table:ac2d3}
% list the errors and orders on the  triangular meshes.
Again, we use $Q^k/P^k$ polynomials with $1\le k\le 3$. 
% We use the following numerical flux for the method (A) on 
% general quadrilateral or triangular meshes

From the tables, we observe 
optimal $(k+1)$th order of accuracy for all the variable $p_h, u_h, v_h,$ and $\phi_h$ 
for the 4-components energy-conserving DG method (A) on both types of meshes. 
We also observe optimal convergence for the 6-components, energy conserving DG method (A-Double)
on triangular meshes, with a smaller absolute error compared with the method (A).
The optimal convergence on rectangular meshes is 
understood; see Remark \ref{rk-error}. 
But we do not have a theoretical proof 
for such optimality on the triangular meshes for both methods.
% 
% The optimal convergence on rectangular meshes is supported by 
% our convergence result in Theorem \ref{thm:sys2dX}. We do not have a theoretical proof 
% for such optimality on the general quadrilateral or triangular meshes.
We also observe optimal convergence for the method (U), but
suboptimal convergence for the  method (C) for all polynomial degrees.

\begin{table}[htbp]
\caption{\label{table:ac2d1} 
The $L^2$-errors
and orders for Example 4.8 for the upwinding DG method (U), the central DG method (C), and the 
new DG method (A) on a non-uniform rectangular mesh with $N\times N$ cells. 
$T =0.1$. 
} \centering
\bigskip
\begin{tabular}{|c|c|cc|cc|cc|cc|}
\cline{1-8}
 & &\multicolumn{6}{|c|}{(U)}\\
 \cline{1-8}
 &  {$N$}      & $\|p-p_h\|$ & Order & $\|u-u_h\|$ & Order &  $\|v-v_h\|$ & Order\\
 \cline{1-8}
\multirow{4}{*}{$P^1$}
 & 10 & 1.22e-02  &  -0.00 & 1.40e-02  &  -0.00 & 1.40e-02  &  -0.00 \\ 
 & 20 & 3.34e-03  &  1.86 & 3.30e-03  &  2.08 & 3.57e-03  &  1.97 \\ 
 & 40 & 8.70e-04  &  1.94 & 7.98e-04  &  2.05 & 8.94e-04  &  2.00 \\ 
 \cline{1-8}
\multirow{4}{*}{$P^2$}
 & 10 & 6.97e-04  &  -0.00 & 6.93e-04  &  -0.00 & 7.43e-04  &  -0.00 \\ 
 & 20 & 9.10e-05  &  2.94 & 8.40e-05  &  3.04 & 9.31e-05  &  3.00 \\ 
 & 40 & 1.16e-05  &  2.98 & 1.03e-05  &  3.03 & 1.16e-05  &  3.01 \\ 
 \cline{1-8}
\multirow{3}{*}{$P^3$}
 & 10 & 2.84e-05  &  -0.00 & 2.74e-05  &  -0.00 & 2.94e-05  &  -0.00 \\ 
 & 20 & 1.85e-06  &  3.94 & 1.68e-06  &  4.03 & 1.88e-06  &  3.97 \\ 
 & 40 & 1.17e-07  &  3.98 & 1.04e-07  &  4.01 & 1.17e-07  &  4.01 \\ 
 \cline{1-8}
 & &\multicolumn{6}{|c|}{(C)}\\
 \cline{1-8}
 &  {$N$}      & $\|p-p_h\|$ & Order & $\|u-u_h\|$ & Order &  $\|v-v_h\|$ & Order\\
 \cline{1-8}
\multirow{4}{*}{$P^1$}
 & 10 & 2.93e-02  &  -0.00 & 8.09e-02  &  -0.00 & 3.52e-02  &  -0.00 \\ 
 & 20 & 1.39e-02  &  1.08 & 4.10e-02  &  0.98 & 1.71e-02  &  1.04 \\ 
 & 40 & 6.86e-03  &  1.02 & 2.05e-02  &  1.00 & 8.49e-03  &  1.01 \\ 
 \cline{1-8}
\multirow{4}{*}{$P^2$}
 & 10 & 6.66e-04  &  -0.00 & 8.29e-03  &  -0.00 & 9.82e-04  &  -0.00 \\ 
 & 20 & 1.10e-04  &  2.60 & 1.54e-03  &  2.43 & 4.73e-04  &  1.05 \\ 
 & 40 & 1.76e-05  &  2.64 & 7.86e-05  &  4.29 & 1.18e-04  &  2.01 \\ 
\cline{1-8}
\multirow{3}{*}{$P^3$}
 & 10 & 6.52e-05  &  -0.00 & 2.62e-04  &  -0.00 & 7.97e-05  &  -0.00 \\ 
 & 20 & 8.10e-06  &  3.01 & 2.72e-05  &  3.27 & 8.55e-06  &  3.22 \\ 
 & 40 & 1.01e-06  &  3.01 & 2.68e-06  &  3.34 & 1.05e-06  &  3.03 \\ 
\cline{1-10}
 & &\multicolumn{8}{|c|}{(A)}\\
 \cline{1-10}
 &  {$N$}      & $\|p-p_h\|$ & Order & $\|u-u_h\|$ & Order &  $\|v-v_h\|$ & Order
 &  $\|\phi_h\|$ & Order\\
 \cline{1-10}
\multirow{4}{*}{$P^1$}
 & 10 & 2.14e-02  &  -0.00 & 3.24e-02  &  -0.00 & 1.40e-02  &  -0.00 & 4.39e-03  &  -0.00 \\ 
 & 20 & 4.45e-03  &  2.27 & 8.10e-03  &  2.00 & 3.11e-03  &  2.17 & 1.80e-03  &  1.29 \\ 
 & 40 & 1.05e-03  &  2.08 & 2.00e-03  &  2.01 & 7.42e-04  &  2.07 & 4.64e-04  &  1.95 \\ 
 \cline{1-10}
\multirow{4}{*}{$P^2$}
 & 10 & 9.88e-04  &  -0.00 & 1.66e-03  &  -0.00 & 7.31e-04  &  -0.00 & 2.95e-04  &  -0.00 \\ 
 & 20 & 1.11e-04  &  3.16 & 2.01e-04  &  3.04 & 8.09e-05  &  3.18 & 5.15e-05  &  2.52 \\ 
 & 40 & 1.36e-05  &  3.03 & 2.53e-05  &  2.99 & 9.51e-06  &  3.09 & 6.05e-06  &  3.09 \\ 
 \cline{1-10}
\multirow{3}{*}{$P^3$}
 & 10 & 3.63e-05  &  -0.00 & 6.39e-05  &  -0.00 & 2.61e-05  &  -0.00 & 1.51e-05  &  -0.00 \\ 
 & 20 & 2.25e-06  &  4.02 & 4.06e-06  &  3.98 & 1.58e-06  &  4.04 & 8.94e-07  &  4.08 \\ 
 & 40 & 1.38e-07  &  4.02 & 2.54e-07  &  4.00 & 9.70e-08  &  4.03 & 5.83e-08  &  3.94 \\ 
\hline
\end{tabular}
\end{table}

\begin{table}[htbp]
\caption{\label{table:ac2d3} 
The $L^2$-errors
and orders for Example 4.8 for the upwinding DG method (U), the central DG method (C), the 
new DG method (A) (4-components), and the DG method (A-Double) (6-components)
on a unstructured triangular mesh with mesh size $h = 1/N$. 
The scalar $\phi_h$ is the auxiliary variable for method (A), while the (3-components)
vector $\bld \phi_h$ is the auxiliary variable for method (A-Double).
$T =0.1$.
} \centering
\bigskip
\begin{tabular}{|c|c|cc|cc|cc|cc|}
\cline{1-8}
 & &\multicolumn{6}{|c|}{(U)}\\
 \cline{1-8}
 &  {$N$}      & $\|p-p_h\|$ & Order & $\|u-u_h\|$ & Order &  $\|v-v_h\|$ & Order\\
% \hline
% \multirow{4}{*}{$P^0$}
%  & 10 & 3.53e-01  &  -0.00 & 3.79e-01  &  -0.00 & 1.85e-01  &  -0.00 & 9.24e-02  &  -0.00 \\ 
%  & 20 & 1.91e-01  &  0.88 & 1.59e-01  &  1.26 & 9.18e-02  &  1.01 & 4.57e-02  &  1.02 \\ 
%  & 40 & 1.05e-01  &  0.86 & 1.01e-01  &  0.65 & 4.59e-02  &  1.00 & 1.43e-02  &  1.67 \\ 
%  & 80 & 5.08e-02  &  1.05 & 5.32e-02  &  0.93 & 2.29e-02  &  1.00 & 1.10e-02  &  0.38 \\ 
 \cline{1-8}
\multirow{4}{*}{$P^1$}
 & 10 & 9.88e-03  &  -0.00 & 1.16e-02  &  -0.00 & 1.00e-02  &  -0.00 \\ 
 & 20 & 2.65e-03  &  1.90 & 2.79e-03  &  2.05 & 2.60e-03  &  1.95 \\ 
 & 40 & 6.80e-04  &  1.96 & 6.85e-04  &  2.03 & 6.43e-04  &  2.02 \\ 
 \cline{1-8}
\multirow{4}{*}{$P^2$}
 & 10 & 6.94e-04  &  -0.00 & 7.94e-04  &  -0.00 & 6.72e-04  &  -0.00 \\ 
 & 20 & 8.81e-05  &  2.98 & 1.02e-04  &  2.97 & 8.34e-05  &  3.01 \\ 
 & 40 & 1.07e-05  &  3.05 & 1.18e-05  &  3.11 & 9.94e-06  &  3.07 \\ 
 \cline{1-8}
\multirow{3}{*}{$P^3$}
 & 10 & 3.75e-05  &  -0.00 & 4.29e-05  &  -0.00 & 3.63e-05  &  -0.00 \\ 
 & 20 & 2.38e-06  &  3.98 & 2.65e-06  &  4.02 & 2.32e-06  &  3.97 \\ 
 & 40 & 1.40e-07  &  4.08 & 1.50e-07  &  4.15 & 1.31e-07  &  4.14 \\ 
 \cline{1-8}
 & &\multicolumn{6}{|c|}{(C)}\\
 \cline{1-8}
 &  {$N$}      & $\|p-p_h\|$ & Order & $\|u-u_h\|$ & Order &  $\|v-v_h\|$ & Order\\
 \cline{1-8}
\multirow{4}{*}{$P^1$}
 & 10 & 3.11e-02  &  -0.00 & 2.10e-01  &  -0.00 & 6.07e-02  &  -0.00 \\ 
 & 20 & 1.58e-02  &  0.98 & 1.09e-01  &  0.95 & 2.84e-02  &  1.10 \\ 
 & 40 & 7.58e-03  &  1.06 & 5.66e-02  &  0.94 & 1.33e-02  &  1.09 \\ 
 \cline{1-8}
\multirow{4}{*}{$P^2$}
 & 10 & 2.83e-03  &  -0.00 & 6.84e-03  &  -0.00 & 5.28e-03  &  -0.00 \\ 
 & 20 & 6.22e-04  &  2.18 & 1.34e-03  &  2.36 & 1.25e-03  &  2.08 \\ 
 & 40 & 1.50e-04  &  2.05 & 3.03e-04  &  2.14 & 2.65e-04  &  2.23 \\ 
\cline{1-8}
\multirow{3}{*}{$P^3$}
 & 10 & 1.62e-04  &  -0.00 & 8.24e-04  &  -0.00 & 2.94e-04  &  -0.00 \\ 
 & 20 & 1.82e-05  &  3.15 & 9.54e-05  &  3.11 & 3.24e-05  &  3.18 \\ 
 & 40 & 2.28e-06  &  3.00 & 8.38e-06  &  3.51 & 3.99e-06  &  3.02 \\ 
\cline{1-10}
 & &\multicolumn{8}{|c|}{(A)}\\
 \cline{1-10}
 &  {$N$}      & $\|p-p_h\|$ & Order & $\|u-u_h\|$ & Order &  $\|v-v_h\|$ & Order
 &  $\|\phi_h\|$ & Order\\
 \cline{1-10}
\multirow{4}{*}{$P^1$}
 & 10 & 1.51e-02  &  -0.00 & 1.73e-02  &  -0.00 & 1.35e-02  &  -0.00 & 1.48e-02  &  -0.00 \\ 
 & 20 & 3.38e-03  &  2.16 & 4.55e-03  &  1.93 & 3.44e-03  &  1.97 & 3.58e-03  &  2.05 \\ 
 & 40 & 8.26e-04  &  2.03 & 1.09e-03  &  2.07 & 7.29e-04  &  2.24 & 6.67e-04  &  2.42 \\ 
 \cline{1-10}
\multirow{4}{*}{$P^2$}
 & 10 & 9.38e-04  &  -0.00 & 1.35e-03  &  -0.00 & 1.02e-03  &  -0.00 & 1.10e-03  &  -0.00 \\ 
 & 20 & 1.12e-04  &  3.06 & 1.76e-04  &  2.94 & 1.11e-04  &  3.20 & 1.31e-04  &  3.07 \\ 
 & 40 & 1.30e-05  &  3.11 & 1.78e-05  &  3.30 & 1.13e-05  &  3.30 & 1.29e-05  &  3.34 \\ 
 \cline{1-10}
\multirow{3}{*}{$P^3$}
 & 10 & 5.07e-05  &  -0.00 & 6.36e-05  &  -0.00 & 5.06e-05  &  -0.00 & 6.19e-05  &  -0.00 \\ 
 & 20 & 3.04e-06  &  4.06 & 4.19e-06  &  3.93 & 2.95e-06  &  4.10 & 3.09e-06  &  4.32 \\ 
 & 40 & 1.68e-07  &  4.18 & 2.29e-07  &  4.19 & 1.51e-07  &  4.29 & 1.49e-07  &  4.38 \\ 
\cline{1-10}
 & &\multicolumn{8}{|c|}{(A-Double)}\\
 \cline{1-10}
 &  {$N$}      & $\|p-p_h\|$ & Order & $\|u-u_h\|$ & Order &  $\|v-v_h\|$ & Order
 &  $\|\bld \phi_h\|$ & Order\\
 \cline{1-10}
\multirow{4}{*}{$P^1$}
 & 10 & 7.76e-03  &  -0.00 & 7.94e-03  &  -0.00 & 6.96e-03  &  -0.00 & 1.41e-02  &  -0.00\\
 & 20 & 1.93e-03  &  2.01 & 2.09e-03  &  1.92 & 1.77e-03  &  1.98 & 3.43e-03  &  2.04\\
 & 40 & 4.79e-04  &  2.01 & 4.98e-04  &  2.07 & 4.35e-04  &  2.02 & 8.61e-04  &  1.99\\
 \cline{1-10}
\multirow{4}{*}{$P^2$}
 & 10 & 5.08e-04  &  -0.00 & 5.81e-04  &  -0.00 & 5.11e-04  &  -0.00 & 8.80e-04  &  -0.00\\
 & 20 & 6.71e-05  &  2.92 & 7.03e-05  &  3.05 & 6.02e-05  &  3.09 & 1.12e-04  &  2.97\\
 & 40 & 7.99e-06  &  3.07 & 8.28e-06  &  3.09 & 7.28e-06  &  3.05 & 1.34e-05  &  3.07\\
 \cline{1-10}
\multirow{3}{*}{$P^3$}
 & 10 & 2.85e-05  &  -0.00 & 3.03e-05  &  -0.00 & 2.48e-05  &  -0.00 & 5.17e-05  &  -0.00\\
 & 20 & 1.70e-06  &  4.07 & 1.98e-06  &  3.94 & 1.64e-06  &  3.92 & 3.12e-06  &  4.05\\
 & 40 & 9.87e-08  &  4.10 & 1.06e-07  &  4.22 & 9.24e-08  &  4.15 & 1.79e-07  &  4.12\\
\hline
\end{tabular}
\end{table}

\subsection*{Example 4.9: 2D acoustics, zero background velocity}
We consider the acoustic equations \eqref{ac2d} with  coefficients 
$\rho_0=K_0=1$, and $u_0=0$, $v_0 = 0$, i.e.,
\begin{align*}
% \label{eq2-ac}
 p_t + u_x +v_y =& 0,\\
 u_t +  p_x =& 0,\\
 v_t +  p_y =& 0,
\end{align*}
on a unit square $\Omega=[0,1]\times [0,1]$ with 
initial condition $p(x,y, 0) = \sin(2\pi (x+y)), u(x,0)=v(x,0) = 0$, and a periodic boundary condition.
The exact solution is the following plane wave solution
\begin{align*}
p(x, y, t) = &\frac12\sin(2\pi (x+y-(\sqrt{2})t))+\frac12\sin(2\pi (x+y+(\sqrt{2})t)),\\
u(x, y, t) = &\frac{\sqrt{2}}{4}\sin(2\pi (x+y-(\sqrt{2})t))-\frac{\sqrt{2}}{4}\sin(2\pi (x+y+(\sqrt{2})t)),\\
v(x, y, t) = &\frac{\sqrt{2}}{4}\sin(2\pi (x+y-(\sqrt{2})t))-\frac{\sqrt{2}}{4}\sin(2\pi (x+y+(\sqrt{2})t)).
\end{align*}
Similar to the previous example, 
we present numerical results with the upwinding DG method (U), central DG method (C), and 
the 3-components DG method (A) with an alternating flux in section \ref{sec:ac2d}. 
We also present numerical results with the DG method (A-Double) 
for the augmented 6-components system \eqref{sys:2daux} on triangular meshes.

Note that the DG method with an alternating numerical flux (A)
on triangular meshes was observed to be suboptimal in \cite{LiShiShu18} for the Maxwell's equation.

Table \ref{table:ac2dn1} and Table \ref{table:ac2dn3}
lists the numerical errors and their orders with the above DG methods at $T=0.1$
on rectangular, and triangular meshes, respectively.
We use $Q^k/P^k$ polynomials with $1\le k\le 3$.

From the tables, while we still observe 
optimal $(k+1)$th order of accuracy for all the variable 
for method (A) on rectangular meshes, which is 
in agreement with Remark \ref{rk-error}, 
we only get suboptimal convergence order of $k$ for the velocity variables $u_h$ and $v_h$ 
on triangular meshes. 
On the other hand, the method (A-Double) is still optimally convergent on triangular meshes.
% On the other hand, optimal convergence of the method (U) is observed on both types of meshes.
% The convergence of the method (C) is worse than the method (A).
% We also observe optimal convergence for the DG method with an upwinding flux, but
% We also observe 
% suboptimal convergence for the  method (C) on both types of meshes, with a larger absolute error 
% than that for method (A) on triangular meshes.
% three types of meshes.
Similar as the previous example, the method (U) is optimally convergent, but (C) is suboptimal.

\begin{table}[htbp]
\caption{\label{table:ac2dn1} 
The $L^2$-errors
and orders for Example 4.9 for the upwinding DG method (U), the central DG method (C), and the 
new DG method (A) on a non-uniform rectangular mesh with $N\times N$ cells. 
$T =0.1$. 
} \centering
\bigskip
\begin{tabular}{|c|c|cc|cc|cc|}
\cline{1-8}
 & &\multicolumn{6}{|c|}{(U)}\\
 \cline{1-8}
 &  {$N$}      & $\|p-p_h\|$ & Order & $\|u-u_h\|$ & Order &  $\|v-v_h\|$ & Order\\
 \cline{1-8}
\multirow{4}{*}{$P^1$}
 & 10 & 1.25e-02  &  -0.00 & 1.37e-02  &  -0.00 & 1.37e-02  &  -0.00 \\ 
 & 20 & 3.37e-03  &  1.89 & 3.21e-03  &  2.10 & 3.22e-03  &  2.09 \\ 
 & 40 & 8.71e-04  &  1.95 & 7.77e-04  &  2.05 & 7.78e-04  &  2.05 \\ 
 \cline{1-8}
\multirow{4}{*}{$P^2$}
 & 10 & 7.08e-04  &  -0.00 & 6.92e-04  &  -0.00 & 6.91e-04  &  -0.00 \\ 
 & 20 & 9.12e-05  &  2.96 & 8.20e-05  &  3.08 & 8.24e-05  &  3.07 \\ 
 & 40 & 1.16e-05  &  2.98 & 1.01e-05  &  3.02 & 1.01e-05  &  3.03 \\ 
 \cline{1-8}
\multirow{3}{*}{$P^3$}
 & 10 & 2.86e-05  &  -0.00 & 2.68e-05  &  -0.00 & 2.66e-05  &  -0.00 \\ 
 & 20 & 1.86e-06  &  3.94 & 1.64e-06  &  4.03 & 1.65e-06  &  4.01 \\ 
 & 40 & 1.17e-07  &  3.98 & 1.02e-07  &  4.01 & 1.02e-07  &  4.02 \\ 
 \cline{1-8}
 & &\multicolumn{6}{|c|}{(C)}\\
 \cline{1-8}
 &  {$N$}      & $\|p-p_h\|$ & Order & $\|u-u_h\|$ & Order &  $\|v-v_h\|$ & Order\\
 \cline{1-8}
\multirow{4}{*}{$P^1$}
 & 10 & 2.99e-02  &  -0.00 & 2.73e-01  &  -0.00 & 2.73e-01  &  -0.00 \\ 
 & 20 & 1.41e-02  &  1.08 & 1.39e-01  &  0.98 & 1.39e-01  &  0.98 \\ 
 & 40 & 6.95e-03  &  1.02 & 6.96e-02  &  1.00 & 6.96e-02  &  0.99 \\ 
 \cline{1-8}
\multirow{4}{*}{$P^2$}
 & 10 & 6.54e-04  &  -0.00 & 8.70e-03  &  -0.00 & 2.25e-03  &  -0.00 \\ 
 & 20 & 1.08e-04  &  2.60 & 1.59e-03  &  2.45 & 1.51e-03  &  0.58 \\ 
 & 40 & 1.77e-05  &  2.61 & 1.17e-04  &  3.76 & 3.77e-04  &  2.00 \\ 
\cline{1-8}
\multirow{3}{*}{$P^3$}
 & 10 & 5.84e-05  &  -0.00 & 5.79e-04  &  -0.00 & 5.57e-04  &  -0.00 \\ 
 & 20 & 7.22e-06  &  3.02 & 7.13e-05  &  3.02 & 6.90e-05  &  3.01 \\ 
 & 40 & 9.06e-07  &  2.99 & 8.68e-06  &  3.04 & 8.67e-06  &  2.99 \\ 
\cline{1-8}
 & &\multicolumn{6}{|c|}{(A)}\\
 \cline{1-8}
 &  {$N$}      & $\|p-p_h\|$ & Order & $\|u-u_h\|$ & Order &  $\|v-v_h\|$ & Order\\
 \cline{1-8}
\multirow{4}{*}{$P^1$}
 & 10 & 1.52e-02  &  -0.00 & 9.27e-02  &  -0.00 & 9.25e-02  &  -0.00  \\ 
 & 20 & 3.88e-03  &  1.97 & 2.37e-02  &  1.97 & 2.36e-02  &  1.97 \\ 
 & 40 & 9.74e-04  &  1.99 & 5.95e-03  &  2.00 & 5.94e-03  &  1.99  \\ 
 \cline{1-8}
\multirow{4}{*}{$P^2$}
 & 10 & 7.90e-04  &  -0.00 & 4.71e-03  &  -0.00 & 4.72e-03  &  -0.00  \\ 
 & 20 & 1.01e-04  &  2.97 & 6.01e-04  &  2.97 & 5.96e-04  &  2.98  \\ 
 & 40 & 1.27e-05  &  2.99 & 7.49e-05  &  3.00 & 7.49e-05  &  2.99  \\ 
 \cline{1-8}
\multirow{3}{*}{$P^3$}
 & 10 & 3.15e-05  &  -0.00 & 1.85e-04  &  -0.00 & 1.87e-04  &  -0.00 \\ 
 & 20 & 2.03e-06  &  3.95 & 1.19e-05  &  3.95 & 1.18e-05  &  3.99  \\ 
 & 40 & 1.27e-07  &  4.00 & 7.43e-07  &  4.01 & 7.43e-07  &  3.99  \\ 
\hline
\end{tabular}
\end{table}

\begin{table}[htbp]
\caption{\label{table:ac2dn3} 
The $L^2$-errors
and orders for Example 4.9 for the upwinding DG method (U), the central DG method (C), the 
alternating flux DG method (A), and the DG method (A-Double)  (6-components)
on a unstructured triangular mesh with mesh size $h = 1/N$. 
$T =0.1$. 
} \centering
\bigskip
\begin{tabular}{|c|c|cc|cc|cc|}
\cline{1-8}
 & &\multicolumn{6}{|c|}{(U)}\\
 \cline{1-8}
 &  {$N$}      & $\|p-p_h\|$ & Order & $\|u-u_h\|$ & Order &  $\|v-v_h\|$ & Order\\
 \cline{1-8}
\multirow{4}{*}{$P^1$}
 & 10 & 9.24e-03  &  -0.00 & 1.90e-02  &  -0.00 & 2.16e-02  &  -0.00 \\ 
 & 20 & 2.41e-03  &  1.94 & 4.11e-03  &  2.21 & 4.23e-03  &  2.35 \\ 
 & 40 & 6.11e-04  &  1.98 & 8.68e-04  &  2.24 & 9.03e-04  &  2.23 \\ 
 \cline{1-8}
\multirow{4}{*}{$P^2$}
 & 10 & 6.29e-04  &  -0.00 & 8.04e-04  &  -0.00 & 8.15e-04  &  -0.00 \\ 
 & 20 & 8.20e-05  &  2.94 & 9.12e-05  &  3.14 & 9.17e-05  &  3.15 \\ 
 & 40 & 1.00e-05  &  3.04 & 1.03e-05  &  3.14 & 1.07e-05  &  3.10 \\ 
 \cline{1-8}
\multirow{3}{*}{$P^3$}
 & 10 & 3.44e-05  &  -0.00 & 3.97e-05  &  -0.00 & 3.83e-05  &  -0.00 \\ 
 & 20 & 2.12e-06  &  4.02 & 2.39e-06  &  4.05 & 2.39e-06  &  4.00 \\ 
 & 40 & 1.25e-07  &  4.09 & 1.33e-07  &  4.17 & 1.36e-07  &  4.14 \\ 
 \cline{1-8}
 & &\multicolumn{6}{|c|}{(C)}\\
 \cline{1-8}
 &  {$N$}      & $\|p-p_h\|$ & Order & $\|u-u_h\|$ & Order &  $\|v-v_h\|$ & Order\\
 \cline{1-8}
\multirow{4}{*}{$P^1$}
 & 10 & 1.49e-02  &  -0.00 & 6.14e-01  &  -0.00 & 6.12e-01  &  -0.00 \\ 
 & 20 & 3.88e-03  &  1.94 & 3.14e-01  &  0.97 & 3.17e-01  &  0.95 \\ 
 & 40 & 9.50e-04  &  2.03 & 1.55e-01  &  1.02 & 1.57e-01  &  1.02 \\ 
 \cline{1-8}
\multirow{4}{*}{$P^2$}
 & 10 & 7.09e-04  &  -0.00 & 3.61e-02  &  -0.00 & 3.62e-02  &  -0.00 \\ 
 & 20 & 8.62e-05  &  3.04 & 9.41e-03  &  1.94 & 9.58e-03  &  1.92 \\ 
 & 40 & 1.00e-05  &  3.10 & 2.19e-03  &  2.11 & 2.29e-03  &  2.06 \\ 
\cline{1-8}
\multirow{3}{*}{$P^3$}
 & 10 & 4.21e-05  &  -0.00 & 3.75e-03  &  -0.00 & 3.77e-03  &  -0.00 \\ 
 & 20 & 3.48e-06  &  3.60 & 4.71e-04  &  2.99 & 4.76e-04  &  2.99 \\ 
 & 40 & 3.74e-07  &  3.22 & 5.69e-05  &  3.05 & 5.77e-05  &  3.04 \\ 
\cline{1-8}
 & &\multicolumn{6}{|c|}{(A)}\\
 \cline{1-8}
 &  {$N$}      & $\|p-p_h\|$ & Order & $\|u-u_h\|$ & Order &  $\|v-v_h\|$ & Order
\\
 \cline{1-8}
\multirow{4}{*}{$P^1$}
 & 10 & 1.15e-02  &  -0.00 & 2.48e-01  &  -0.00 & 2.05e-01  &  -0.00 \\ 
 & 20 & 2.88e-03  &  1.99 & 1.40e-01  &  0.83 & 1.15e-01  &  0.83 \\ 
 & 40 & 7.15e-04  &  2.01 & 5.57e-02  &  1.33 & 6.66e-02  &  0.79 \\ 
 \cline{1-8}
\multirow{4}{*}{$P^2$}
 & 10 & 7.07e-04  &  -0.00 & 2.21e-02  &  -0.00 & 2.47e-02  &  -0.00 \\ 
 & 20 & 9.57e-05  &  2.88 & 6.53e-03  &  1.76 & 6.28e-03  &  1.97 \\ 
 & 40 & 1.14e-05  &  3.07 & 1.43e-03  &  2.19 & 1.65e-03  &  1.93 \\ 
 \cline{1-8}
\multirow{3}{*}{$P^3$}
 & 10 & 3.97e-05  &  -0.00 & 2.03e-03  &  -0.00 & 1.54e-03  &  -0.00 \\ 
 & 20 & 2.32e-06  &  4.09 & 2.62e-04  &  2.95 & 2.21e-04  &  2.80 \\ 
 & 40 & 1.38e-07  &  4.08 & 2.50e-05  &  3.39 & 2.95e-05  &  2.91 \\ 
 \cline{1-8}
 & &\multicolumn{6}{|c|}{(A-Double)}\\
 \cline{1-8}
 &  {$N$}      & $\|p-p_h\|$ & Order & $\|u-u_h\|$ & Order &  $\|v-v_h\|$ & Order
\\
 \cline{1-8}
\multirow{4}{*}{$P^1$}
 & 10 & 7.67e-03  &  -0.00 & 1.75e-02  &  -0.00 & 2.02e-02  &  -0.00 \\
 & 20 & 1.91e-03  &  2.00 & 3.81e-03  &  2.20 & 3.81e-03  &  2.41 \\
 & 40 & 4.77e-04  &  2.00 & 7.83e-04  &  2.28 & 7.91e-04  &  2.27 \\
 \cline{1-8}
\multirow{4}{*}{$P^2$}
 & 10 & 5.03e-04  &  -0.00 & 6.49e-04  &  -0.00 & 6.62e-04  &  -0.00 \\
 & 20 & 6.77e-05  &  2.89 & 6.67e-05  &  3.28 & 6.90e-05  &  3.26 \\
 & 40 & 8.01e-06  &  3.08 & 7.81e-06  &  3.09 & 8.11e-06  &  3.09 \\
 \cline{1-8}
\multirow{3}{*}{$P^3$}
 & 10 & 2.82e-05  &  -0.00 & 2.97e-05  &  -0.00 & 2.81e-05  &  -0.00 \\
 & 20 & 1.66e-06  &  4.09 & 1.91e-06  &  3.96 & 1.91e-06  &  3.88 \\
 & 40 & 9.78e-08  &  4.08 & 1.03e-07  &  4.21 & 1.07e-07  &  4.16 \\
\hline
\end{tabular}
\end{table}

\subsection*{Example 4.10: 2D elastodynamics with periodic boundary condition}
We consider the equations for elastodynamics \eqref{elas} with  coefficients 
$\lambda=2, \mu=1, \rho=1$.
The domain is a unit square $\Omega=[0,1]\times [0,1]$, and a periodic boundary condition is used.
The initial condition is chosen such that exact solution is the following plane wave solution:
\begin{align*}
 \sigma_{xx}(x,y,t) =&\; -\mu\sin(2\pi(x+y+\frac{\sqrt{2}}{2}c_s t)) + (\lambda+\mu)\sin(2\pi(x+y-\frac{\sqrt{2}}{2}c_p t)),\\
 \sigma_{yy}(x,y,t) =&\; \mu\sin(2\pi(x+y+\frac{\sqrt{2}}{2}c_s t)) + (\lambda+\mu)\sin(2\pi(x+y-\frac{\sqrt{2}}{2}c_p t)),\\
 \sigma_{xy}(x,y,t) =&\;  \mu\sin(2\pi(x+y-\frac{\sqrt{2}}{2}c_p t)),\\
 v(x,y,t) =&\; -\frac{\sqrt{2}}{2}c_s\sin(2\pi(x+y+\frac{\sqrt{2}}{2}c_s t)) -\frac{\sqrt{2}}{2}c_p \sin(2\pi(x+y-\frac{\sqrt{2}}{2}c_p t)),\\
 w(x,y,t) =&\; \frac{\sqrt{2}}{2}c_s\sin(2\pi(x+y+\frac{\sqrt{2}}{2}c_s t)) -\frac{\sqrt{2}}{2}c_p \sin(2\pi(x+y-\frac{\sqrt{2}}{2}c_p t)),
\end{align*}
where 
$c_p = \sqrt{\frac{\lambda+2\mu}{\rho}}$ is the P wave speed, and 
$c_s = \sqrt{\frac{\mu}{\rho}}$ is the S wave speed.

% We present numerical results for the following three DG methods:
% the DG method with the Lax-Friedrichs flux \eqref{flux:lf}, denoted as (U),
% the DG method with central flux, denoted as (C), 
% and the DG method with alternating flux discussed in section \ref{sec:es2d}, denoted as (A).
Similar to the previous example, 
we present numerical results with the upwinding DG method (U), central DG method (C), and 
the 5-components DG method (A) with an alternating flux in section \ref{sec:es2d}. 
We also present numerical results with the DG method (A-Double) 
for the augmented 10-components system \eqref{sys:2daux} on triangular meshes.
% Here we choose 
% % the easier-to-implement, 
% the
% more dissipative Lax-Friedrichs flux \eqref{flux:lf}
% over the upwinding flux \eqref{flux:upwind} 
% due to its ease of implementation, especially on
% unstructured meshes.
% since the implementation of the upwinding flux needs the eigenvalue decomposition of the $5\times 5$ coefficient matrix.

The results are similar to Example 4.9.
From the tables, while we still observe 
optimal $(k+1)$th order of accuracy for all the variable 
for method (A) on rectangular meshes, which is 
in agreement with Remark \ref{rk-error}, 
we only get suboptimal convergence order of $k$ for the stress variables $\sigma_{xx,h}, \sigma_{yy,h}$
and $\sigma_{xy,h}$
on triangular meshes.
On the other hand, the method (A-Double) is still optimally convergent.
% optimal convergence of the method (U) is observed on both meshes (we do observe 
% a bit order reduction  on the stress variables on rectangular meshes for $Q^2$ space).
% The convergence of the method (C) is worse than the method (A).
% We also observe optimal convergence for the DG method with an upwinding flux, but
Similar as Example 4.9, the method (U) is optimally convergent, but (C) is suboptimal.
% We also observe 
% suboptimal convergence for the  method (C) on both types of meshes, with a larger absolute error 
% than that for method (A) on triangular meshes.
% three types of meshes.

\begin{table}[htbp]
\caption{\label{table:es2d1} 
The $L^2$-errors
and orders for Example 4.10 for the methods (U), (C), and
(A) on a non-uniform rectangular mesh with $N\times N$ cells. 
The stress errors 
$e_s^1 = \|\sigma_{xx}-\sigma_{xx,h}\|$, 
$e_s^2 = \|\sigma_{yy}-\sigma_{yy,h}\|$, 
$e_s^3 = \|\sigma_{xy}-\sigma_{xy,h}\|$, 
and the velocity error 
$e_v = \sqrt{\|v-v_{h}\|^2+\|w-w_h\|^2}$ are recorded.
$T =0.1$. 
} \centering
\bigskip
\begin{tabular}{|c|c|cc|cc|cc|cc|}
\cline{1-10}
 & &\multicolumn{8}{|c|}{(U)}\\
 \cline{1-10}
 &  {$N$}      & $e_s^1$ & Order & $e_s^2$ & Order
 & $e_s^3$ & Order
 & $e_v$ & Order \\
 \cline{1-10}
\multirow{4}{*}{$P^1$}
 & 10 & 7.66e-02  &  -0.00 & 6.70e-02  &  -0.00 & 3.80e-02  &  -0.00 & 4.66e-02  &  -0.00 \\ 
 & 20 & 2.04e-02  &  1.91 & 1.82e-02  &  1.88 & 7.07e-03  &  2.43 & 1.18e-02  &  1.98 \\ 
 & 40 & 5.50e-03  &  1.89 & 4.62e-03  &  1.98 & 1.50e-03  &  2.23 & 2.96e-03  &  2.00 \\ 
 \cline{1-10}
\multirow{4}{*}{$P^2$}
 & 10 & 1.16e-02  &  -0.00 & 9.58e-03  &  -0.00 & 3.45e-03  &  -0.00 & 3.03e-03  &  -0.00 \\ 
 & 20 & 1.81e-03  &  2.67 & 1.66e-03  &  2.53 & 4.11e-04  &  3.07 & 3.95e-04  &  2.94 \\ 
 & 40 & 2.87e-04  &  2.66 & 2.75e-04  &  2.59 & 5.01e-05  &  3.04 & 4.96e-05  &  2.99 \\ 
 \cline{1-10}
\multirow{3}{*}{$P^3$}
 & 10 & 2.73e-04  &  -0.00 & 2.22e-04  &  -0.00 & 5.84e-05  &  -0.00 & 1.02e-04  &  -0.00 \\ 
 & 20 & 1.93e-05  &  3.83 & 1.59e-05  &  3.80 & 3.50e-06  &  4.06 & 6.42e-06  &  3.98 \\ 
 & 40 & 1.31e-06  &  3.87 & 9.93e-07  &  4.00 & 2.03e-07  &  4.11 & 4.01e-07  &  4.00 \\ 
 \cline{1-10}
 & &\multicolumn{8}{|c|}{(C)}\\
 \cline{1-10}
 &  {$N$}      & $e_s^1$ & Order & $e_s^2$ & Order
 & $e_s^3$ & Order
 & $e_v$ & Order \\
 \cline{1-10}
\multirow{4}{*}{$P^1$}
 & 10 & 1.66e+00  &  -0.00 & 1.06e+00  &  -0.00 & 6.35e-01  &  -0.00 & 8.58e-02  &  -0.00 \\ 
 & 20 & 8.48e-01  &  0.97 & 5.32e-01  &  0.99 & 3.12e-01  &  1.02 & 4.08e-02  &  1.07 \\ 
 & 40 & 4.26e-01  &  0.99 & 2.67e-01  &  1.00 & 1.55e-01  &  1.01 & 2.02e-02  &  1.02 \\ 
 \cline{1-10}
\multirow{4}{*}{$P^2$}
 & 10 & 5.37e-02  &  -0.00 & 6.01e-03  &  -0.00 & 4.75e-03  &  -0.00 & 3.94e-03  &  -0.00 \\ 
 & 20 & 9.82e-03  &  2.45 & 3.11e-03  &  0.95 & 1.23e-03  &  1.95 & 9.72e-04  &  2.02 \\ 
 & 40 & 7.38e-04  &  3.73 & 7.54e-04  &  2.05 & 2.63e-04  &  2.23 & 1.90e-04  &  2.35 \\ 
\cline{1-10}
\multirow{3}{*}{$P^3$}
 & 10 & 3.68e-03  &  -0.00 & 1.62e-03  &  -0.00 & 6.47e-04  &  -0.00 & 2.76e-04  &  -0.00 \\ 
 & 20 & 4.55e-04  &  3.01 & 2.05e-04  &  2.98 & 8.14e-05  &  2.99 & 3.46e-05  &  3.00 \\ 
 & 40 & 5.56e-05  &  3.03 & 2.59e-05  &  2.99 & 1.01e-05  &  3.01 & 4.27e-06  &  3.01 \\ 
\cline{1-10}
 & &\multicolumn{8}{|c|}{(A)}\\
 \cline{1-10}
  &  {$N$}      & $e_s^1$ & Order & $e_s^2$ & Order
 & $e_s^3$ & Order
 & $e_v$ & Order \\
 \cline{1-10}
\multirow{4}{*}{$P^1$}
 & 10 & 5.57e-01  &  -0.00 & 1.84e-01  &  -0.00 & 5.32e-02  &  -0.00 & 6.62e-02  &  -0.00 \\ 
 & 20 & 1.44e-01  &  1.95 & 4.30e-02  &  2.10 & 8.53e-03  &  2.64 & 1.49e-02  &  2.15 \\ 
 & 40 & 3.63e-02  &  1.99 & 1.06e-02  &  2.02 & 1.68e-03  &  2.34 & 3.54e-03  &  2.08 \\ 
 \cline{1-10}
\multirow{4}{*}{$P^2$}
 & 10 & 2.82e-02  &  -0.00 & 8.66e-03  &  -0.00 & 2.69e-03  &  -0.00 & 3.33e-03  &  -0.00 \\ 
 & 20 & 3.65e-03  &  2.95 & 1.07e-03  &  3.01 & 2.38e-04  &  3.50 & 3.68e-04  &  3.18 \\ 
 & 40 & 4.57e-04  &  3.00 & 1.34e-04  &  3.00 & 2.26e-05  &  3.39 & 4.50e-05  &  3.03 \\ 
 \cline{1-10}
\multirow{3}{*}{$P^3$}
 & 10 & 1.12e-03  &  -0.00 & 3.37e-04  &  -0.00 & 8.05e-05  &  -0.00 & 1.19e-04  &  -0.00 \\ 
 & 20 & 7.27e-05  &  3.95 & 2.11e-05  &  4.00 & 3.85e-06  &  4.39 & 7.42e-06  &  4.00 \\ 
 & 40 & 4.53e-06  &  4.00 & 1.33e-06  &  3.99 & 2.13e-07  &  4.18 & 4.57e-07  &  4.02 \\ 
\hline
\end{tabular}
\end{table}

\begin{table}[htbp]
\caption{\label{table:es2d3} 
The $L^2$-errors
and orders for Example 4.10 for the methods (U), (C), (A), and (A-Double)
% the central DG method (C), and the 
% new DG method (A) 
on a triangular mesh with mesh size $h = 1/N$. 
The stress errors 
$e_s^1 = \|\sigma_{xx}-\sigma_{xx,h}\|$, 
$e_s^2 = \|\sigma_{yy}-\sigma_{yy,h}\|$, 
$e_s^3 = \|\sigma_{xy}-\sigma_{xy,h}\|$, 
and the velocity error 
$e_v = \sqrt{\|v-v_{h}\|^2+\|w-w_h\|^2}$ are recorded.
$T =0.1$. 
} \centering
\bigskip
\begin{tabular}{|c|c|cc|cc|cc|cc|}
\cline{1-10}
 & &\multicolumn{8}{|c|}{(U)}\\
 \cline{1-10}
 &  {$N$}      & $e_s^1$ & Order & $e_s^2$ & Order
 & $e_s^3$ & Order
 & $e_v$ & Order \\
 \cline{1-10}
\multirow{4}{*}{$P^1$}
 & 10 & 6.12e-02  &  -0.00 & 5.71e-02  &  -0.00 & 2.66e-02  &  -0.00 & 3.79e-02  &  -0.00 \\ 
 & 20 & 1.57e-02  &  1.97 & 1.42e-02  &  2.00 & 6.15e-03  &  2.11 & 9.85e-03  &  1.94 \\ 
 & 40 & 3.84e-03  &  2.03 & 3.52e-03  &  2.02 & 1.43e-03  &  2.10 & 2.42e-03  &  2.02 \\ 
 \cline{1-10}
\multirow{4}{*}{$P^2$}
 & 10 & 5.80e-03  &  -0.00 & 6.13e-03  &  -0.00 & 2.86e-03  &  -0.00 & 2.81e-03  &  -0.00 \\ 
 & 20 & 9.28e-04  &  2.64 & 1.05e-03  &  2.55 & 5.93e-04  &  2.27 & 3.46e-04  &  3.02 \\ 
 & 40 & 1.22e-04  &  2.93 & 1.46e-04  &  2.84 & 8.09e-05  &  2.87 & 4.23e-05  &  3.03 \\ 
 \cline{1-10}
\multirow{3}{*}{$P^3$}
 & 10 & 2.98e-04  &  -0.00 & 3.39e-04  &  -0.00 & 1.64e-04  &  -0.00 & 1.33e-04  &  -0.00 \\ 
 & 20 & 1.80e-05  &  4.05 & 2.32e-05  &  3.87 & 9.79e-06  &  4.06 & 8.91e-06  &  3.90 \\ 
 & 40 & 1.11e-06  &  4.01 & 1.53e-06  &  3.92 & 6.18e-07  &  3.99 & 4.97e-07  &  4.16 \\ 
 \cline{1-10}
 & &\multicolumn{8}{|c|}{(C)}\\
 \cline{1-10}
 &  {$N$}      & $e_s^1$ & Order & $e_s^2$ & Order
 & $e_s^3$ & Order
 & $e_v$ & Order \\
 \cline{1-10}
\multirow{4}{*}{$P^1$}
 & 10 & 3.10e+00  &  -0.00 & 1.59e+00  &  -0.00 & 1.63e+00  &  -0.00 & 5.32e-02  &  -0.00 \\ 
 & 20 & 1.50e+00  &  1.05 & 6.65e-01  &  1.26 & 8.35e-01  &  0.97 & 1.15e-02  &  2.21 \\ 
 & 40 & 7.37e-01  &  1.02 & 3.11e-01  &  1.09 & 4.07e-01  &  1.04 & 2.81e-03  &  2.04 \\ 
 \cline{1-10}
\multirow{4}{*}{$P^2$}
 & 10 & 1.22e-01  &  -0.00 & 9.32e-02  &  -0.00 & 9.13e-02  &  -0.00 & 3.34e-03  &  -0.00 \\ 
 & 20 & 2.69e-02  &  2.19 & 2.30e-02  &  2.02 & 2.28e-02  &  2.00 & 3.41e-04  &  3.29 \\ 
 & 40 & 5.55e-03  &  2.28 & 5.32e-03  &  2.11 & 5.04e-03  &  2.18 & 4.00e-05  &  3.09 \\ 
\cline{1-10}
\multirow{3}{*}{$P^3$}
 & 10 & 1.94e-02  &  -0.00 & 1.09e-02  &  -0.00 & 1.01e-02  &  -0.00 & 2.09e-04  &  -0.00 \\ 
 & 20 & 2.29e-03  &  3.08 & 1.09e-03  &  3.33 & 1.26e-03  &  3.01 & 1.73e-05  &  3.60 \\ 
 & 40 & 2.76e-04  &  3.05 & 1.30e-04  &  3.07 & 1.50e-04  &  3.07 & 1.55e-06  &  3.48 \\ 
\cline{1-10}
 & &\multicolumn{8}{|c|}{(A)}\\
 \cline{1-10}
  &  {$N$}      & $e_s^1$ & Order & $e_s^2$ & Order
 & $e_s^3$ & Order
 & $e_v$ & Order \\
 \cline{1-10}
\multirow{4}{*}{$P^1$}
 & 10 & 1.14e+00  &  -0.00 & 3.62e-01  &  -0.00 & 4.08e-01  &  -0.00 & 4.22e-02  &  -0.00 \\ 
 & 20 & 6.69e-01  &  0.77 & 2.27e-01  &  0.67 & 2.31e-01  &  0.82 & 1.08e-02  &  1.97 \\ 
 & 40 & 2.21e-01  &  1.60 & 1.29e-01  &  0.81 & 1.06e-01  &  1.12 & 2.63e-03  &  2.04 \\ 
 \cline{1-10}
\multirow{4}{*}{$P^2$}
 & 10 & 9.88e-02  &  -0.00 & 3.72e-02  &  -0.00 & 3.99e-02  &  -0.00 & 2.88e-03  &  -0.00 \\ 
 & 20 & 3.06e-02  &  1.69 & 1.21e-02  &  1.62 & 1.13e-02  &  1.82 & 3.47e-04  &  3.05 \\ 
 & 40 & 6.16e-03  &  2.31 & 3.32e-03  &  1.87 & 2.80e-03  &  2.01 & 4.11e-05  &  3.08 \\ 
 \cline{1-10}
\multirow{3}{*}{$P^3$}
 & 10 & 1.02e-02  &  -0.00 & 2.83e-03  &  -0.00 & 2.96e-03  &  -0.00 & 1.39e-04  &  -0.00 \\ 
 & 20 & 1.27e-03  &  3.00 & 4.71e-04  &  2.59 & 4.18e-04  &  2.82 & 9.35e-06  &  3.90 \\ 
 & 40 & 1.03e-04  &  3.63 & 6.43e-05  &  2.88 & 4.98e-05  &  3.07 & 5.22e-07  &  4.16 \\ 
 \cline{1-10}
 & &\multicolumn{8}{|c|}{(A-Double)}\\
 \cline{1-10}
  &  {$N$}      & $e_s^1$ & Order & $e_s^2$ & Order
 & $e_s^3$ & Order
 & $e_v$ & Order \\
 \cline{1-10}
\multirow{4}{*}{$P^1$}
 & 10 & 7.65e-02  &  -0.00 & 5.76e-02  &  -0.00 & 3.86e-02  &  -0.00 & 2.73e-02  &  -0.00 \\ 
 & 20 & 1.98e-02  &  1.95 & 1.62e-02  &  1.83 & 9.72e-03  &  1.99 & 6.99e-03  &  1.96 \\ 
 & 40 & 4.79e-03  &  2.05 & 4.01e-03  &  2.01 & 2.31e-03  &  2.07 & 1.71e-03  &  2.03 \\ 
 \cline{1-10}
\multirow{4}{*}{$P^2$}
 & 10 & 9.88e-02  &  -0.00 & 3.72e-02  &  -0.00 & 3.99e-02  &  -0.00 & 2.88e-03  &  -0.00 \\ 
 & 20 & 3.06e-02  &  1.69 & 1.21e-02  &  1.62 & 1.13e-02  &  1.82 & 3.47e-04  &  3.05 \\ 
 & 40 & 6.16e-03  &  2.31 & 3.32e-03  &  1.87 & 2.80e-03  &  2.01 & 4.11e-05  &  3.08 \\ 
 \cline{1-10}
\multirow{3}{*}{$P^3$}
 & 10 & 2.03e-04  &  -0.00 & 1.26e-04  &  -0.00 & 7.98e-05  &  -0.00 & 9.67e-05  &  -0.00 \\ 
 & 20 & 1.35e-05  &  3.91 & 8.59e-06  &  3.88 & 5.43e-06  &  3.88 & 6.46e-06  &  3.90 \\ 
  & 40 & 7.43e-07  &  4.19 & 4.95e-07  &  4.12 & 3.15e-07  &  4.11 & 3.60e-07  &  4.17 \\ 
\hline
\end{tabular}
\end{table}

\subsection*{Example 4.11: long time simulation: advection of a plane wave}
We consider the advection equation \eqref{eq-ad2} on the unit square with 
periodic boundary condition and 
initial condition 
$u(x,y,0) = \sin(4\pi (x+y))$. This is a two dimensional extension of the test Example 4.4.
We present numerical results for the three $P^2$-DG methods, (U) for upwinding flux,  (C) for central flux, and 
(A) for the new method. Both uniform rectangular mesh and unstructured triangular mesh are considered. 
We use the RK3 time stepping. The CFL number is taken to be $0.05$ on the rectangular mesh, and $0.02$ on the triangular mesh.

Numerical results on the cut line $y=.5$ at time $T = 40$ (wave propagates 80 cycles) are shown in Figure \ref{fig:pw2d}.
Figure \ref{fig:pw2d} indicates the superior performance of the new method on both rectangular and triangular meshes over the 
dissipative upwinding DG method. 
The results on the rectangular mesh for the central DG method and the new methods 
are comparable, both have small dissipation error with a slight phase shift.
However, the result on triangular mesh for the new method is clearly better than the 
central DG method, which is very oscillatory.
% Finally, we remark that changing the RK3 time stepping to the LF4 time stepping 

\begin{figure}[ht!]
 \caption{Numerical solution at $T=40$ on the cut line $y = 0.5$ 
 for Example 4.11. RK3 time stepping. 
Top row: method (U). 
Middle row: method (C).
Bottom row: method (A).
 Left: square mesh with $10\times 10$ cells, $Q^2$ space. Right: triangular mesh with meshsize $h = 0.1$, 
 $P^2$ space.
 (Roughly 15 dofs/wavelength in each direction)
 }
 \label{fig:pw2d}
 \includegraphics[width=.45\textwidth]{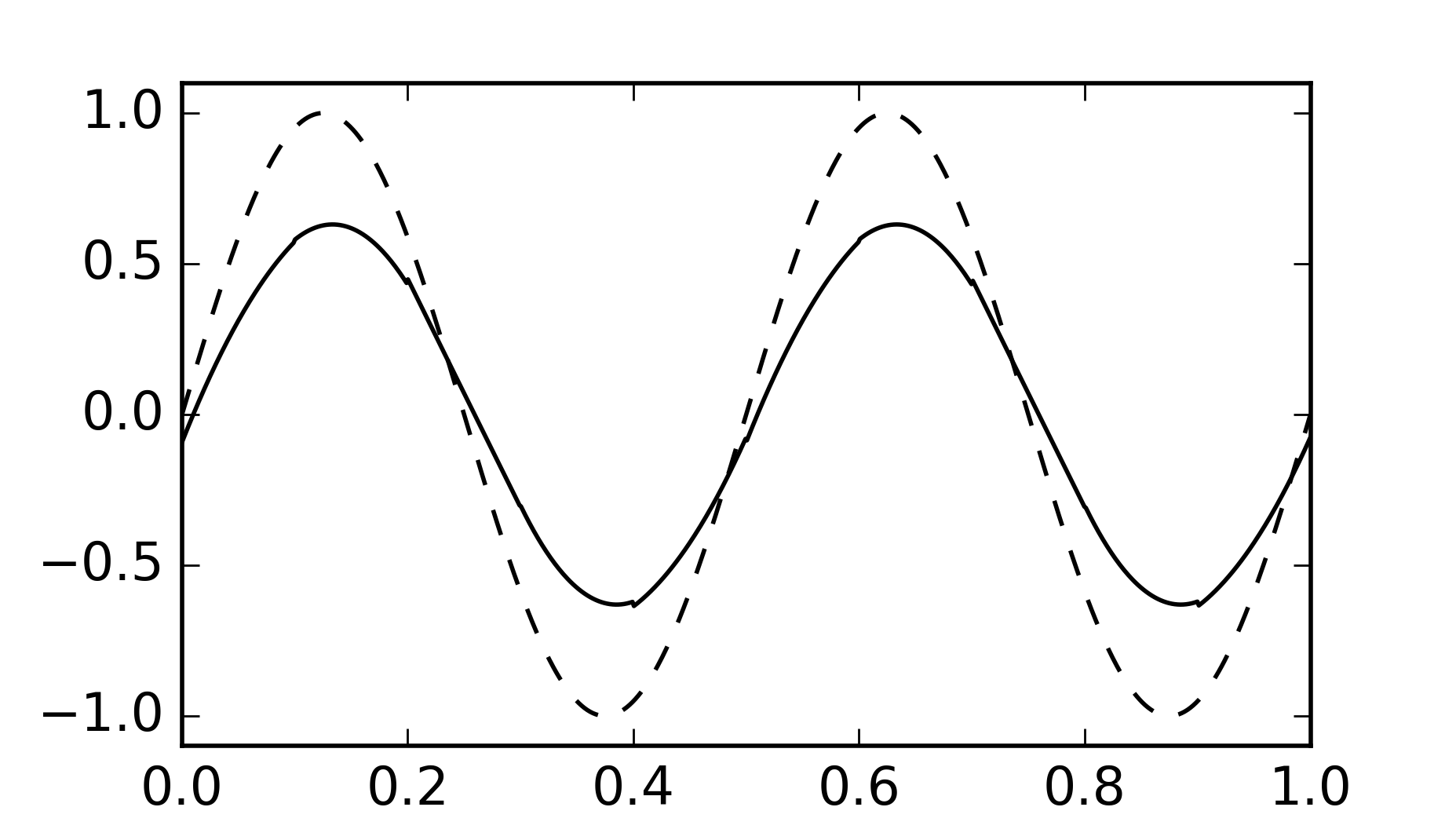}
 \includegraphics[width=.45\textwidth]{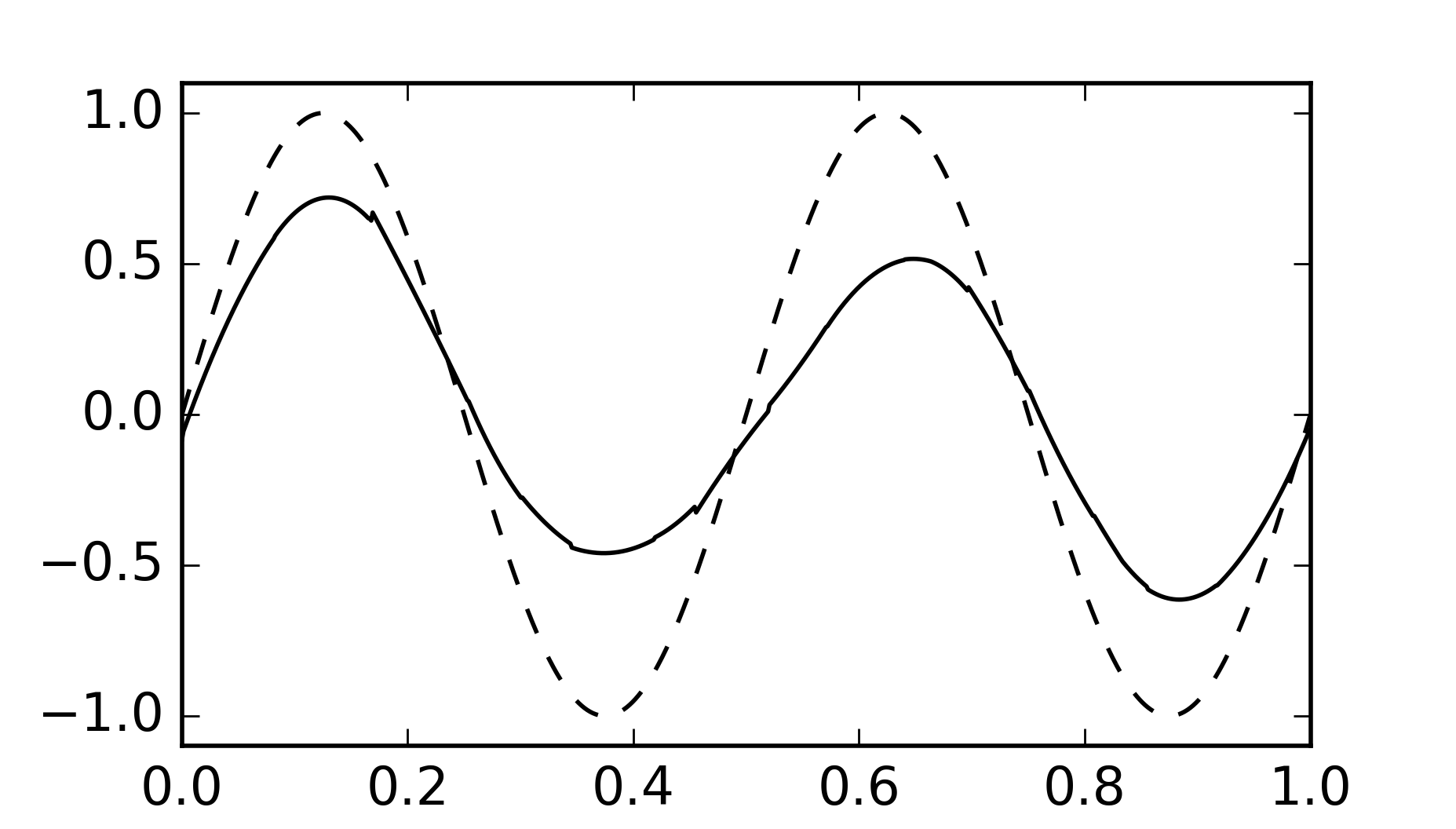}
 \includegraphics[width=.45\textwidth]{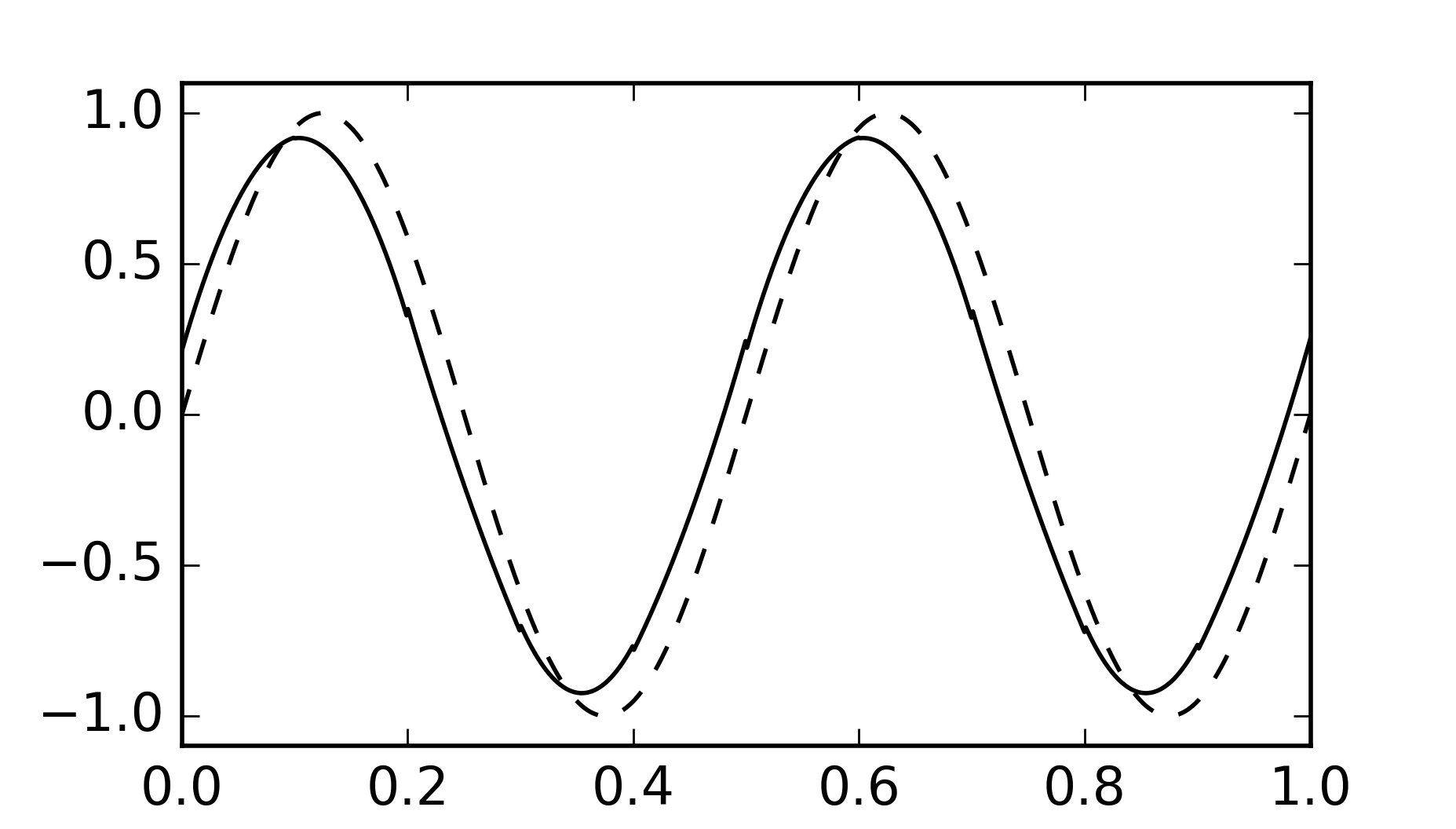}
 \includegraphics[width=.45\textwidth]{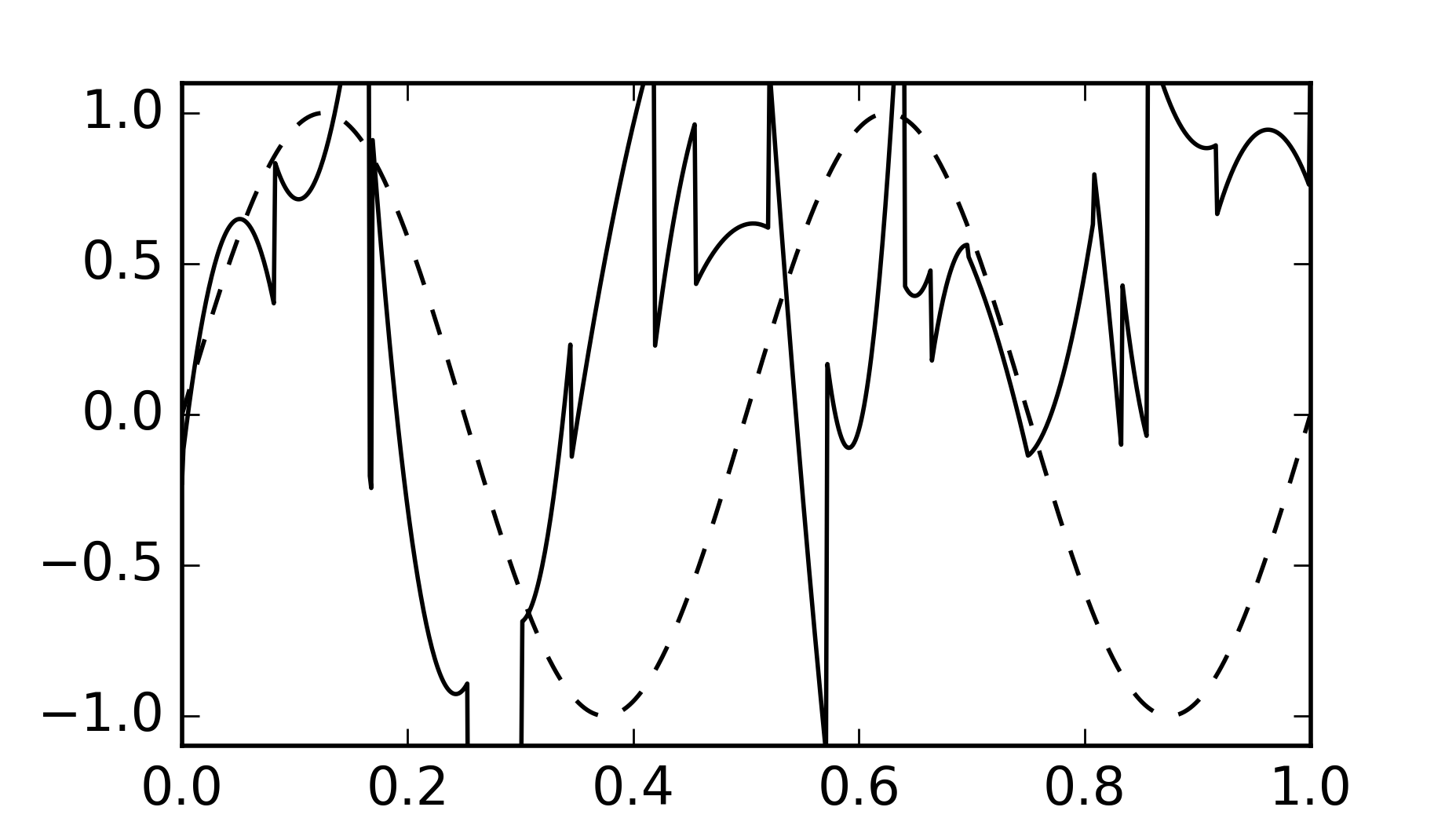}
 \includegraphics[width=.45\textwidth]{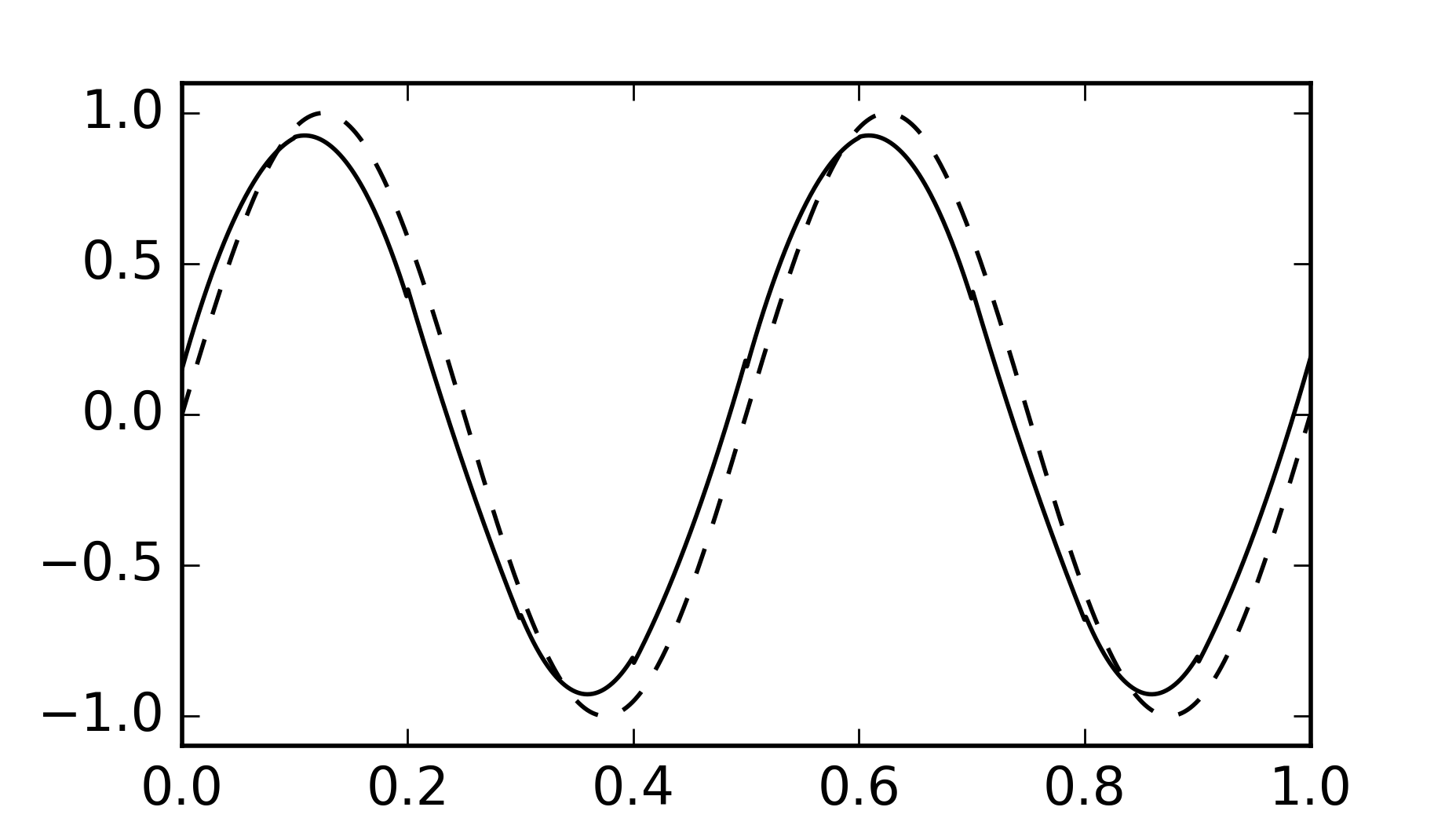}
 \includegraphics[width=.45\textwidth]{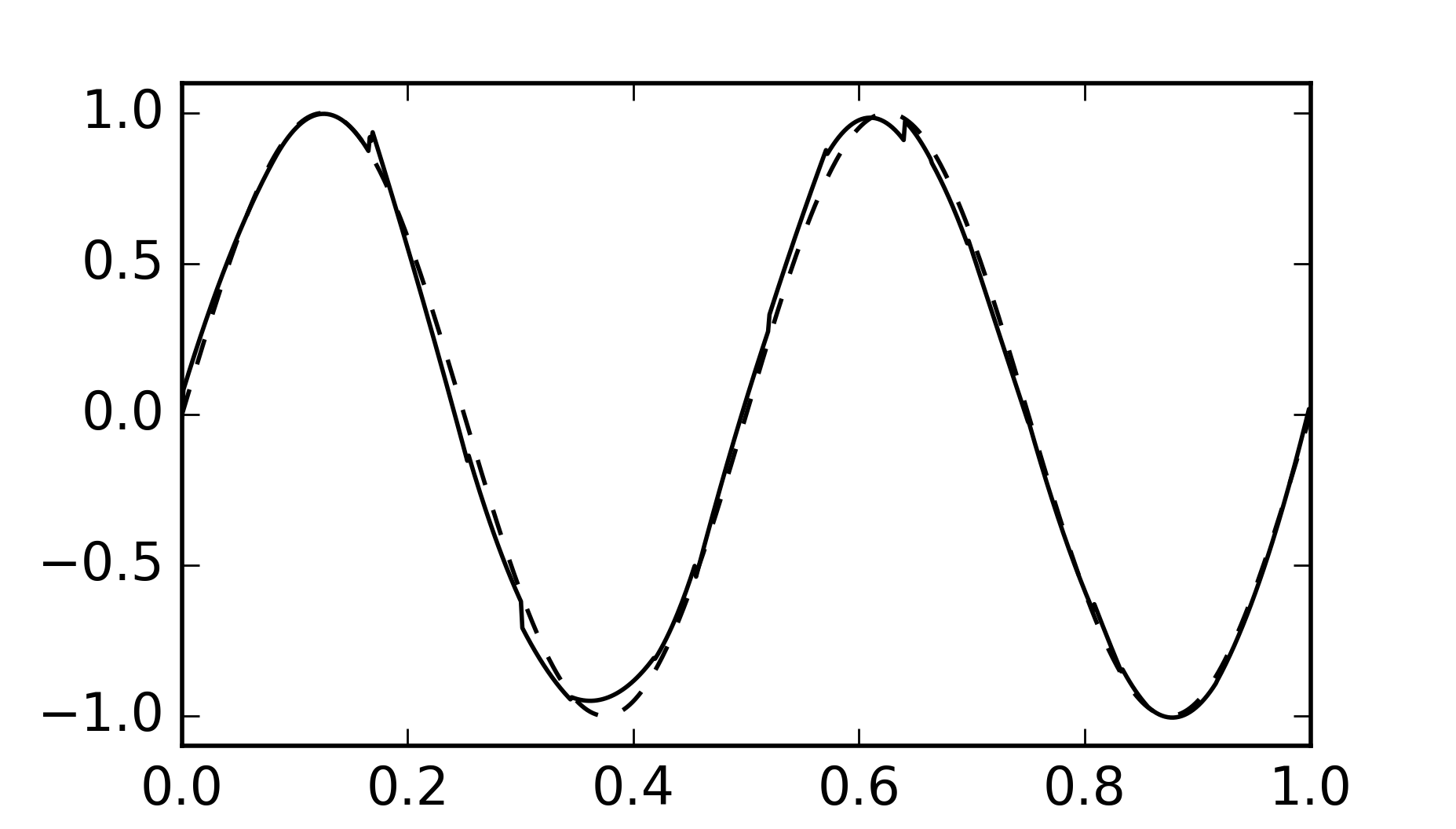}
\end{figure}

\subsection*{Example 4.12: long time simulation: advection of a Gaussian pulse}
We consider the advection equation \eqref{eq-ad2} on the unit square with 
periodic boundary condition and 
initial condition 
$u(x,y,0) = \exp(-200((x-0.5)^2+(y-.5)^2))$. This is a two dimensional analog of the test Example 4.5.
We present numerical results for the three DG methods using $Q^2/P^2$ space. 
Both uniform rectangular mesh and unstructured triangular mesh are considered. 
We use the RK3 time stepping. The CFL number is taken to be $0.05$ on the rectangular mesh, and $0.02$ on the triangular mesh.

Numerical results on the cut line $y=.5$ at time $T = 10$ 
(wave propagates 10 cycles) are shown in Figure \ref{fig:gs2d}.
Figure \ref{fig:gs2d} indicates the superior performance of the new method on both rectangular 
and triangular meshes over the 
dissipative upwinding DG method in terms of dissipation error. It is also superior over 
the central DG method in terms of both dissipation and dispersion error. 

\begin{figure}[ht!]
 \caption{Numerical solution at $T=10$ on the cut line $y = 0.5$ 
 for Example 4.12. RK3 time stepping. 
Top row: method (U). 
Middle row: method (C).
Bottom row: method (A).
Left: square mesh with $20\times 20$ cells, $Q^2$ space. 
Right: triangular mesh with meshsize $h = 0.05$, $P^2$ space.
 (Roughly 60 dofs in each direction)
 }
 \label{fig:gs2d}
 \includegraphics[width=.45\textwidth]{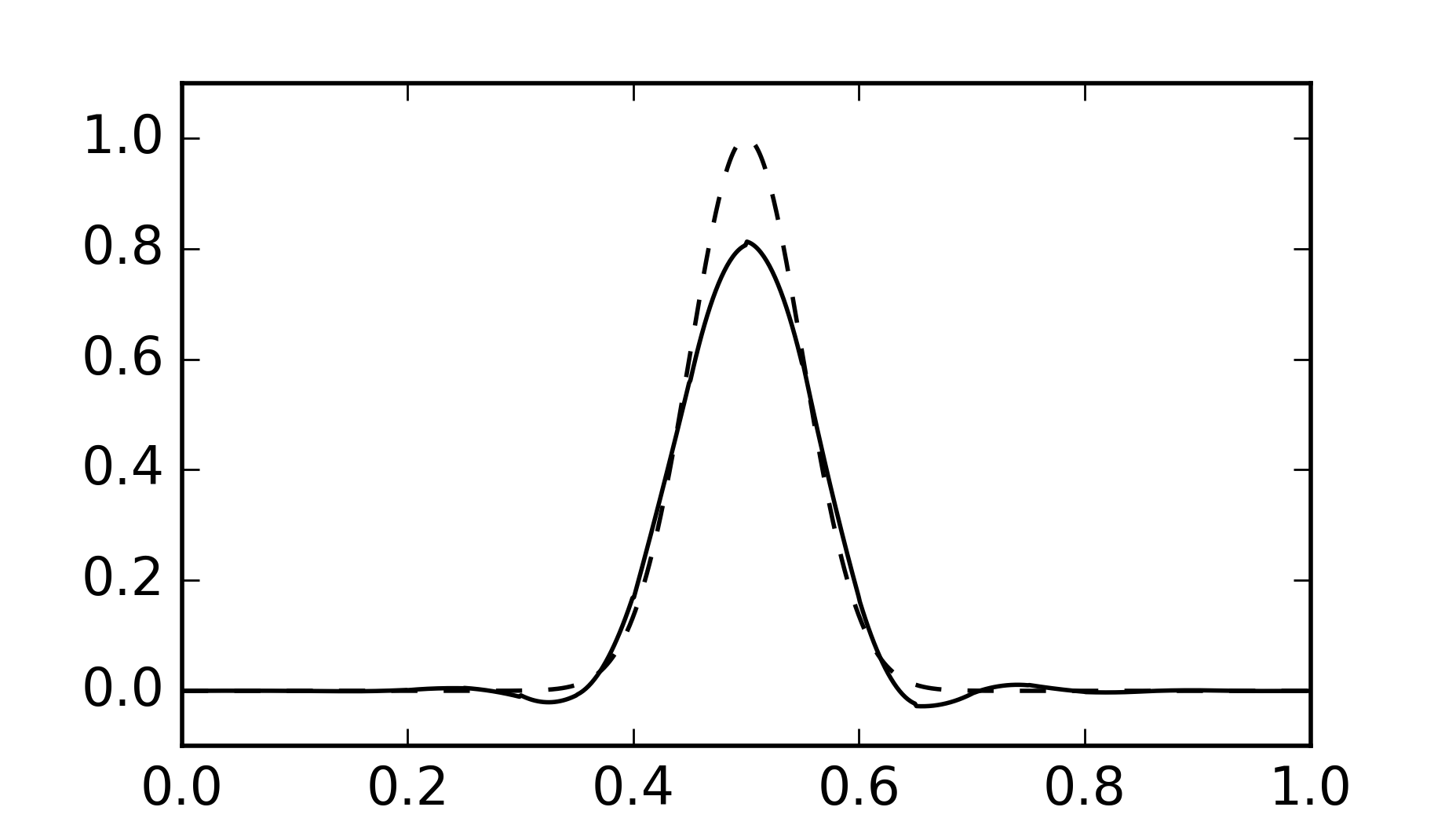}
 \includegraphics[width=.45\textwidth]{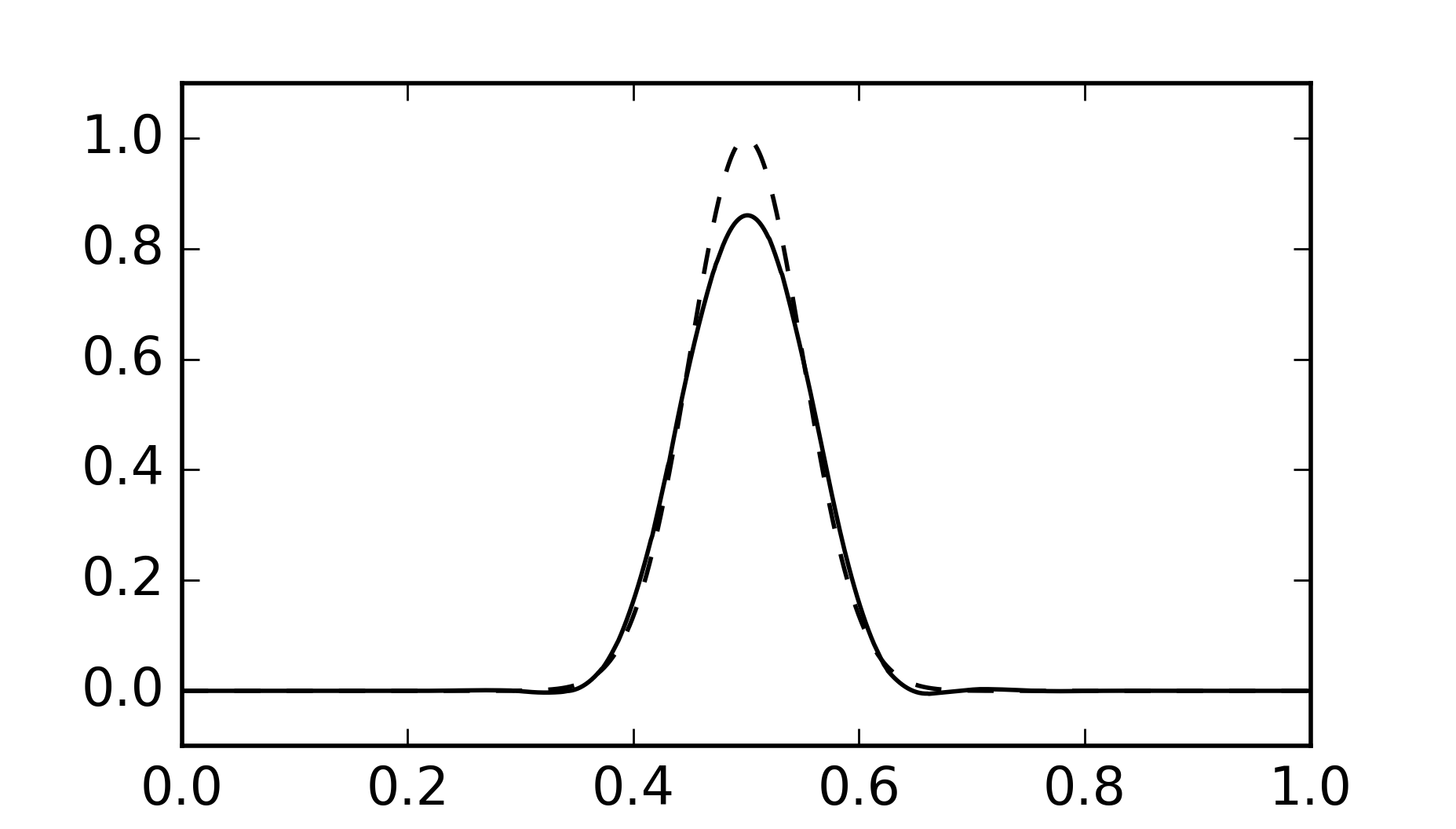}
 \includegraphics[width=.45\textwidth]{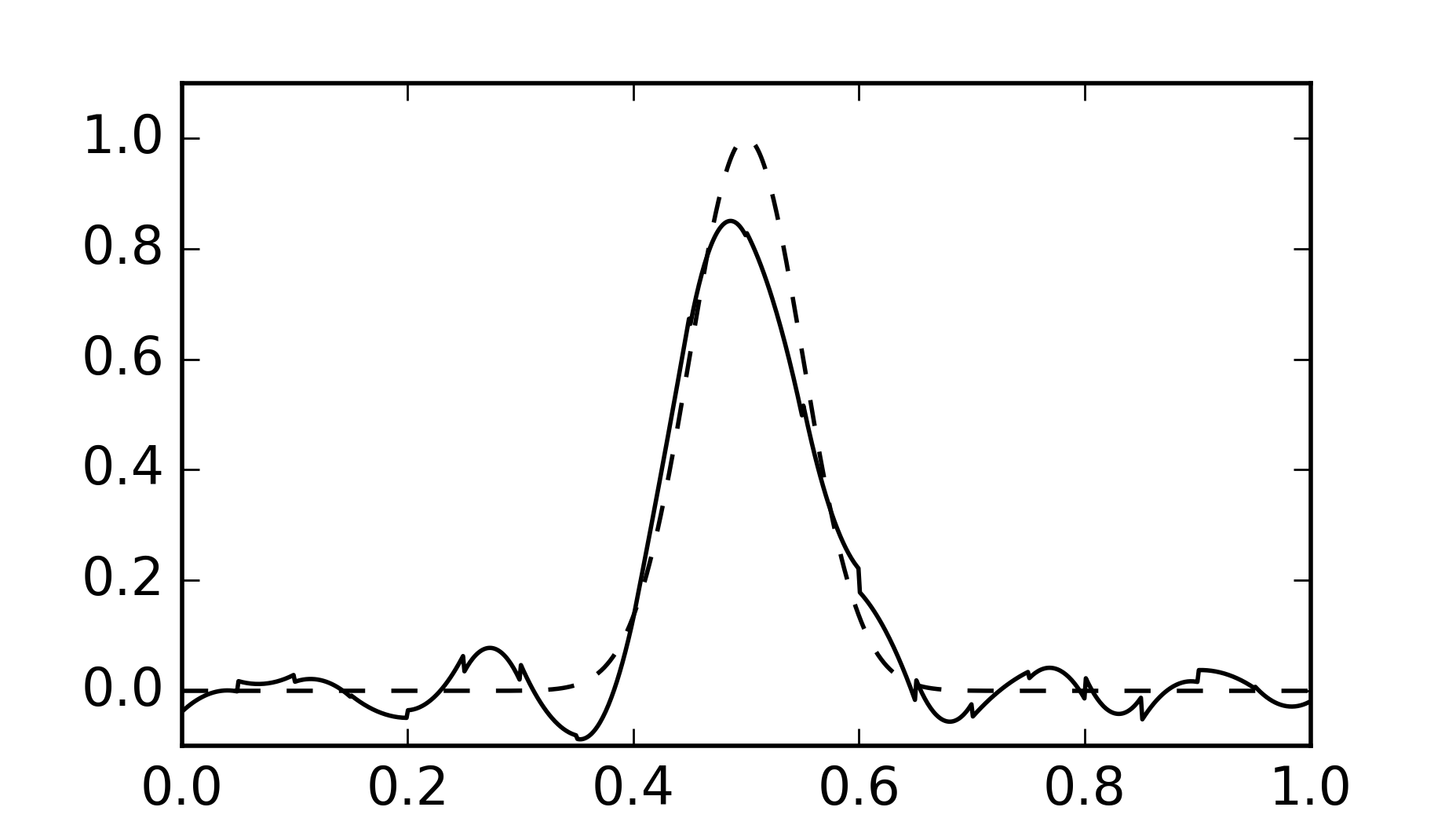}
 \includegraphics[width=.45\textwidth]{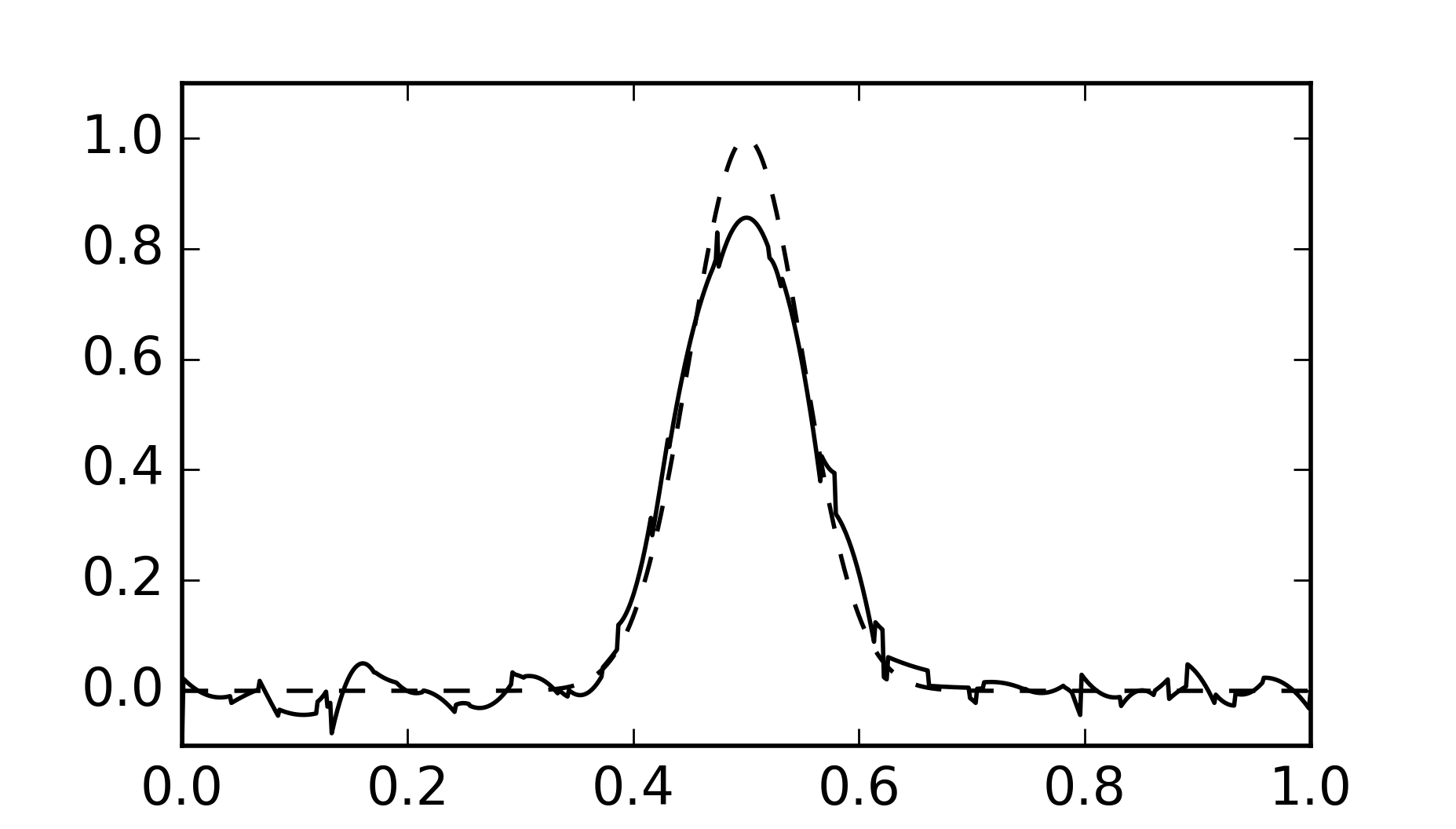}
 \includegraphics[width=.45\textwidth]{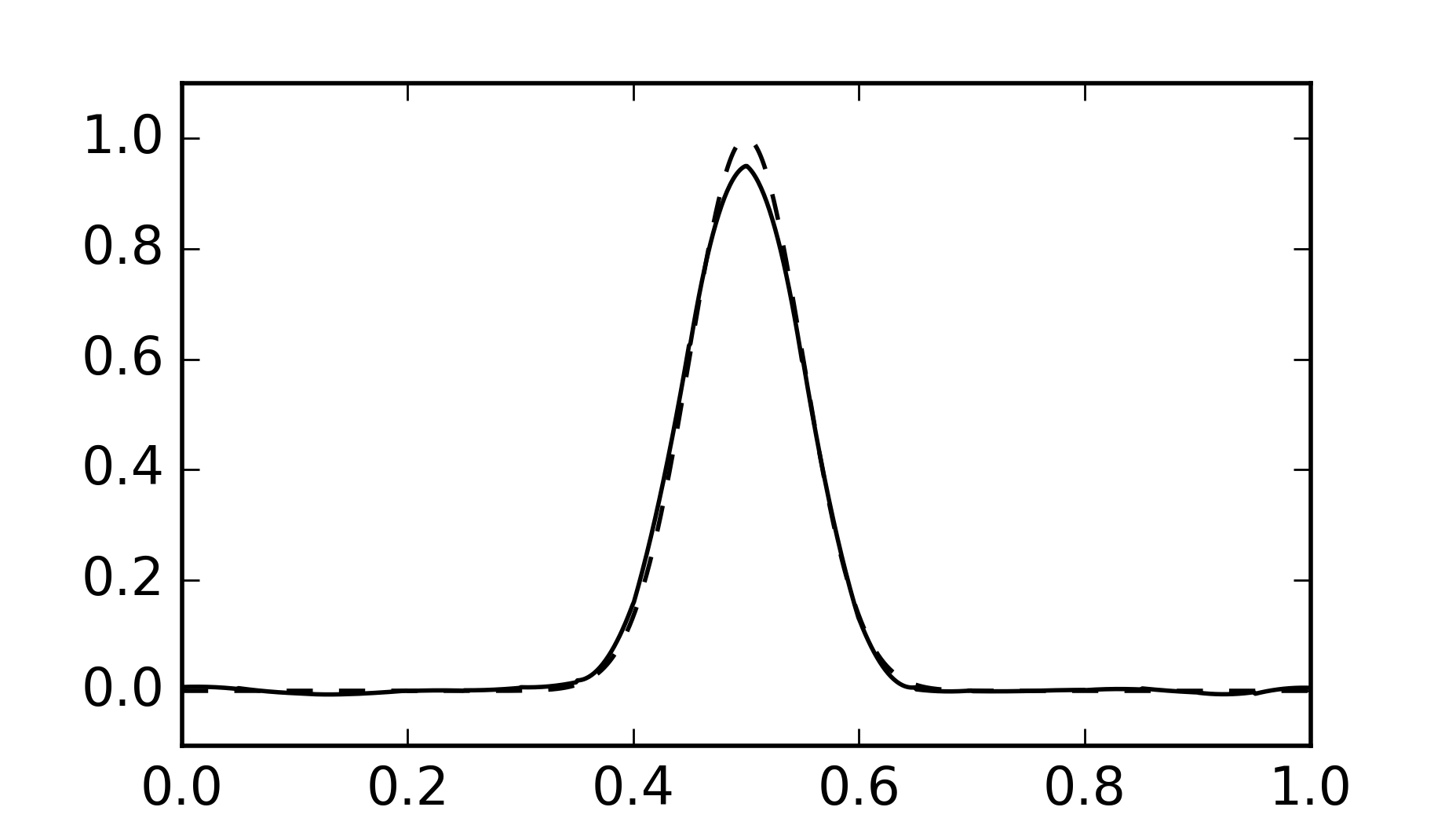}
 \includegraphics[width=.45\textwidth]{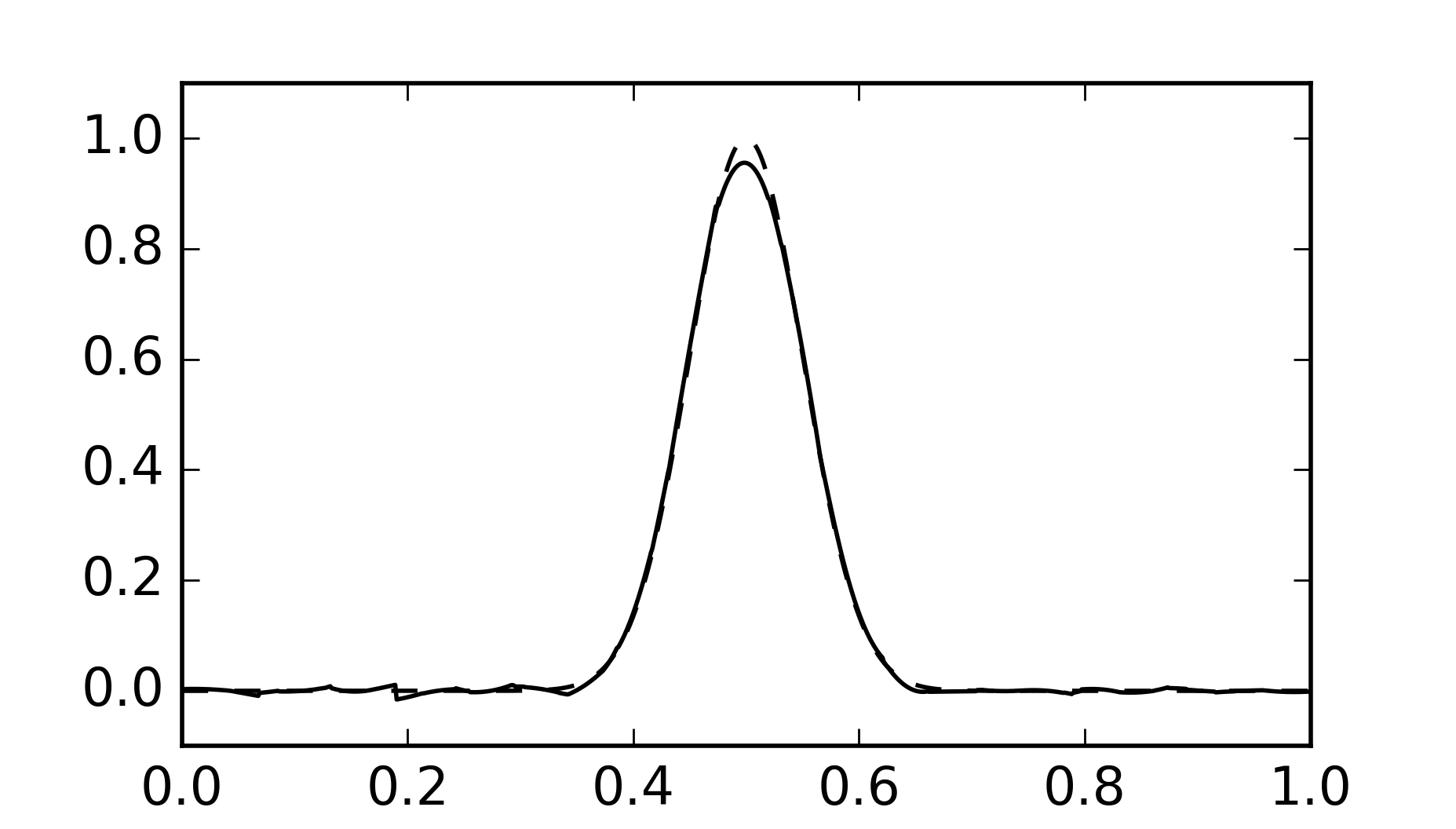}
\end{figure}

\subsection*{Example 4.13: long time simulation: 2D acoustics with time periodic source}
We consider the following acoustic equations on the whole space $\mathrm{R}^2$ with
time periodic source:
\begin{align*}
% \label{ac2d-source}
p_t + u_x + v_y = &\; S,\\
u_t + p_x = &\; 0,\\
v_t + p_y = &\; 0,
\end{align*}
where the source term
\[
 S=\exp\left[-{\mathrm{ln}(2)\left(\frac{x^2+y^2}{(0.2)^2}\right)}\right]\sin(\omega t),
 \quad\quad {\text{ with }} \omega = 4\pi.
\]
Zero initial condition is considered.
The exact solution to the above equations can be found in \cite{CAA2}.
We specifically mention that, the exact solution $p(x,y,t)$ is purely radial, with its spatial 
dependence only through the radius $r = \sqrt{x^2+y^2}$, and 
at any physical location $(x,y)$, it is $0$ (at rest) for $t<r$, and is
time-periodic with frequency $\omega=4\pi$ for $t > r$.

We shall consider the numerical solution on a stretched rectangular domain 
\[
 \Omega = [0, 12]\times [0, 1].
\]
The final time of the simulation is $T=10$.

The boundary treatment is given as follows.
By symmetry of the problem, the symmetry (wall) boundary condition is used 
along the left ($x=0$) and bottom ($y=0$) boundaries: 
\begin{align*}
\widehat{\Bn \bld u_h} =& [0, p, 0]'\quad\quad \text{ on left boundary $x=0$},\\ 
\widehat{\Bn \bld u_h} =& [0, 0, p]'\quad\quad \text{ on bottom boundary $y=0$}. 
\end{align*}
At time $T=10$, the solution is still at rest on the right boundary ($x=12$), and a simple 
outflow boundary condition is imposed there.
To treat the top boundary ($y=1$), 
we impose a perfectly matched layer (PML) \cite{pmlDG} with thickness $0.5$,
\[
 \Omega_{pml} = [0, 12]\times [1, 1.5].
\] 
We solve the following PML-ODE system from \cite{pmlDG} on the PML domain $\Omega_{pml}$:
\begin{align*}
 p_t + u_x + v_y = &\; -\sigma p,\\
u_t + p_x = &\; \sigma(u+\tilde u),\\
v_t + p_y = &\; -\sigma v,\\
\tilde u_t =&\;- \sigma(u+\tilde u),
\end{align*}
with the absorption constant $\sigma$ taken to be $\sigma =10$.

We present numerical results for the three DG methods with $P^2$ space on, 
(U) for upwinding flux, 
(C) for central flux, and 
(A) for the alternating flux.
The RK3 time stepping is used, and CFL number is taken to be $0.05$.
We use a triangular mesh with mesh size $h=0.2$, see Figure \ref{fig:pmlX}.
\begin{figure}[ht!]
\caption{Computational mesh for Example 4.13.The PML region $\Omega_{pml}$ is colored in red, and 
the domain $\Omega$ is colored in green.}
\label{fig:pmlX}
\includegraphics[width=.8\textwidth]{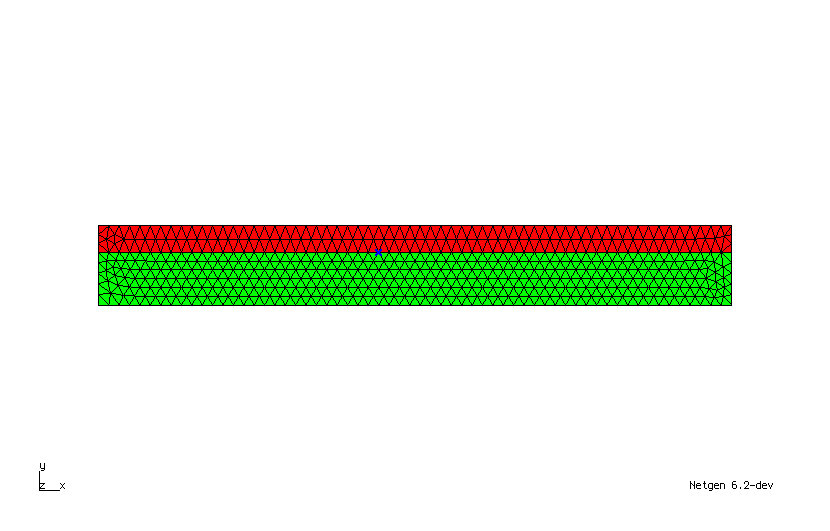}
\end{figure}

Numerical results
for the pressure field 
$p_h$ on the segment $5\le x\le 9$ along the x-axis are shown in Figure \ref{fig:pml}.
The method (U) produce visible dissipation error, while the method (C) produce slight phase shift.
The method (A) is better than (U) in terms of dissipation error, and better than (C) in terms of phase shift.

\begin{figure}[ht!]
 \caption{Pressure field at $T=10$ on the segment  $\{(x,0): 5\le x\le 9\}$ 
 for Example 4.14. RK3 time stepping. 
Top: method (U). 
Middle: method (C).
Bottom: method (A).
Solid line: numerical solution. Dashed line: exact solution.
 }
 \label{fig:pml}
 \includegraphics[width=.8\textwidth]{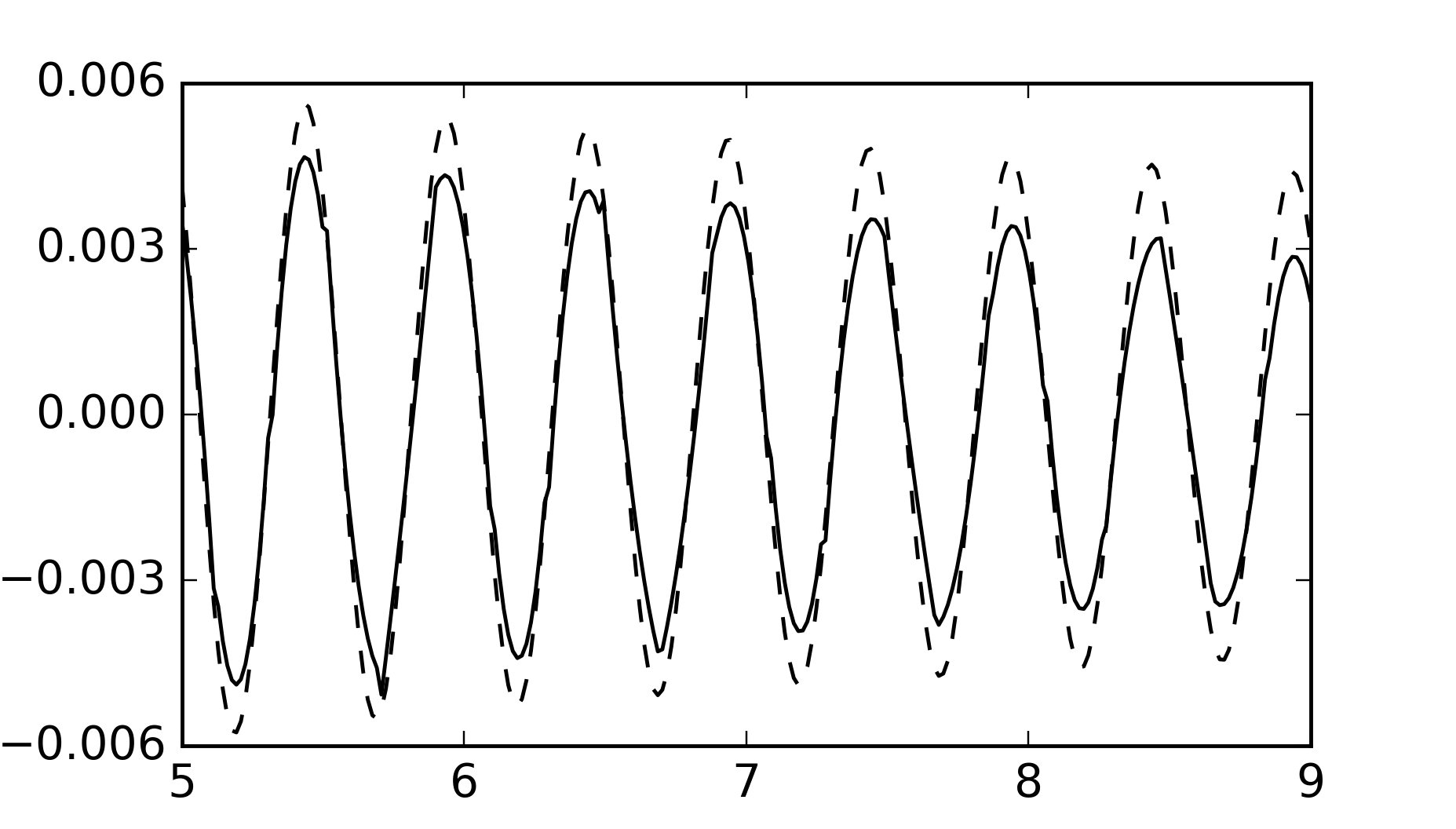}
 \includegraphics[width=.8\textwidth]{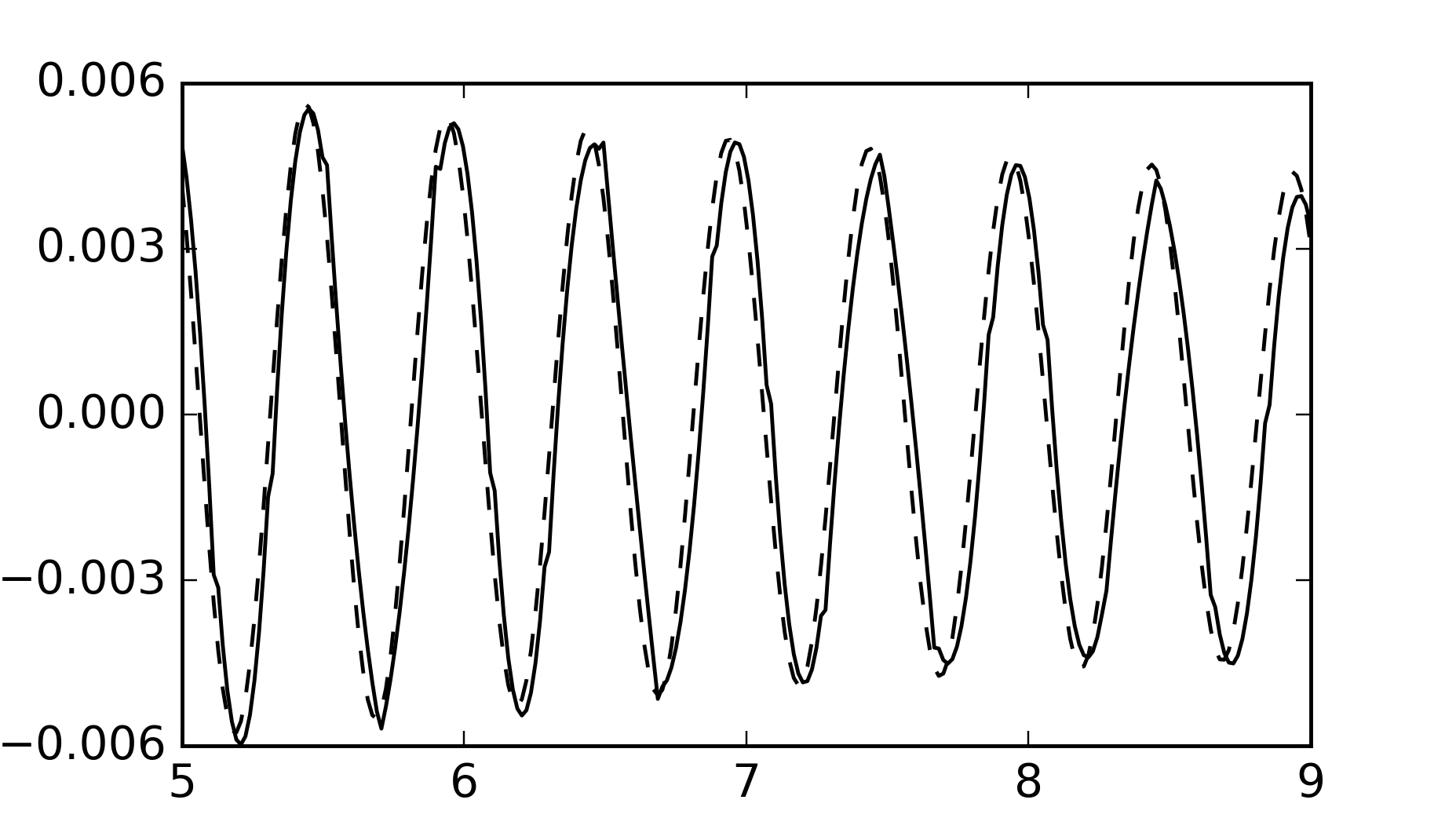}
 \includegraphics[width=.8\textwidth]{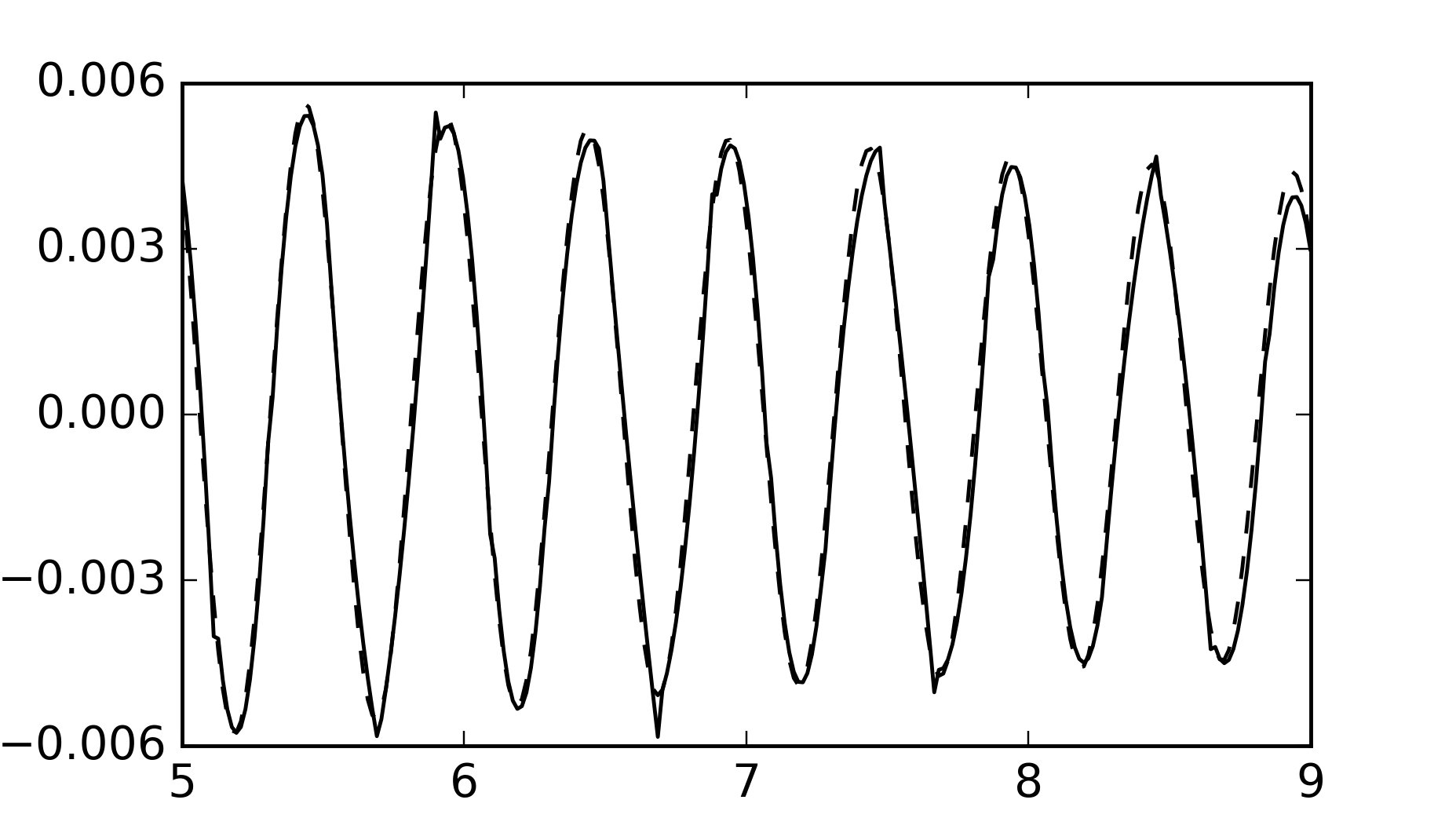}
\end{figure}

\section{Concluding remarks}
\label{sec:conclude}
In this paper, we have proposed an energy conserving DG method for linear 
symmetric hyperbolic systems.
The method is proven to be optimal convergent in one-space dimension, and 
in multi-space dimension on rectangular meshes. 

Extensive numerical results are presented to assess the proposed method.
In particular, we observe the optimal $L^2$-convergence of the method in one-space dimension, 
and in two-space dimension using rectangular meshes.
We also observe the optimal convergence of the method (with the doubling unknowns approach)
on triangular meshes for all the tests considered in this paper.
% for the scalar advection equation, 
% and the acoustic equations with a subsonic background mean flow velocity, these optimal convergence are,
% however, not covered by our current analysis.
Numerical comparison of the new method with the DG methods using upwinding numerical fluxes, and
central numerical fluxes  for long time simulations are also presented. 
The new method is found to be better than the upwinding DG method in terms of the dissipation error, and 
to be better than the central DG method in terms of the dispersion error for all the numerical tests conducted 
in this paper.

\bibliographystyle{siam}

% \bibliography{../../../BIB/all}

\end{document}